\documentclass[10pt]{article}
\usepackage{bbm}
\usepackage{amsmath}
\usepackage{amsthm}
\usepackage{amssymb} 
\usepackage{mathrsfs}
\usepackage[dvipsnames]{xcolor}
\usepackage{graphicx}
\usepackage{authblk}
\usepackage[backref=none,hypertexnames=false,colorlinks=true,urlcolor=blue,linkcolor=blue,citecolor=blue]{hyperref}
\usepackage[numbers,comma,square,sort&compress]{natbib}
\usepackage[letterpaper,text={6.5in,9in},centering]{geometry}

\graphicspath{{eps/}{pdf/}}

\makeatletter\@addtoreset{equation}{section}\makeatother
\renewcommand{\theequation}{\arabic{section}.\arabic{equation}}

\newtheorem{thm}{Theorem}[section] 
\newtheorem{lem}[thm]{Lemma}

\newtheorem{hyp}{Hypothesis}

\newtheorem{rmk}[thm]{Remark}

 {\begin{trivlist}\item[]\textbf{Acknowledgments.}}{\end{trivlist}}

\begin{document}

\title{
Dihedral rings of patterns emerging from a Turing bifurcation
}
\author[1]{Dan J. Hill}
\author[2]{Jason J. Bramburger}
\author[3,*]{David J.B. Lloyd}

\affil[1]{\small Fachrichtung Mathematik, Universit\"at des Saarlandes, 66041 Saarbr\"ucken, Germany}

\affil[2]{\small Department of Mathematics and Statistics, Concordia University, Montr\'eal, QC, Canada}

\affil[3]{\small Department of Mathematics, University of Surrey, Guildford, GU2 7XH, UK}

\affil[*]{\small Corresponding author: D.Lloyd@surrey.ac.uk}

\date{}
\maketitle

\begin{abstract}
Collective organisation of patterns into ring-like configurations has been well-studied when patterns are subject to either weak or semi-strong interactions. However, little is known numerically or analytically about their formation when the patterns are strongly interacting. We prove that approximate strongly interacting patterns can emerge in various ring-like dihedral configurations, bifurcating from quiescence near a Turing instability in generic two-component reaction-diffusion systems. The methods used are constructive and provide accurate initial conditions for numerical continuation methods to path-follow these ring-like patterns in parameter space. Our analysis is complemented by numerical investigations that illustrate our findings. 
\end{abstract}

%%%%%%%%%%%%%%%%%%%%%%%%%%%%%%%%%%%%%%%%%

\section{Introduction}

Localised interacting patterns arranged in a ring-like configuration occur naturally in a range of physical situations, from fairy circles \cite{Zelnik2015Gradual,Zhao2021Fairy,Getzin2015Fairy,Getzin2016FairyCircles} to optical solitons \cite{menesguen2006optical,KIVSHAR1998Dark,Vladimirov2002Clusters}. There has also been various numerical and theoretical studies on these collective organisations of ring-like dihedral  patterns~\cite{Verscheuren2021Disk,Gomila2021_Radial_Homoclinic,Nishiura2022_N_spot_rings}. There is a well-established mathematical theory for when the interacting patterns are weakly interacting~\cite{Vladimirov2002Clusters,zelik2009multi}
or semi-strong interacting~\cite{Wong2021_Spot_Schnakenberg,Chang2019_Spot_Brusselator,Moyles2017_Ring_GM,Chen2011_Spot_GS,Byrnes2022_Radial_Veg}. However, near the pattern-forming or Turing instabilities often seen in many biological, chemical, and physical systems, there is, to the best knowledge of the authors, no mathematical theory beyond localised patterns that are axisymmetric shown in Figure~\ref{fig:SpotRing}(b). In this paper, we develop a novel approximate theory of stationary, strongly interacting patterns with a ring-like arrangement i.e., the maximum amplitude of the pattern does not occur at the centre, near a Turing instability in general two-component reaction-diffusion systems; see Figure~\ref{fig:SpotRing}(d) and Figure~\ref{fig:ExRings}. The theory is a significant development, as it provides a starting point for more thorough numerical investigations of these patterns. 

\begin{figure}[t!] %Figure: Differences between "Spots" and "Rings"
    \centering
    \includegraphics[width=\linewidth]{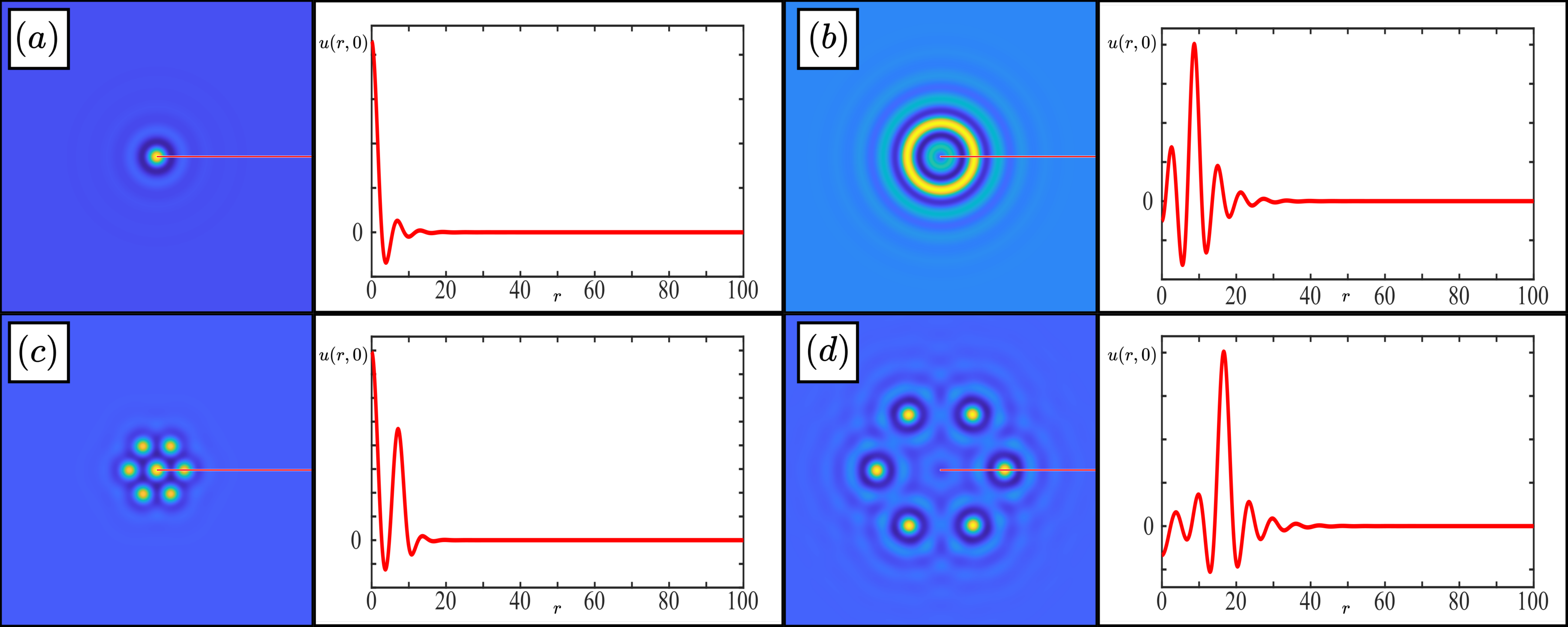}
    \caption{Comparisons between axisymmetric (top) and dihedral (bottom) spots and rings in the Swift--Hohenberg equation \eqref{e:SH} with $\gamma=1.6$. Spot solutions (left) have a global maximum at the origin, while ring solutions (right) have global extrema away from the origin, as can be observed in the radial profiles of each pattern provided in red.
    }
    \label{fig:SpotRing}
\end{figure}

We consider localised stationary solutions of a general two-component reaction diffusion system 
\begin{equation}\label{e:RDsys}
    \mathbf{u}_t = \mathbf{D}\Delta\mathbf{u} - \mathbf{f}(\mathbf{u},\mu),\qquad \mathbf{u}\in\mathbb{R}^2,
\end{equation}
where $\mathbf{D} \in \mathbb{R}^{2\times 2}$ is an invertible matrix, { $\Delta = \partial_r^2 + \frac{1}{r}\partial_r + \frac{1}{r^2}\partial_\theta^2$ is the planar Laplacian operator in polar coordinates}, $\mu$ is a bifurcation parameter, and $\mathbf{f}\in C^k(\mathbb{R}^{2}\times\mathbb{R},\mathbb{R}^{2})$ is a nonlinear function for some $k\geq3$. In this set-up, our theory also identifies stationary solutions of the prototypical pattern-forming Swift--Hohenberg equation
\begin{align}\label{e:SH}
    u_t &= -\left(1 + \Delta\right)^{2}u - \mu u + \gamma u^{2} - u^{3},\qquad u\in\mathbb{R}
\end{align}
where $\gamma$ is a parameter. { This comes from the fact that we only consider steady-state solutions, i.e. $u_t = 0$ in \eqref{e:SH} and $\mathbf{u}_t = \mathbf{0}$ in \eqref{e:RDsys}, and so the change of variable 
\begin{equation}\label{SH2RD}
    \mathbf{u} = \begin{pmatrix}
    u \\ (1 + \Delta)u
    \end{pmatrix}
\end{equation}
casts the Swift--Hohenberg equation \eqref{e:SH} in the form \eqref{e:RDsys}.} Due to its prevalence in the pattern-forming literature, the Swift--Hohenberg equation \eqref{e:SH} will form the basis for our numerical investigations herein. Figure~\ref{fig:SpotRing} provides examples of the diverse localised solutions of the Swift--Hohenberg equation that can be identified numerically with $\gamma= 1.6$ near the pattern-forming instability $\mu=0$. 

Previous work has concentrated on localised axisymmetric spot and ring solutions ~\cite{lloyd2008localized,Mccalla2010snaking,mccalla2013spots,Bramburger2019LocalizedRolls,Barkman2020_Ring_Soliton,castillo2019extended,Hill2020Localised}, as depicted in  Figure~\ref{fig:SpotRing}(a)-(b), and in \cite{hill2022approximate} we provided the first approximate theory for localised dihedral patterns whose solution has a maximum at the centre of the pattern; see Figure~\ref{fig:SpotRing}(c). The current work makes a significant departure from Hill et al.~\cite{hill2022approximate} in that we focus on ring-like dihedral arrangements of strongly interacting patterns, which we call approximate localised dihedral rings. The term ring comes from the Swift--Hohenberg equation, where the solutions go through $0$ at $r = 0$ { at onset}. { Note that, away from the pattern-forming instability $\mu=0$ ring patterns do not necessarily go through $0$ at $r=0$ (for example, see \cite[Figure 7]{lloyd2009localized}), and so we instead identify localised rings by having a maximum/minimum away from $r=0$.} This is different from localised spots which have a local maximum/minimum at $r = 0$. We show in Figure~\ref{fig:SpotRing}(d) and Figure~\ref{fig:ExRings} some of the different types of patterns we are interested in and the  huge variety of different possible configurations.

Our approach starts in the same way as~\cite{hill2022approximate}, where we pose~\eqref{e:RDsys} in polar coordinates and carry out a finite Fourier decomposition in the angular variable. This then leads to a large, but finite, coupled radial ordinary differential equation system, where we employ invariant manifold theory near the Turing instability to prove the existence of these ring-like dihedral arrangements of patterns. 
\begin{figure}[t!] %Figure: Examples of Ring Patterns
    \centering
    \includegraphics[width=\linewidth]{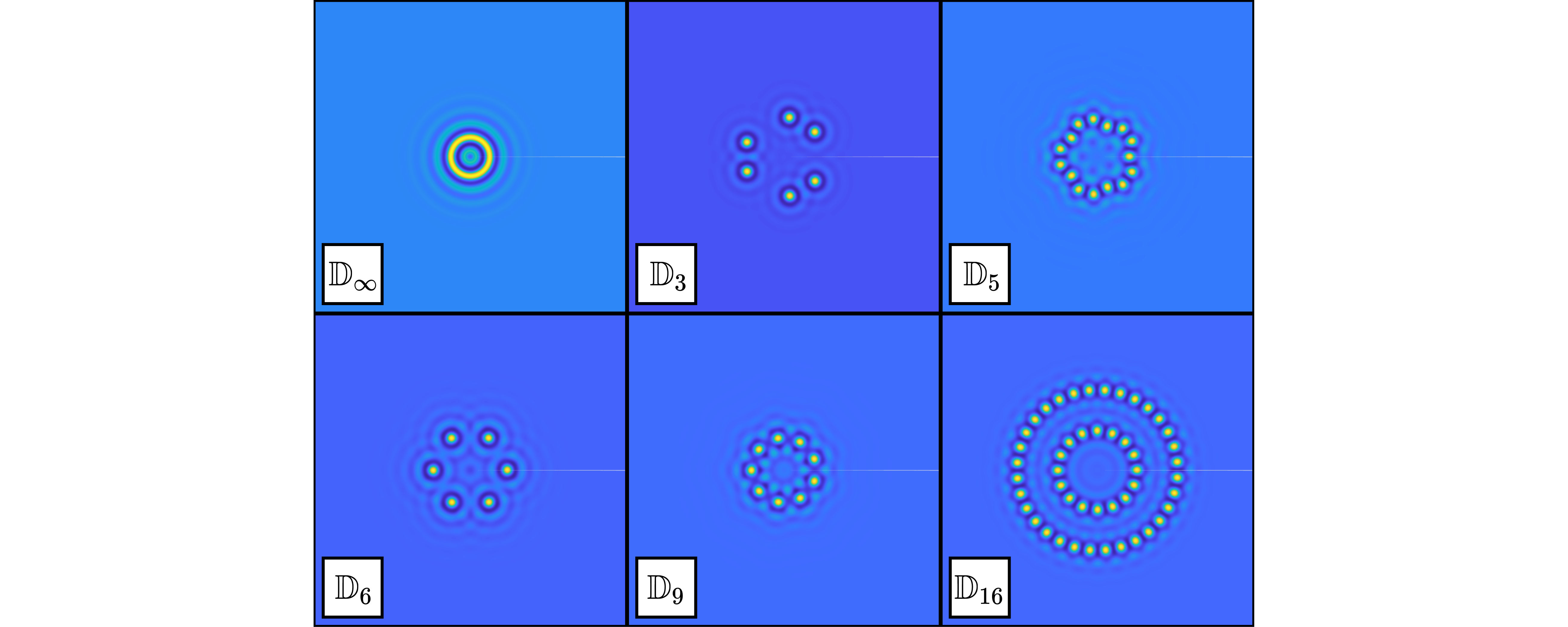}
    \caption{A plethora of numerically identified dihedral ring solutions to the Swift--Hohenberg equation \eqref{e:SH} with $\gamma = 1.6$. Dihedral group symmetries of each pattern are indicated by $\mathbb{D}_m$ as an inset on each panel. Here we adopt the notation $\mathbb{D}_\infty$ to denote an axisymmetric ring pattern since the radial symmetry endows it with invariance with respect to every dihedral symmetry.}
    \label{fig:ExRings}
\end{figure}
The emergence of these localised patterns from the Turing instability is mainly driven by the nonlinear localisation in the leading order far-field equations. This is different to the approximate patches in~\cite{hill2022approximate} where the leading order far-field equations are linear. This mechanism leads to a very different scaling law for the amplitudes of the solutions. For patches, the centre of the patch scales with the square root of the bifurcation parameter, whereas the ring-like patterns have three-quarters scaling. Crucially, the Turing instability needs to be subcritical in order to observe the emergence of these dihedral ring patterns, which is not required for the dihedral patches studied in~\cite{hill2022approximate}.

The paper is outlined as follows. In Section~\ref{sec:Results} we provide our main results. We provide the setting for the problem, full system hypotheses, and a complete statement of the results. Section~\ref{sec:Results} further includes numerical investigations of a resulting matching equations which determines the existence of our approximate dihedral ring solutions. Our main result, Theorem~\ref{thm:Ring} below, is then proved in Section~\ref{sec:Matching} using radial centre manifold theory that matches solutions in a core manifold to those in a far-field manifold to produce the desired radial solutions. We conclude in Section~\ref{sec:Discussion} with a discussion of our results and a presentation of avenues for possible follow-up investigations.

%%%%%%%%%%%%%%%%%%%%%%%%%%%%%%%%%%%%%%%%%
\section{Main Results}\label{sec:Results}

Through this manuscript we focus on the existence of localised steady-state solutions to \eqref{e:RDsys}, with the system hypotheses being similar as those in \cite{hill2022approximate} used to demonstrate the existence of approximate localised dihedral spots in \eqref{e:RDsys}. { Exploiting the invertibility of the diffusion matrix $\mathbf{D}$}, the relevant steady-state equation becomes
\begin{equation}\label{e:RDsysSteady}
    \mathbf{0} = \Delta\mathbf{u} - \mathbf{D}^{-1}\mathbf{f}(\mathbf{u},\mu),   
\end{equation}
where throughout we assume that $\mathbf{f}\in C^k(\mathbb{R}^{2}\times\mathbb{R},\mathbb{R}^{2}),k\geq3$. Now, we are interested in the emergence of localised solutions of \eqref{e:RDsysSteady} from a homogeneous state. This will be achieved through the following assumption on the existence of a Turing instability in the system \eqref{e:RDsys}. { In what follows we use $D_u\mathbf{f}$ to denote the Jacobian of the function $f$ with respect to $u$.} 

{ \begin{hyp}[\em Turing Instability]\label{R-D:hyp;1} %Hypothesis: Turing in R-D
    There exists $k_{c}\in\mathbb{R}^{+}$ so that $\mathbf{f}(\mathbf{u},\mu)$ in \eqref{e:RDsys} satisfies { $\mathbf{f}(\mathbf{0},0) = \mathbf{0}$ and}
	\begin{equation}
    \det\,\big(\mathbf{D}^{-1}D_{\mathbf{u}}\big[\mathbf{f}(\mathbf{0},0)\big] + k_{c}^{2}\mathbbm{1}_2\big) = 0
      % \det\,\big(\mathbf{D}^{-1}D_{\mathbf{u}}\big[\mathbf{f}(\mathbf{0},0)\big] - \lambda\mathbbm{1}_2\big) = 0 \qquad \iff \qquad \lambda = -k_{c}^{2}.
	\end{equation}
	Furthermore, the eigenvalue, $-k_{c}^{2}$, of $\mathbf{D}^{-1}D_{\mathbf{u}}\big[\mathbf{f}(\mathbf{0},0)\big]$ is algebraically double and geometrically simple with generalised eigenvectors $\hat U_0, \hat U_1 \in \mathbb{R}^2$, defined such that
	\begin{equation}
	    \bigg(\mathbf{D}^{-1}D_{\mathbf{u}}\big[\mathbf{f}(\mathbf{0},0)\big] + k_{c}^{2} \mathbbm{1}_{2}\bigg)\hat U_{0} = \mathbf{0}, \qquad \bigg(\mathbf{D}^{-1}D_{\mathbf{u}}\big[\mathbf{f}(\mathbf{0},0)\big] + k_{c}^{2} \mathbbm{1}_{2}\bigg)\hat U_{1} = k_{c}^{2} \hat U_0, \qquad \langle \hat U_i^*, \hat U_j\rangle_2 = \delta_{i,j},
	\end{equation}
	for each $i,j\in\{0,1\}$, where $\hat U_0^*$, $\hat U_1^*$ are the respective adjoint vectors for $\hat U_0$, $\hat U_1$. Without loss of generality we assume $k_c = 1$.
\end{hyp}
}

Hypothesis~\ref{R-D:hyp;1} provides the setting for our problem, in that we are attempting to identify small-amplitude steady-state solutions of \eqref{e:RDsys} that emerge from the trivial state at $\mu = 0$. In a neighbourhood of $(\mathbf{u},\mu) = (\mathbf{0},0)$ we may expand \eqref{e:RDsysSteady} as 
\begin{equation}\label{eqn:R-D}
    \mathbf{0} = \Delta\mathbf{u} - \mathbf{M}_{1}\mathbf{u} - \mu \mathbf{M}_{2}\mathbf{u} - \mathbf{Q}(\mathbf{u},\mathbf{u}) - \mathbf{C}(\mathbf{u},\mathbf{u},\mathbf{u}) + h.o.t.,
\end{equation}
where `h.o.t.' denotes the higher-order terms in the Taylor expansion. { More precisely, $\mathbf{M}_1 \equiv \mathbf{D}^{-1}D_{\mathbf{u}}\big[\mathbf{f}(\mathbf{0},0)\big]$ and $\mathbf{M}_2 \in\mathbb{R}^{2\times 2}$ are linear terms, while $\mathbf{Q}:\mathbb{R}^{2}\times\mathbb{R}^{2}\to\mathbb{R}^{2}$ and $\mathbf{C}:\mathbb{R}^{2}\times\mathbb{R}^{2}\times\mathbb{R}^{2}\to\mathbb{R}^{2}$ are the quadratic and cubic terms in $\mathbf{u}$}, respectively.
{
\begin{rmk}
    One could equivalently consider a reaction--diffusion system of the form
    \begin{equation*}
        \mathbf{u}_t = \mathbf{D}(\mu)\Delta\mathbf{u} - \mathbf{f}(\mathbf{u}),
    \end{equation*}
    where we assume $\mathbf{D}(\mu)$ is invertible for all $0\leq\mu\leq\mu_*$, for an appropriate $\mu_*>0$. Then, Hypothesis~\ref{R-D:hyp;1} is replaced by an equivalent condition $\det\,(\mathbf{D}(0)^{-1} D_{\mathbf{u}}\mathbf{f}(\mathbf{0}) + k_c^2\mathbbm{1}_2)=0$, while the expansion \eqref{eqn:R-D} is unchanged. As a result, our analysis is also valid for traditional diffusion-driven Turing instabilities.
\end{rmk}
}

Only the terms shown in \eqref{eqn:R-D} will be relevant to our analysis, and we require the following subcriticality assumption for the Turing bifurcation assumed by the previous hypothesis. 

\begin{hyp}[\em Subcriticality condition]\label{R-D:hyp;2} %Hypothesis: Non-degeneracy in R-D
	We assume that $\mathbf{M}_2$, $\mathbf{Q}$ and $\mathbf{C}$ satisfy the following conditions:
	\begin{equation}\label{c0c3}\begin{split}
    		c_{0} :&= \frac{1}{4}\big\langle \hat U_1^*, -\mathbf{M}_2 \hat U_0\big\rangle_2>0, \\ 
      c_3 :&=- \bigg[\bigg(\frac{5}{6}\big[\big\langle \hat{U}_{0}^{*},Q_{0,0}\big\rangle_{2} + \big\langle \hat{U}_{1}^{*},Q_{0,1}\big\rangle_{2}\big] + \frac{19}{18}\big\langle \hat{U}_{1}^{*},Q_{0,0}\big\rangle_{2} \bigg)\big\langle \hat{U}_{1}^{*},Q_{0,0}\big\rangle_{2} +\frac{3}{4}\big\langle \hat{U}_{1}^{*}, C_{0,0,0}\big\rangle_{2}\bigg] < 0,
	\end{split}\end{equation}
	where $Q_{i,j}:=\mathbf{Q}(\hat{U}_{i},\hat{U}_{j})$, $C_{i,j,k}:=\mathbf{C}(\hat{U}_{i},\hat{U}_{j},\hat{U}_{k})$, and $\hat U_0, \hat U_1 \in \mathbb{R}^2$ are generalised eigenvectors of $\mathbf{M}_1$ defined in Hypothesis~\ref{R-D:hyp;1} with respective adjoint vectors $\hat U_0^*$, $\hat U_1^*$.
\end{hyp}

With Hypothesis~\ref{R-D:hyp;1} and \ref{R-D:hyp;2} we will demonstrate the existence of approximate localised dihedral ring patterns to the steady-state equation \eqref{e:RDsysSteady}. 
% In particular, these solutions are bounded as $\mu \to 0^+$, satisfy $\lim_{r \to \infty} u(r,\theta) = 0$, are bounded as $r \to 0^+$, and satisfy $u(r,\theta + 2\pi/m) = u(r,\theta)$ for some $m \in \mathbb{N}$. As in \cite{hill2022approximate}, we introduce a $\mathbb{D}_m$, $N$-truncated Fourier series approximation for a solution of \eqref{e:RDsysSteady} as
% \begin{equation}\label{FourierExp:R-D}
% 	\mathbf{u}(r,\theta) = \sum_{n = -N}^{N} \mathbf{u}_{|n|}(r) \cos\left(m n\theta\right).
% \end{equation}
% Note that the above series does indeed satisfy the dihedral symmetry property $u(r,\theta + 2\pi/m) = u(r,-\theta) = u(r,\theta)$, but only those with $N = \infty$ can be true solutions of \eqref{e:RDsysSteady}. 
In particular, these solutions are bounded as $\mu \to 0^+$, satisfy $\lim_{r \to \infty} \mathbf{u}(r,\theta) = 0$, and are bounded as $r \to 0^+$. We also assume $\mathbf{u}\in\mathbb{D}_{m}$ for some $m \in \mathbb{N}$, where $\mathbb{D}_{m}$ is the set of $m^\mathrm{th}$-order dihedral functions, defined by
\begin{equation}
    \mathbb{D}_{m}:=\left\{ f(r,\theta):\mathbb{R}^{+}\times[0,2\pi)\to\mathbb{R}^{2}:\quad f(r,\theta) = f(r,\theta + 2\pi/m) = f(r,-\theta)\right\}.
\end{equation}
Any dihedral function $\mathbf{u}\in\mathbb{D}_{m}$ { can be decomposed into} a Fourier cosine series
\begin{equation}\label{FourierExp:R-D;Inf}
	\mathbf{u}(r,\theta) = \sum_{n \in\mathbb{Z}} \mathbf{u}_{|n|}(r) \cos\left(m n\theta\right).
\end{equation}
Substituting \eqref{FourierExp:R-D;Inf} into the expansion \eqref{eqn:R-D} and projecting onto each Fourier mode $\cos(mn\theta)$ results in the infinite-dimensional system
\begin{equation} \label{eqn:R-D;Proj}
    0 = \Delta_{n}\,\mathbf{u}_{n} - \mathbf{M}_{1}\mathbf{u}_{n} - \mu \mathbf{M}_{2}\mathbf{u}_{n} - \sum_{\substack{i+j=n\\i,j\in\mathbb{Z}}} \mathbf{Q}(\mathbf{u}_{|i|},\mathbf{u}_{|j|}) - \sum_{\substack{i+j+k=n\\i,j,k\in\mathbb{Z}}} \mathbf{C}(\mathbf{u}_{|i|},\mathbf{u}_{|j|},\mathbf{u}_{|k|}) + h.o.t.,
\end{equation}
for all $n\in\mathbb{Z}$, where
\begin{equation}\label{Delta_n}
    \Delta_{n} := \left(\partial_{rr} + \tfrac{1}{r}\partial_{r} - \tfrac{(mn)^{2}}{r^{2}}\right).
\end{equation} 
We are unable to use standard techniques from spatial dynamics to study \eqref{eqn:R-D;Proj}, since the system is infinite dimensional. Likewise, we are unable to reduce \eqref{eqn:R-D;Proj} to some finite-dimensional centre-manifold, since the centre eigenspace is also infinite-dimensional. To circumvent these issues, we instead introduce a truncated Galerkin approximation for \eqref{eqn:R-D;Proj}, for some truncation order $N\geq0$. That is, we assume $\mathbf{u}_n(r)\equiv0$ for all $n>N$ such that
\begin{equation}\label{FourierExp:R-D}
	\mathbf{u}(r,\theta) = \sum_{n = -N}^{N} \mathbf{u}_{|n|}(r) \cos\left(m n\theta\right).
\end{equation}
Then, projecting onto the first $(N+1)$ Fourier modes, we reduce \eqref{eqn:R-D;Proj} to the finite-dimensional Galerkin system
\begin{equation} \label{eqn:R-D;Galerk}
    \mathbf{0} = \left\{\Delta_{n}\,\mathbf{u}_{n} - \mathbf{M}_{1}\mathbf{u}_{n} - \mu \mathbf{M}_{2}\mathbf{u}_{n} - \sum_{i+j=n} \mathbf{Q}(\mathbf{u}_{|i|},\mathbf{u}_{|j|}) - \sum_{i+j+k=n} \mathbf{C}(\mathbf{u}_{|i|},\mathbf{u}_{|j|},\mathbf{u}_{|k|}) + h.o.t.\right\}_{n=0}^{N}
\end{equation}
where the above summations are over all $i,j,k\in[-N,N]$ that satisfy the summation condition. We prove the existence of localised solutions to the Galerkin system \eqref{eqn:R-D;Galerk}, for which the truncated series \eqref{FourierExp:R-D} approximates a localised dihedral solution to \eqref{e:RDsysSteady}.
% Putting \eqref{FourierExp:R-D} into the expansion \eqref{eqn:R-D} and projecting onto each Fourier mode, $\cos(m n \theta)$, results in the finite-dimensional Galerkin system
% \begin{equation} \label{eqn:R-D;Galerk}
%     0 = \Delta_{n}\,\mathbf{u}_{n} - \mathbf{M}_{1}\mathbf{u}_{n} - \mu \mathbf{M}_{2}\mathbf{u}_{n} - \sum_{i+j=n} \mathbf{Q}(\mathbf{u}_{|i|},\mathbf{u}_{|j|}) - \sum_{i+j+k=n} \mathbf{C}(\mathbf{u}_{|i|},\mathbf{u}_{|j|},\mathbf{u}_{|k|}) + h.o.t.,
% \end{equation}
% for all $n\in[0,N]$, where the above summations are over all $i,j,k\in[-N,N]$ that satisfy the summation condition, and 
% \begin{equation}\label{Delta_n}
%     \Delta_{n} := \left(\partial_{rr} + \tfrac{1}{r}\partial_{r} - \tfrac{(mn)^{2}}{r^{2}}\right).
% \end{equation} 
We remind the reader that radially-symmetric solutions of \eqref{e:RDsysSteady}, i.e. $u(r,\theta) = u(r)$, were proven to exist in \cite{lloyd2009localized}. Taking $N = 0$ in our expansion \eqref{eqn:R-D;Galerk} reduces to these radially-symmetric patterns, and so we exclusively focus on the { case of finite $N \geq 1$}. 

{In section~\ref{subsec:Core}, we show that there are no linear solutions to \eqref{eqn:R-D;Galerk} that exhibit exponential decay as $r\to\infty$ (there are only Bessel function solutions which decay algebraically as $r\rightarrow\infty$), and so nonlinear analysis is crucial to properly describe localised radial solutions to \eqref{eqn:R-D;Galerk}. 
}
% We note that there are no linear solutions to \eqref{eqn:R-D;Galerk} that remain bounded for all values of $r\in[0,\infty)$, let alone exhibit exponential decay as $r\to\infty$, and so nonlinear analysis is crucial to properly describe localised radial solutions to \eqref{eqn:R-D;Galerk}. 
A key distinction between our results herein and those of our previous investigation \cite{hill2022approximate} can be seen in the regions of $r$ for which linear analysis is insufficient. For spot-like solutions, as studied in \cite{hill2022approximate,mccalla2013spots,lloyd2009localized}, the localisation is determined by a linear flow for very large values of $r$ (henceforth, the `far-field') intersecting with a quadratic expansion of locally-bounded solutions close to the origin (henceforth, the `core'). This results in solutions with a geometric arrangement of peaks in the core that decay at a constant exponential rate in the far-field. In contrast, localised ring-like solutions are the result of linear bounded solutions in the core intersecting with nontrivial homoclinic orbits to a cubic-order system in the far-field. As a result, ring-like solutions continue to grow at an algebraic rate in the far-field before then decaying exponentially as $r\to\infty$. The result is that much of the geometric structure of these localised solutions is captured away from the core. 

While the quadratic terms in the core expansion are essential in finding non-trivial localised spot-like solutions, they are obtained from a Taylor expansion and do not significantly affect the subsequent analysis. However, in the case of localised ring-like solutions, proving the existence of homoclinic orbits in the far-field requires a more delicate analysis of the far-field amplitude equations than presented in \cite{hill2022approximate} and results in a qualitatively different class of localised planar solutions. As we will show in Section~\ref{subsec:Rescaling}, we can introduce an invariant subspace of the far-field amplitude equations which reduces the problem to proving the existence of a nontrivial homoclinic solution to the radial Ginzburg--Landau equation
\begin{equation}\label{CompAssistGL}
    \left(\tfrac{\mathrm{d}}{\mathrm{d} s} + \tfrac{1}{2s}\right)^{2}q(s) = c_0 q(s) + c_3 q(s)^3, \qquad s > 0,
\end{equation}
for $c_0$ and $c_3$ given in Hypothesis~\ref{R-D:hyp;2}{ , where $s=\mu^{\frac{1}{2}}r$ is the rescaled radial coordinate in the far-field}. The necessity for a nontrivial homoclinic solution of \eqref{CompAssistGL} to exist arises under similar circumstances for the existence of radially-symmetric solutions to the Swift--Hohenberg equation in \cite{mccalla2013spots}, which provides numerical verification of a solution. Fortunately, the work of van den Berg, Groothedde, and Williams \cite{vandenberg2015Rigorous} provides the necessary existence result using computer-assisted methods. Their result will be stated precisely in Lemma~\ref{Lem:Ginzb} below, but for now we note that it guarantees the existence of constants $q_0>0$ and $q_+\neq0$ so that the solution satisfies
\begin{equation}\label{qs_eqn}
    q(s) = \begin{cases}
         q_0 s^{\frac{1}{2}} + \mathcal{O}(s^{\frac{3}{2}}), & s\to 0,\\
         (q_+ + \mathcal{O}(\mathrm{e}^{-\sqrt{c_0}s}))s^{-\frac{1}{2}}\mathrm{e}^{-\sqrt{c_0}s}, & s\to \infty.
    \end{cases}
\end{equation}
{ Equation~\eqref{qs_eqn} enables} the following theorem, which is the main result of this work. The presence of $q$ denotes the solution to \eqref{CompAssistGL} provided in \cite{vandenberg2015Rigorous}. 

\begin{thm}\label{thm:Ring} %Theorem: Ring existence
	Assume Hypotheses~\ref{R-D:hyp;1} and \ref{R-D:hyp;2}. Fix $m,N \in \mathbb{N}$ and assume the constants $\{a_{n}\}_{n=0}^{N}$ are solutions of the nonlinear matching condition 
	\begin{equation}\label{MatchEq}
        a_{n} = \sum_{i+j+k=n} (-1)^{\frac{m(|i| + |j| - |k| - n)}{2}}a_{|i|}a_{|j|}a_{|k|},
    \end{equation}
    for each $n = 0,1,\dots, N$ and $-N \leq i,j,k \leq N$. Then, there exist constants $q_0,\mu_0,r_0,r_1 > 0$ such that the Galerkin system \eqref{eqn:R-D;Galerk} has a radially localised solution of the form 
\begin{equation}\label{RadialProfile}
    \mathbf{u}_{n}(r) = 2 a_{n}
    \begin{cases}
         \mu^{\frac{3}{4}}q_0 \sqrt{\frac{\pi}{2}}\left[r J_{mn+1}(r)\hat{U}_{0} + 2J_{mn}(r)\hat{U}_{1}\right] + \mathcal{O}(\mu), & 0\leq r\leq r_{0},  \\
         \mu^{\frac{3}{4}} q_0\left[r^{\frac{1}{2}} \sin(\psi_{n}(r))\hat{U}_{0} + 2r^{-\frac{1}{2}} \cos(\psi_{n}(r))\hat{U}_{1}\right] + \mathcal{O}(\mu), & r_{0} \leq r \leq r_1\mu^{-\frac{1}{2}},\\
         \mu^{\frac{1}{2}}\left[q(\mu^{\frac{1}{2}}r)\sin(\psi_{n}(r))\hat{U}_{0} + 2\left(\tfrac{\mathrm{d}}{\mathrm{d}r} + \tfrac{1}{2r}\right)q(\mu^{\frac{1}{2}}r)\cos(\psi_{n}(r))\hat{U}_{1}\right] + \mathcal{O}(\mu), & r_1\mu^{-\frac{1}{2}} \leq r,
    \end{cases}
\end{equation} 
for each $\mu\in(0,\mu_{0})$, $n\in[0,N]$, where $\psi_n(r):= r - \frac{mn\pi}{2} - \frac{\pi}{4}$ and $J_\nu$ is the $\nu^\mathrm{th}$ order Bessel function of the first kind. 
\end{thm}

\begin{rmk}
There are a few key differences between these solutions and the Spot A-type solutions that were identified in \cite{hill2022approximate}. First, the solutions \eqref{RadialProfile} exhibit a scaling of $\mu^{3/4}$ as $\mu \to 0^+$, while the spot A solutions scale as $\mu^{1/2}$. Second, we have the presence of the Bessel functions $J_{mn+1}$, while the Spot A solutions only have $J_{mn}$. Finally, notice that if $\{a_n\}_{n = 0}^N$ is a solution of the matching equation \eqref{MatchEq}, then so is $\{-a_n\}_{n = 0}^N$. This means that solutions \eqref{RadialProfile} come in pairs, related by negating all $a_n$, and thus these solutions emanate in a pitchfork bifurcation from $\mu = 0$. The Spot A solutions emanate in a single branch from $\mu = 0$, without a symmetric counterpart, since their associated quadratic matching equation lacks the odd symmetry of \eqref{MatchEq}.
\end{rmk}

{
Crucial to the application of Theorem~\ref{thm:Ring}, is the satisfying of the matching condition \eqref{MatchEq}. While the specifics of this will be presented in Lemma~\ref{Lem:Resc;Evo} below, the presence of this condition comes from how we construct the solutions. We construct a manifold of ``core" solutions that capture all bounded solutions as $r\rightarrow0$ and match this with a ``far-field" manifold that captures all bounded solutions as $r\rightarrow\infty$. This matching process leads to the algebraic matching condition \eqref{MatchEq}.
}

Our results are particularly applicable to the Swift--Hohenberg equation \eqref{e:SH}, { via the transformation \eqref{SH2RD}.}  Moreover, the system hypotheses are easily verified to find that $k_c = 1$, $\hat{U}_0 = (1,0)^T$, and $\hat{U}_1 = (0,1)^T$. What is important here is that based on these $\hat{U}_0$ and $\hat{U}_1$ and the transformation \eqref{SH2RD}, Theorem~\ref{thm:Ring} gives an approximate solution to the Swift--Hohenberg equation of the form
\begin{equation}\label{SHsol}
    u(r,\theta) = u_{0}(r) + 2\sum_{n = 1}^N u_n(r)\cos(m n \theta)
\end{equation}
with 
\begin{equation}
    u_{n}(r) = 2 a_{n}
    \begin{cases}
         \mu^{\frac{3}{4}}\sqrt{\frac{\pi}{2}} q_0r J_{mn+1}(r) + \mathcal{O}(\mu), & 0\leq r\leq r_{0},  \\
         \mu^{\frac{3}{4}} q_0r^{\frac{1}{2}} \sin(\psi_{n}(r)) + \mathcal{O}(\mu), & r_{0} \leq r \leq r_1\mu^{-\frac{1}{2}},\\
         \mu^{\frac{1}{2}}q(\mu^{\frac{1}{2}}r)\sin(\psi_{n}(r)) + \mathcal{O}(\mu), & r_1\mu^{-\frac{1}{2}} \leq r.
    \end{cases}
\end{equation} 
{Notice that $u_n(0) = 0$ for all $n$, and so the approximate solution vanishes at the origin. This is where the terminology `rings' comes from since there is localisation of the solution in the far-field, given by the exponential decay of $q$ in $r$, and localisation at the core, given by the function vanishing at $r = 0$. Therefore, all structure of the approximate solution to the Swift--Hohenberg equation is observed in an annulus of the planar domain. This is generically not the case for general reaction-diffusion systems of the form \eqref{e:RDsys}, but we keep the terminology that these are ring solutions in analogy with the well-studied Swift--Hohenberg equation.   
}

We now explicitly compute solutions to the matching condition \eqref{MatchEq} for low truncation orders $N=0,1,2,3,4$ and demonstrate some of the resulting solutions. We provide all solutions for $N = 0,1,2,3,4$ up to the symmetries, denoted $\mathcal{R}$ and $\mathcal{S}$, outlined in Section~\ref{app:Match;sym} of the appendix. We further avoid {\em harmonic} solutions, which are solutions for smaller values of $N$ that carry over to larger values via the transformation \eqref{HarmonicTransform} in the appendix. Through these symmetries we are able to restrict to numerical solutions where $a_0\geq0$, $a_1\geq0$. For these small values of $N$ the solution can be obtained through algebraic manipulations, however we will only present the numerical values of the solutions herein. All are approximated to three significant digits.

\underline{$N=0$:} Taking $N=0$ reduces the matching equation \eqref{MatchEq} to the scalar equation,
\begin{equation}\label{Match:N0}\begin{split}
a_{0} &= a_{0}^{3},
\end{split}\end{equation}
and so we recover the trivial solution $\mathbf{a} = a_0 =0$ and the axisymmetric ring $\mathbf{a} = a_0 =1$.

\begin{figure}[t!] %Figure: N1 amplitudes
    \centering
    \includegraphics[width=\linewidth]{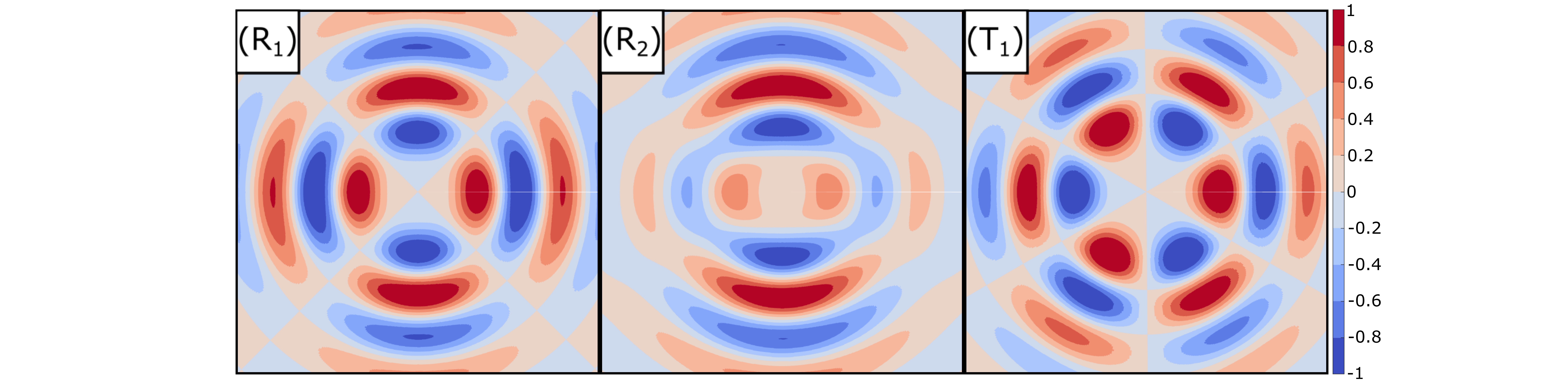}
    \caption{Projected solutions $\langle\hat{U}_{0}^{*},\mathbf{u}\rangle_2$ when $N=1$, given by Theorem~\ref{thm:Ring} for (R)hombic $(\mathbb{D}_{2})$ and (T)riangular $(\mathbb{D}_{3})$ rings are { computed} in a circular disc of radius 20 circumscribing the square region, with $\mu=0.08$.}
    \label{fig:N1_Amps}
\end{figure}

\underline{$N = 1$:} Taking $N=1$ reduces the matching equation \eqref{MatchEq} to the two-dimensional system
\begin{equation}\label{Match:N1}\begin{split}
a_{0} &= a_{0}^{3} + 2 D_m a_{1}^{2}a_{0},\\
a_{1} &= D_m a_{1}a_{0}^{2} + 3 a_{1}^{3},\\
\end{split}\end{equation}
where we have defined $D_{m}:=2+(-1)^m$. In the case when $m$ is odd we obtain a solution $\mathbf{a}=(0,0.577)$ to \eqref{Match:N1}; when $m$ is even we obtain two solutions to \eqref{Match:N1}, namely $\mathbf{a}=(0,0.577), (0.447,0.365)$. Solutions as given by Theorem~\ref{thm:Ring} are provided in Figure~\ref{fig:N1_Amps} for $m = 2$ (rhombic symmetry) and $m = 3$ (triangular symmetry) which exemplifies both even and odd $m$ solution profiles. Here we only present the planar profiles of $\langle \hat{U}_0^*, \mathbf{u}\rangle_2$, which from the form \eqref{SHsol} correspond to ring solutions appearing in the Swift--Hohenberg equation \eqref{e:SH}. We note that other reaction-diffusion systems will have solutions that possess a combination of $r J_{mn+1}(k r)$ and $J_{mn}(k r)$ terms, depending on the values of $\hat{U}_{0}$ and $\hat{U}_{1}$.

\begin{figure}[t!] %Figure: N2 amplitudes
    \centering
    \includegraphics[width=\linewidth]{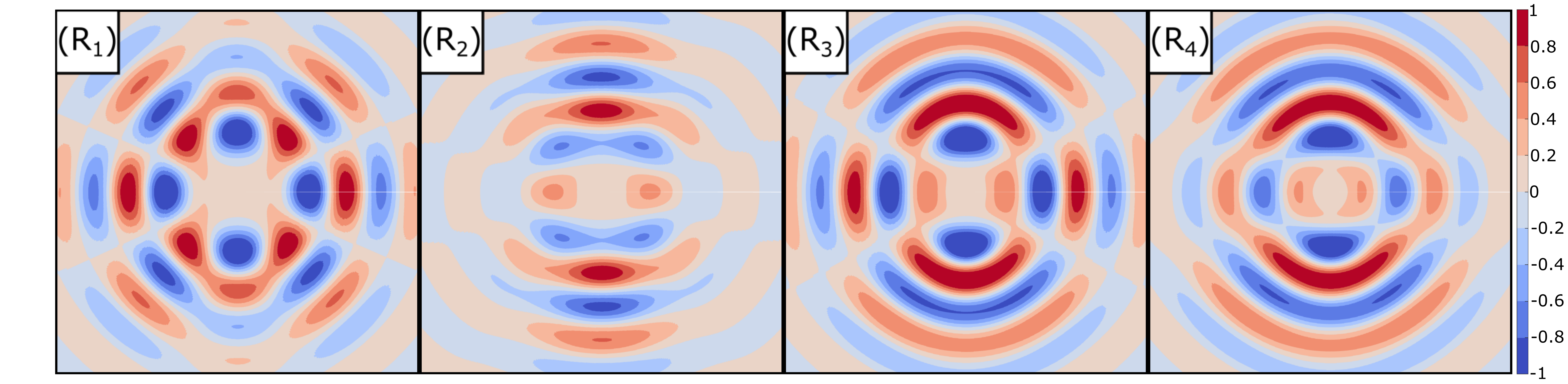}
    \caption{Projected solutions $\langle\hat{U}_{0}^{*},\mathbf{u}\rangle_2$ when $N=2$, given by Theorem~\ref{thm:Ring} for (R)hombic $(\mathbb{D}_{2})$ rings are { computed} in a circular disc of radius 25 circumscribing the square region, with $\mu=0.08$.}
    \label{fig:N2_Amps}
\end{figure}
\underline{$N=2$:} Taking $N=2$ reduces the matching equation \eqref{MatchEq} to the three-dimensional system
\begin{equation}\label{Match:N2}\begin{split}
a_{0} &= a_{0}^{3} + 2D_m a_{0}a_{1}^{2} + 6 a_{0} a_{2}^{2} + 2(-1)^{m}D_m a_{1}^{2}a_{2} ,\\
a_{1} &= D_m a_{0}^{2} a_{1} + 3a_{1}^{3} + 2 D_m a_{1}a_{2}^{2} + 2(-1)^{m}D_m  a_{0}a_{1}a_{2} ,\\
a_{2} &= 3a_{2}^{3} + 2D_m a_{1}^{2} a_{2} + 3a_{0}^{2}a_{2} + (-1)^{m}D_m a_{0}a_{1}^{2}.\\
\end{split}\end{equation}
When $m$ is odd, there are no non-harmonic solutions to \eqref{Match:N2}. When $m$ is even, however, we obtain four solutions to \eqref{Match:N2}
\begin{equation*}
    \begin{aligned}
        \mathbf{a} &= 
        % (0,0,0), (1,0,0), (0,0.577,0), (0,0,0.577),(0.447,0,0.365),
        (0.206,0.215,-0.467), (0.274, 0.255,0.203), (0.367, 0.507, -0.250),(0.649,0.293,-0.189).
    \end{aligned}
\end{equation*}
The associated rhombic patterns are plotted in Figure~\ref{fig:N2_Amps}.

\begin{figure}[t!] %Figure: N2 amplitudes
    \centering
    \includegraphics[width=\linewidth]{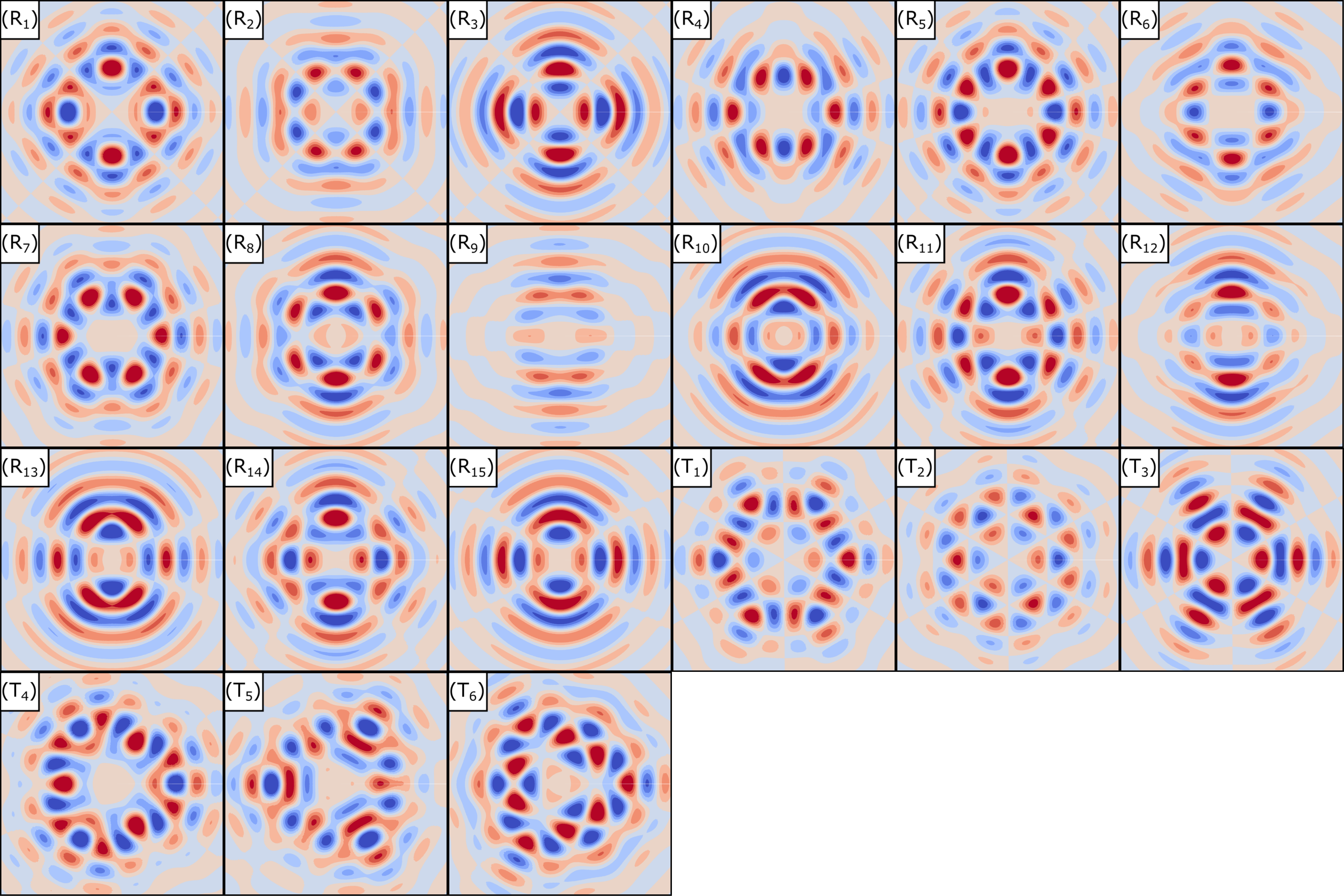}
    \caption{Projected solutions $\langle\hat{U}_{0}^{*},\mathbf{u}\rangle_2$ when $N=3$, given by Theorem~\ref{thm:Ring} for (R)hombic $(\mathbb{D}_{2})$ and (T)riangular $(\mathbb{D}_{3})$ rings are { computed} in a circular disc of radius 30 circumscribing the square region, with $\mu=0.08$.}
    \label{fig:N3_Amps}
\end{figure}

\underline{$N=3$:} Taking $N=3$ reduces the matching equation \eqref{MatchEq} to the four-dimensional system
\begin{equation}\label{Match:N3}\begin{split}
a_{0} &= a_0^3 + 2D_{m}a_0a_1^2 + 6a_0a_2^2 + 2(-1)^{m}D_{m}a_1^2 a_2 + 2D_{m}a_0 a_3^2 + 4D_{m}a_1 a_2 a_3,\\ 
a_{1} &= D_m a_0^2 a_1 + 3a_1^3 + 2D_{m}a_1a_2^2 + 2(-1)^{m}D_{m}a_0 a_1 a_2 + 3(-1)^m a_1^2 a_3 + 2D_{m} a_0a_2a_3\\
&\quad + (-1)^m D_{m}a_2^2 a_3 + 6a_1a_3^2,\\
a_{2} &= 3a_2^3 + 2D_{m}a_1^2 a_2 + 3a_0^2a_2 + (-1)^{m} D_{m}a_0a_1^2  + 2D_{m}a_0a_1a_3 + 2(-1)^m D_{m}a_1a_2a_3 + 2D_{m}a_2a_3^2,\\
a_{3} &= (-1)^m a_1^3 + 2D_{m}a_0a_1a_2 + (-1)^{m}D_{m}a_1a_2^2 + D_{m}a_0^2a_3 + 6 a_1^2 a_3 + 2D_{m}a_2^2a_3 + 3a_3^3.\\
\end{split}\end{equation}
When $m$ is odd we obtain six solutions to \eqref{Match:N3}
\begin{equation*}
    \begin{aligned}
    \mathbf{a} &= 
        (0,0.286,0,0.433), (0, 0.331, 0, -0.262), (0,0.617,0,0.171), \\
        &\quad (0.272, 0.056,0.308,-0.492), (0.384,0.219,0.241,0.366),  (0.526, 0.280, -0.263, 0.378),
    \end{aligned}
\end{equation*}
while for $m$ even we obtain fifteen solutions to \eqref{Match:N3}
\begin{equation*}
    \begin{aligned}
        \mathbf{a} &= 
        (0,0.286,0,-0.433), (0, 0.331, 0, 0.262), (0,0.617,0,-0.171), \\
        &\quad (0.162, 0.030, -0.324, 0.310), (0.135, 0.138, 0.146, -0.503), (0.335, 0.139, -0.190,-0.252),\\
        &\quad (0.284,0.152,-0.138, 0.432), (0.578,0.160,0.217,-0.241), (0.197,0.190,0.169,0.139),\\
        &\quad (0.745,0.229,-0.182,0.124), (0.235,0.270,0.390,-0.339), (0.454,0.321,0.027,-0.164), \\
        &\quad(0.540,0.407,-0.288,0.147), (0.096,0.430,0.246,-0.290), (0.268,0.551,-0.200,-0.044).
    \end{aligned}
\end{equation*}
Rhombic and triangular solutions are plotted in Figure~\ref{fig:N3_Amps}.

\begin{figure}[p!] %Figure: N4 amplitudes
    \centering
    \includegraphics[width=\linewidth]{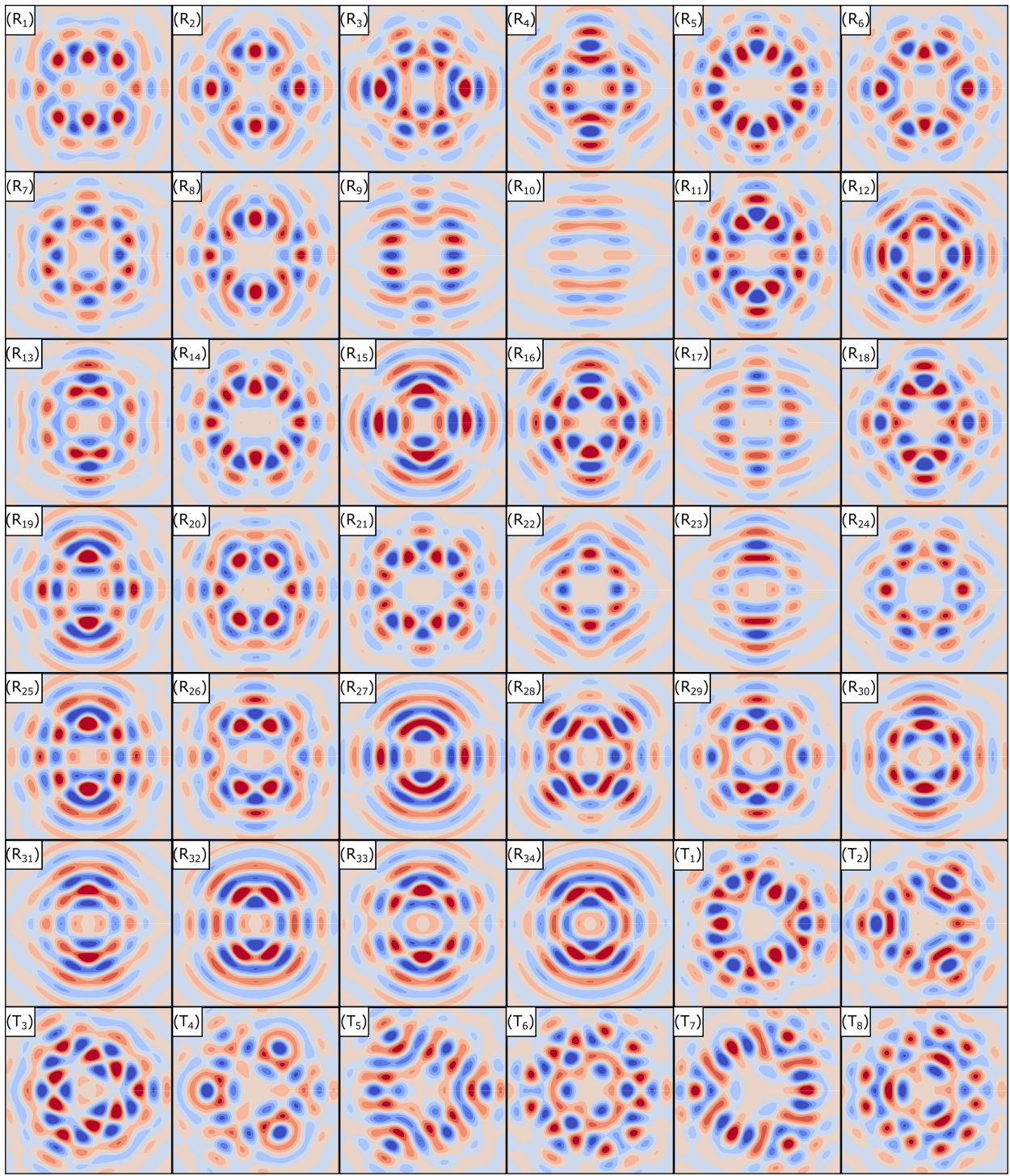}
    \caption{Projected solutions $\langle\hat{U}_{0}^{*},\mathbf{u}\rangle_2$ when $N=4$, given by Theorem~\ref{thm:Ring} for (R)hombic $(\mathbb{D}_{2})$ and (T)riangular $(\mathbb{D}_{3})$ rings are { computed} in a circular disc of radius 30 circumscribing the square region, with $\mu=0.08$.}
    \label{fig:N4_Amps}
\end{figure}

\underline{$N=4$:} Taking $N=4$ reduces the matching condition \eqref{MatchEq} to the five-dimensional system
\begin{equation}\label{Match:N4}\begin{split}
a_{0} &= a_0^3 + 2D_{m}a_0a_1^2 + 6a_0a_2^2 + 2(-1)^{m}D_{m}a_1^2 a_2 + 2D_{m}a_0 a_3^2 + 4D_{m}a_1 a_2 a_3 + 6 a_2^2a_4\\
&\quad + 4(-1)^{m}D_{m}a_1a_3a_4 + 6 a_0a_4^2,\\ 
a_{1} &= D_m a_0^2 a_1 + 3a_1^3 + 2D_{m}a_1a_2^2 + 2(-1)^{m}D_{m}a_0 a_1 a_2 + 3(-1)^m a_1^2 a_3 + 2D_{m} a_0a_2a_3\\
&\quad + (-1)^m D_{m}a_2^2 a_3 + 6a_1a_3^2 + 2(-1)^{m}D_{m}a_1a_2a_4 + 2(-1)^{m}D_{m}a_0a_3a_4 + 2D_{m}a_2a_3a_4 + 2D_{m}a_1a_4^2,\\
a_{2} &= 3a_2^3 + 2D_{m}a_1^2 a_2 + 3a_0^2a_2 + (-1)^{m} D_{m}a_0a_1^2  + 2D_{m}a_0a_1a_3 + 2(-1)^m D_{m}a_1a_2a_3 + 2D_{m}a_2a_3^2\\
&\quad + (-1)^{m}D_{m}a_1^2a_4 + 6 a_0a_2a_4 + 2D_{m}a_1a_3a_4 + (-1)^{m}D_{m}a_3^2a_4 + 6 a_2a_4^2,\\
a_{3} &= (-1)^m a_1^3 + 2D_{m}a_0a_1a_2 + (-1)^{m}D_{m}a_1a_2^2 + D_{m}a_0^2a_3 + 6 a_1^2 a_3 + 2D_{m}a_2^2a_3 + 3a_3^3\\
&\quad + 2(-1)^{m}D_{m}a_0a_1a_4 + 2D_{m}a_1a_2a_4 + 2(-1)^{m}D_{m}a_2a_3a_4 + 2D_{m}a_3a_4^2,\\
a_{4} &= (-1)^{m}D_{m} a_1^2 a_2 + 3 a_0 a_2^2 + 2(-1)^{m} D_{m} a_0a_1a_3 + 2D_{m} a_1a_2a_3 + (-1)^{m}D_{m} a_2a_3^2 + 3a_0^2 a_4 \\
&\quad + 2D_{m}a_1^2 a_4 + 6a_2^2 a_4 + 2D_{m}a_3^2a_4 + 3 a_4^3,\\
\end{split}\end{equation}
When $m$ is odd we obtain eight solutions to \eqref{Match:N4}
\begin{equation}
    \begin{aligned}
    \mathbf{a} &= (0.272, 0.056,0.308,-0.492,0), (0.384,0.219,0.241,0.366,0),\\
    &\quad (0.526, 0.280, -0.263, 0.378,0), (0.266,0.261,0.053,0.365,-0.186),\\
    &\quad (0.015,0.265,-0.205,0.351,-0.348), (0.156,0.364,-0.221,-0.158,-0.430),\\
    &\quad (0.165,0.130,0.026,0.432,0.430), (0.375,0.382,-0.000,-0.139,-0.348),
    \end{aligned}
\end{equation}
while for $m$ even we obtain thirty-four solutions to \eqref{Match:N4}
\begin{equation}
    \begin{aligned}
        \mathbf{a} &= 
        ( 0.00507,  0.156,  0.254,
    -0.152,  0.331), ( 0.0113,  0.177,  -0.209, 
   -0.200,  0.323),\\
   &\quad 
   ( 0.0398,  0.354,  -0.252, 
   0.155,  0.309), ( 0.0638,  0.268,  0.370, 
   -0.00263,  -0.239),\\
   &\quad  
   ( 0.101,  0.102,  0.106, 
   0.111,  -0.521), ( 0.109,  0.264,  -0.103, 
   -0.0431,  0.416),\\
   &\quad 
   ( 0.116,  0.286,  -0.137, 
   -0.0581,  -0.286), ( 0.118,  0.0649,  -0.0855, 
   -0.305,  0.367),\\
   &\quad 
   ( 0.130,  0.0308,  -0.170, 
   -0.286,  -0.223), ( 0.153,  0.150,  0.140, 
   0.125,  0.105), \\
   &\quad 
   ( 0.174,  0.188,  0.233, 
   0.319,  -0.390),( 0.182,  0.290,  -0.432, 
   -0.0747,  -0.0239),\\
   &\quad 
   ( 0.206,  0.406,  0.0189, 
   0.0982,  -0.209), ( 0.210,  0.155,  0.0123, 
   -0.151,  0.465),\\
   &\quad 
   ( 0.215,  0.588,  -0.186, 
   -0.0964,  0.122), ( 0.228,  0.181,  0.484, 
   -0.249,  -0.0914), \\
   &\quad 
   ( 0.264,  0.166,  -0.0465, 
   -0.199,  -0.186),( 0.283,  0.0848,  0.349, 
   0.209,  -0.349),\\
   &\quad 
   (0.285,  0.464,  0.0109, 
   -0.270,  0.196), ( 0.298,  0.148,  -0.0825, 
   0.436,  -0.0631),\\
   &\quad 
       ( 0.324,  0.117,  -0.110, 
       0.101,  -0.414), ( 0.342,  0.111,  -0.236, 
   -0.224,  0.0429),\\
   &\quad 
   ( 0.361,  0.291,  0.123, 
   -0.0475,  -0.136),( 0.363,  0.0968,  -0.119, 
   0.152,  0.278), \\
   &\quad 
   ( 0.390,  0.343,  0.143, 
   -0.361,  0.220), ( 0.416,  0.229,  0.0212, 
   0.273,  -0.306),\\
   &\quad 
   ( 0.425,  0.483,  -0.278, 
   0.0722,  0.0480), ( 0.505,  0.138,  -0.402, 
   0.0699,  -0.218),\\
   &\quad 
   ( 0.544,  0.111,  0.136, 
   0.173,  -0.269), ( 0.569,  0.128,  0.263, 
   -0.209,  -0.0422),\\
   &\quad 
       ( 0.579,  0.309,  -0.0885, 
       -0.0870,  0.129), ( 0.640,  0.333,  -0.270, 
   0.187,  -0.102),\\
   &\quad 
   ( 0.663,  0.197,  0.0636, -0.216,  0.171), (0.800, 0.185,  -0.160, 
       0.127,  -0.0911).
    \end{aligned}
\end{equation}
Rhombic and triangular solutions are plotted in Figure~\ref{fig:N4_Amps}.

We briefly discuss a subset of our low-truncation solutions, which we call $\mathbb{D}_{m}^{-}$ solutions. In the matching equation \eqref{MatchEq}, these solutions are found by setting $a_{2n}=0$ for all $n\in[0,N/2]$. This results in $\mathbb{D}_{m}$ patterns $\mathbf{u}(r,\theta)$ where a half-period rotation is equivalent to sending $\mathbf{u}\mapsto-\mathbf{u}$, such as $(R_1)$ and $(T_1)$ in Figure~\ref{fig:N1_Amps}, as well as $(R_1)-(R_3)$ and $ (T_1)-(T_3)$ in Figure~\ref{fig:N3_Amps}. These patterns have been observed in fluid convection---see \cite{Boronska2010RBconvection,Boronska2010RBconvection2} for Rayleigh-B\'enard convection in a cylinder, and \cite{LoJacono2013localizedBinaryFluid} for convection in porous-media---as well as in the cubic-quintic Swift--Hohenberg equation on a finite disc \cite{Verscheuren2021Disk}. In their prior appearances, respective localised $\mathbb{D}_{2}^{-}$ and $\mathbb{D}_{3}^{-}$ solutions have been studied numerically (see `\textit{pizza}' and `\textit{marigold}' solutions in \cite{Boronska2010RBconvection2}, and $D_{4}^{-}$ and $D_{6}^{-}$ solutions in \cite{Verscheuren2021Disk}), where they were found to bifurcate from the trivial state. However, to the authors' knowledge this is the first analytical result regarding the existence of such solutions.

As one can see, as $N$ increases the number of solutions to the matching equations appears to increase exponentially. We refrain from continuing to provide solutions for $N \geq 5$ and now examine some asymptotic properties of particular solutions to the matching equation for large $N$. When $m$ is even we have that $(-1)^{\frac{m(|i| + |j| - |k| - n)}{2}} = 1$ for all $(i,j,k) \in \{-N,\dots,N\}^3$ and $n = i + j + k$. Thus, the matching equation \eqref{MatchEq} becomes
\begin{equation}\label{MatchEq2}
    a_{n} = \sum_{i+j+k=n} a_{|i|}a_{|j|}a_{|k|},
\end{equation}
losing all $m$-dependence and thus meaning that dihedral patterns with even $m$ can be handled identically via Theorem~\ref{thm:Ring}. In Figure~\ref{fig:LargeN} we present a curious pattern regarding solutions to \eqref{MatchEq2} as $N$ gets large. { More precisely}, we have plotted the points $\{(\frac{n}{N},Na_n)\}_{0 \leq n \leq N}$ for numerically-identified solutions of \eqref{MatchEq2} with large values of $N$. The result appears to be approaching a single continuous curve as $N \to \infty$. To see where this comes from, we note that when $N \gg 1$ the matching equation \eqref{MatchEq2} can be shown to closely resemble a Riemann sum approximation of the continuum matching equation
\begin{equation}\label{MatchCont}
    \begin{split}
\alpha(t) &= 
  2\int_{0}^{1-t}\int_{0}^{1-s} \alpha(t+s)\alpha(x)\alpha(s+x) \,\mathrm{d}x\,\mathrm{d}s +  2\int_{0}^{1-t}\int_{0}^{1-(s+t)} \alpha(s)\alpha(x)\alpha((s+t)+x)\,\mathrm{d}x\,\mathrm{d}s 
\\
&\qquad 
+ \int_{0}^{1-t}\int_{0}^{s} \alpha(t+s)\alpha(x)\alpha(s-x)\,\mathrm{d}x\,\mathrm{d}s + \int_{0}^{1-t}\int_{0}^{(s+t)} \alpha(s) \alpha(x)\alpha((s+t)-x)\,\mathrm{d}x\,\mathrm{d}s
\\
&\qquad
+ 2\int_{0}^{t}\int_{0}^{1-s} \alpha(t-s)\alpha(x)\alpha(s+x)\,\mathrm{d}x\,\mathrm{d}s + \int_{0}^{t}\int_{0}^{s} \alpha(t-s)\alpha(x)\alpha(s-x)\,\mathrm{d}x\,\mathrm{d}s \\
&\qquad + \int_{0}^{t}\int_{0}^{1-s} \alpha(s+1-t)\alpha(x+s)\alpha(1-x)\,\mathrm{d}x\,\mathrm{d}s,
\\
\end{split}
\end{equation}
now posed in terms of an integrable function $\alpha:[0,1] \to \mathbb{R}$. The existence of a positive, continuous solution $\alpha_*$ to \eqref{MatchCont} can be confirmed by following as in the computer-assisted proof of \cite{vandenberg2015Rigorous} to a related integral fixed point problem. This technique was applied to the spot-like solutions in \cite{hill2022approximate} to solve a similar continuum limit of a quadratic matching equation. The important note here is that the methods of proof that move from a continuous solution $\alpha_*$ to \eqref{MatchCont} to a solution to the matching equations \eqref{MatchEq2} for $N \gg 1$ are identical to those of \cite[Theorem~5.8]{hill2022approximate} in our previous investigation.

\begin{figure}[t!] %Figure: Large N solutions
    \centering
    \includegraphics[width=\linewidth]{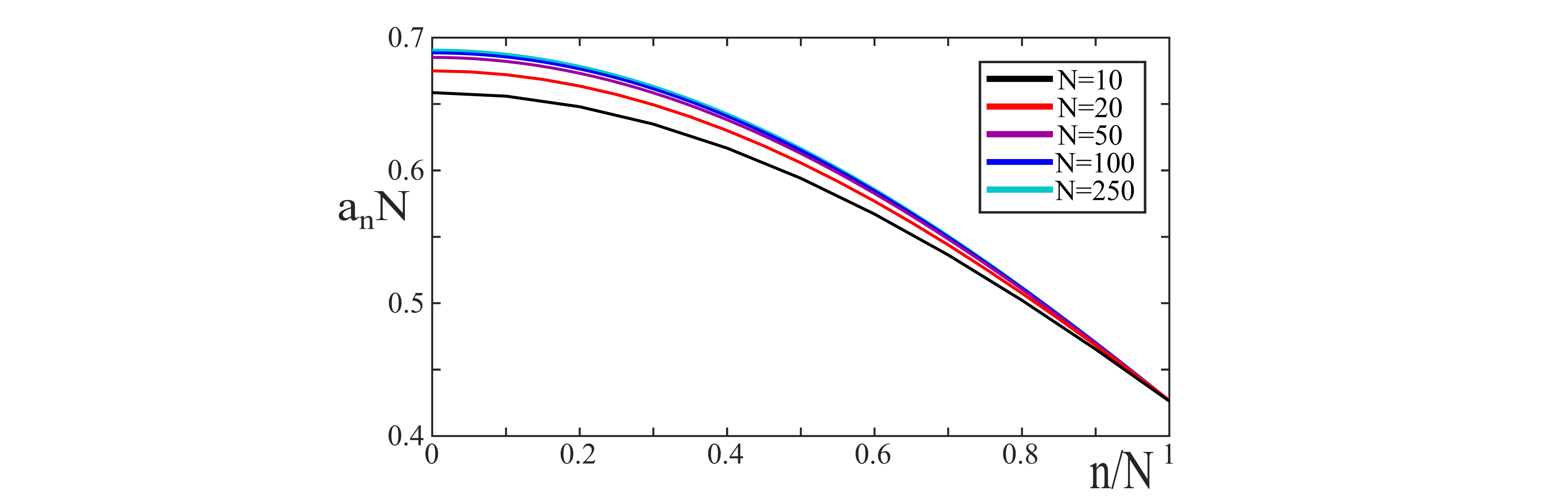}
    \caption{
    Solutions $\{a_n\}_{n = 0}^N$ to the matching equation \eqref{MatchEq2} with $m$ even plotted as $(\frac{n}{N},N a_n)$. The result appears to be converging to a single continuous curve $\alpha_*:[0,1]\to\mathbb{R}$, a solution to the continuum matching equation \eqref{MatchCont}. }
    \label{fig:LargeN}
\end{figure}

%%%%%%%%%%%%%%%%%%%%%%%%%%%%%%%%%%%%%%%%%
\section{Proof of Theorem~\ref{thm:Ring}}\label{sec:Matching}

In this section we provide the proof of Theorem~\ref{thm:Ring} by dividing the dynamics in the radial variable $r$ of the Galerkin system \eqref{eqn:R-D;Galerk} into separate regions. Upon analyzing each of the distinct regions in $r > 0$, we then provide the necessary conditions for matching the solutions in each region together. 

To formulate the problem properly, we express \eqref{eqn:R-D;Galerk} as the first-order system
\begin{equation}\label{R-D:U;vec}
    \frac{\textnormal{d}}{\textnormal{d} r}\mathbf{U} = \mathcal{A}(r)\mathbf{U} + \mathbf{F}(\mathbf{U},\mu), 
\end{equation}
where $\mathbf{U}:=[(\mathbf{u}_{n},\mathbf{v}_{n})^{T}]_{n=0}^{N}$, $\mathbf{v}_{n}(r):=\frac{\textnormal{d}}{\textnormal{d} r}\mathbf{u}_{n}(r)$, $\mathcal{A}(r) = \mathrm{diag}(\mathcal{A}_0(r),\mathcal{A}_1(r),\dots,\mathcal{A}_N(r))$ and $\mathbf{F}(\mathbf{U};\mu) = [\mathbf{F}_n(\mathbf{U};\mu)]_{n = 0}^N$ with
\begin{equation}\label{DecoupledSystem}
    \begin{split}
        \mathcal{A}_{n}(r) &= \begin{pmatrix}
        \mathbb{O}_{2} & \mathbbm{1}_{2} \\ \tfrac{(mn)^2}{r^2}\mathbbm{1}_{N} + \mathbf{M}_{1} & -\frac{1}{r}\mathbbm{1}_{2}
        \end{pmatrix},\\  
        \mathbf{F}_{n}(\mathbf{U};\mu) &= \begin{pmatrix}
        \mathbf{0} \\ \displaystyle \mu \mathbf{M}_{2}\mathbf{u}_{n} + \sum_{i+j=n} \mathbf{Q}(\mathbf{u}_{|i|},\mathbf{u}_{|j|}) + \sum_{i+j+k=n} \mathbf{C}(\mathbf{u}_{|i|},\mathbf{u}_{|j|},\mathbf{u}_{|k|})
        \end{pmatrix},
    \end{split}
\end{equation}
for each $n\in[0,N]$. In the above we have denoted $\mathbb{O}_2$ and $\mathbbm{1}_2$ as the $2\times 2$ zero and identity matrices, respectively. The nonlinear sums in $\mathbf{F}_{n}(\mathbf{U};\mu)$ can equivalently be written as
\begin{equation}
	\begin{split}
    		\sum_{i + j = n} \mathbf{Q}(\mathbf{u}_{|i|},\mathbf{u}_{|j|}) & = \sum_{i= \max\{n,0\}- N}^{\min\{n,0\} + N} \mathbf{Q}(\mathbf{u}_{|i|},\mathbf{u}_{|n-i|}),\\
    		\sum_{i + j + k = n} \mathbf{C}(\mathbf{u}_{|i|},\mathbf{u}_{|j|},\mathbf{u}_{|k|}) &= \sum_{i=-N}^{N} \left\{\sum_{j = \max\{n-i,0\} - N}^{\min\{n-i,0\} + N} \mathbf{C}(\mathbf{u}_{|i|},\mathbf{u}_{|j|},\mathbf{u}_{|n-i-j|})\right\}.
	\end{split}
\end{equation} 

Over the coming subsections we break down that analysis of system \eqref{R-D:U;vec} to identify the solutions \eqref{RadialProfile} provided in Theorem~\ref{thm:Ring}. This process is summarised visually in Figure~\ref{fig:CoreFar}. We begin in \S\ref{subsec:Core} with a characterisation of all small-amplitude solutions of \eqref{R-D:U;vec} that remain bounded as $r \to 0^+$, which we refer to as the core manifold. Then, in \S\ref{subsec:Far} we characterise the set of all solutions to \eqref{R-D:U;vec} that decay exponentially to zero as $r \to \infty$. This set of solutions is referred to as the far-field manifold, and solutions belonging to this manifold guarantee the spatial localisation of our identified solutions. In \S\ref{subsec:Rescaling} we define an appropriate rescaling of the variables to pull the solutions in the far-field back to be matched up with the core manifold. We then conclude this section in \S\ref{subsec:CoreFarMatch} where we match solutions in the core and far-field manifolds, leading to the matching condition \eqref{MatchEq} provided in the statement of Theorem~\ref{thm:Ring}.   

\begin{figure}[t!] %Figure: Invariant Manifold Schematic
    \centering
    \includegraphics[width=\linewidth]{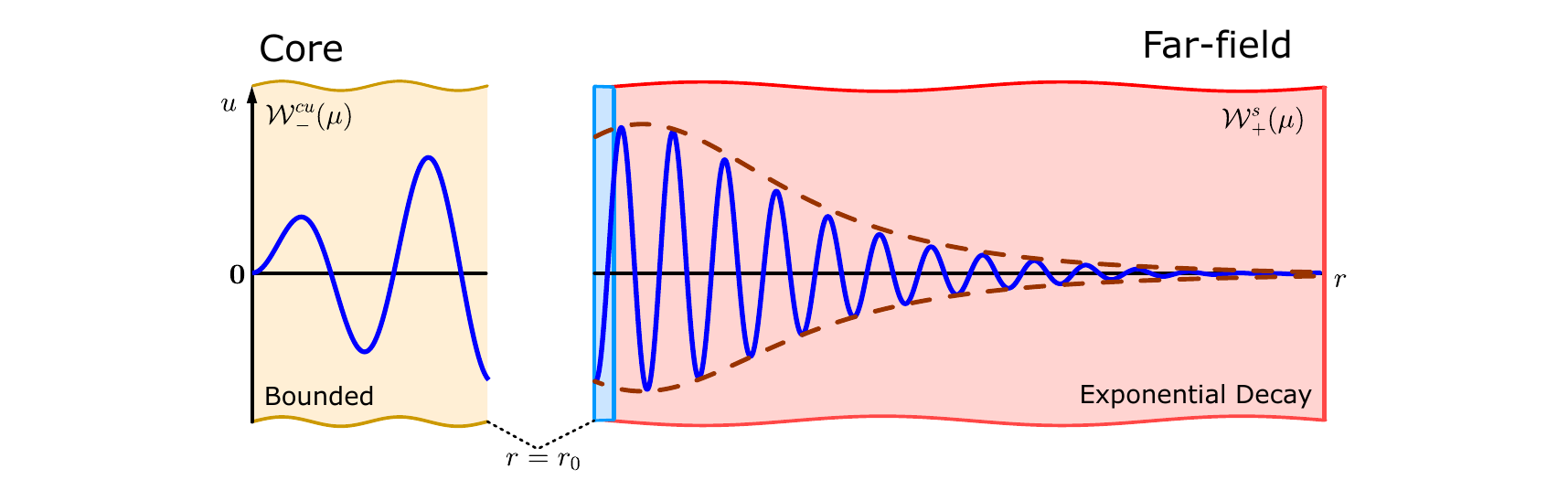}
    \caption{The proof of Theorem~\ref{thm:Ring} is achieved by matching solutions in the core manifold, $\mathcal{W}^{cu}_-(\mu)$, and the far-field manifold, $\mathcal{W}^s_+(\mu)$, for $0 < \mu \ll 1$. The core manifold is constructed locally over the interval $r \in [0,r_0]$, while the far-field is constructed over $r \in [r_0,\infty)$. We then identify solutions lying in the intersection of both manifolds at $r = r_0$ to provide an appropriate solution to the Galerkin system \eqref{eqn:R-D;Galerk}}
    \label{fig:CoreFar}
\end{figure}

%%%%%%%%%%%%%%%%%%%%%%%%%%%%%%%%%%%%%%%%%%%%%%%%%%%%%%%%%%%%%%%%%%%%%%%%%%%%
\subsection{The Core Manifold}\label{subsec:Core}

The core manifold, $\mathcal{W}_{-}^{cu}(\mu)$, captures all small-amplitude solutions of \eqref{eqn:R-D;Galerk} that remain bounded as $r \to 0^+$ when $0 < \mu \ll 1$. First, since we are interested in small-amplitude solutions near $\mu = 0$, we begin with the linearisation of \eqref{eqn:R-D;Galerk} about $\mathbf{U} = 0$ at $\mu = 0$, given by
\begin{equation}\label{Vlinear}
	\frac{\textnormal{d}}{\textnormal{d} r}\mathbf{V} = \mathcal{A}(r)\mathbf{V}, \qquad \mathbf{V}\in\mathbb{R}^{4(N+1)}. 
\end{equation}
The analysis of \eqref{Vlinear} is simplified considerably by noting that it decouples into $(N+1)$ distinct equations, given by 
\begin{equation}
	\frac{\textnormal{d}}{\textnormal{d} r}\mathbf{V}_{n} = \mathcal{A}_{n}(r)\mathbf{V}_{n}, \qquad \mathbf{V}_{n}\in\mathbb{R}^{4}, 
\end{equation} 
for each $n\in[0,N]$. Recall that $\mathcal{A}_n(r)$ is given in \eqref{DecoupledSystem}. { As shown in \cite{hill2022approximate}}, solutions are given as linear combinations of Bessel functions, $J_\nu(r)$ and $Y_\nu(r)$, so that 
\begin{equation}\label{Ntup:SH;lin,soln}
	\mathbf{V}_{n}(r) = \sum_{j=1}^{4}  d^{(n)}_{j} \mathbf{V}^{(n)}_{j}(r), 
\end{equation}
where   
\begin{equation}\label{Lin:Solns}
   \begin{split} 
   \mathbf{V}_{1}^{(n)}(r) &= \sqrt{\frac{\pi}{2}}\begin{pmatrix}
    J_{mn}(r)\hat{U}_{0} \\ \frac{\textnormal{d}}{\textnormal{d} r}\big[J_{mn}(r)\big]\hat{U}_{0}
    \end{pmatrix}, \qquad \mathbf{V}_{2}^{(n)}(r) = \sqrt{\frac{\pi}{2}}\begin{pmatrix}
    r J_{mn +1 }(r) \hat{U}_{0} + 2 J_{mn}(r)\hat{U}_{1} \\ \frac{\textnormal{d}}{\textnormal{d} r}\big[r J_{mn +1}(r)\big]\hat{U}_{0} + 2 \frac{\textnormal{d}}{\textnormal{d} r}\big[J_{mn}(r)\big]\hat{U}_{1}
    \end{pmatrix}, \\
    \mathbf{V}_{3}^{(n)}(r) &= \sqrt{\frac{\pi}{2}}\begin{pmatrix}
    Y_{mn}(r)\hat{U}_{0} \\ \frac{\textnormal{d}}{\textnormal{d} r}\big[Y_{mn}(r)\big]\hat{U}_{0}
    \end{pmatrix}, \qquad \mathbf{V}_{4}^{(n)}(r) = \sqrt{\frac{\pi}{2}}\begin{pmatrix}
    r Y_{mn +1 }(r) \hat{U}_{0} + 2  Y_{mn}(r)\hat{U}_{1} \\ \frac{\textnormal{d}}{\textnormal{d} r}\big[r Y_{mn +1}(r)\big]\hat{U}_{0} + 2 \frac{\textnormal{d}}{\textnormal{d} r}\big[Y_{mn}(r)\big]\hat{U}_{1}
    \end{pmatrix}.
    \end{split}
\end{equation}
Using this information, we write the solution to the full $4(N+1)$-dimensional equation \eqref{Vlinear} as 
\begin{equation}
    \mathbf{V}(r) = \sum_{n=0}^{N} \left\{\sum_{i=1}^{4} d_{i}^{(n)}\mathscr{V}_{i}^{(n)}(r)\right\},
\end{equation}
where 

\begin{equation}
    \left[\mathscr{V}^{(n)}_{i}(r)\right]_{k}= \begin{cases}
        \left[\mathbf{V}_{i}^{(n)}(r)\right]_{k-4n}, & 4n < k \leq 4(n+1), \\
        0, & \textnormal{otherwise}.
    \end{cases} 
\end{equation}
Importantly, for each $n \in [0,N]$ only $\mathscr{V}_{1}^{(n)}(r)$ and $\mathscr{V}_{2}^{(n)}(r)$ remain bounded as $r \to 0^+$, while both $\mathscr{V}_{3}^{(n)}(r)$ and $\mathscr{V}_{4}^{(n)}(r)$ blow up. Thus, to characterise the solutions of \eqref{R-D:U;vec} that remain bounded as $r \to 0^+$ we will define $\mathcal{P}^{cu}_{-}(r_{0})$ for a fixed $r_0 > 0$ to be the projection onto $\{\mathscr{V}_{1}^{(n)}(r_{0}),\mathscr{V}_{2}^{(n)}(r_{0})\}_{n=0}^{N} \subset \mathbb{R}^{4(N+1)}$, along $\{\mathscr{V}_{3}^{(n)}(r_{0}),\mathscr{V}_{4}^{(n)}(r_{0})\}_{n=0}^{N}$. 

We are now able to completely characterise the core manifold using the following lemma which was proven in \cite{hill2022approximate}. Throughout we will adopt the notation $\mathcal{O}_{r_0}(\cdot)$ to have the same meaning as the standard Landau symbol $\mathcal{O}(\cdot)$, except that the constants may depend on $r_0 > 0$.  

\begin{lem}[\cite{hill2022approximate}, Lemma~4.1]\label{Lemma:Ntup;Core} %Lemma: Core manifold
  Fix $m,N\in\mathbb{N}$. For each fixed $r_{0}>0$, there are constants $\delta_{1},\delta_{2}>0$ such that the set $\mathcal{W}^{cu}_{-}(\mu)$ of solutions $\mathbf{U}(r)$ of \eqref{R-D:U;vec} for which $\sup_{0\leq r\leq r_{0}}\|\mathbf{U}(r)\|<\delta_{1}$ is, for $|\mu|<\delta_{1}$, a smooth $2(N+1)$ dimensional manifold. Furthermore, each $\mathbf{U}(r_{0})\in\mathcal{W}^{cu}_{-}(\mu)$ with $|\mathcal{P}^{cu}_{-}(r_{0})\mathbf{U}(r_{0})|<\delta_{2}$ can be written uniquely as 
\begin{equation}\label{U:Core;Ntup}
	\begin{split}
 		\mathbf{U}(r_{0}) &= \sum_{n=0}^{N} \{d_{1}^{(n)}\mathscr{V}_{1}^{(n)}(r_{0}) + d_{2}^{(n)}\mathscr{V}_{2}^{(n)}(r_{0}) + \mathscr{V}_{3}^{(n)}(r_{0})\mathcal{O}_{r_0}(|\mu||\mathbf{d}| + |\mathbf{d}|^{2})\\
 &\qquad \qquad  + \mathscr{V}_{4}^{(n)}(r_{0})[\nu Q^{m}_{n}(\mathbf{d}_{1}) + \mathcal{O}_{r_0}(|\mu||\mathbf{d}| + |\mathbf{d}_{2}|^{2} + |\mathbf{d}_{1}|^{3})]\},
	\end{split}
\end{equation}
where $\nu := \frac{1}{2}\sqrt{\frac{\pi}{6}}\big\langle \hat{U}_{1}^{*}, \mathbf{Q}(\hat{U}_{0},\hat{U}_{0})\big\rangle_{2}$, $\mathbf{d}_{j}:=\left(d_{j}^{(0)}, d_{j}^{(1)},\dots,d_{j}^{(N)}\right)$, $\mathbf{d}:=\left(\mathbf{d}_{1},\mathbf{d}_{2}\right)\in\mathbb{R}^{2(N+1)}$ with $|\mathbf{d}|<\delta_{2}$. Furthermore, the right-hand side of \eqref{U:Core;Ntup} depends smoothly on $(\mathbf{d},\mu)$, and the nonlinear functions $Q^{m}_{n}(\mathbf{d}_{1})$ are defined as
\begin{equation}\label{Psi:d1;defn}
	Q^{m}_{n}(\mathbf{d}_{1}) := 2\sum_{j=1}^{N-n} \cos\left(\frac{m\pi(n-j)}{3}\right) d^{(j)}_{1} d^{(n+j)}_{1} + \sum_{j=0}^{n} \cos\left(\frac{m\pi(n-2j)}{3}\right) d^{(j)}_{1} d^{(n-j)}_{1}.
\end{equation}
\end{lem}

\begin{rmk}
    We briefly note the behaviour of our solution in the core region $r\in[0,r_0]$. For sufficiently small $|\mathbf{d}|\ll1$, we find that 
    \begin{equation}
	\begin{split}
 		\mathbf{U}(r) &= \sum_{n=0}^{N} \{[d_{1}^{(n)} + \mathcal{O}_{r}(|\mu||\mathbf{d}| + |\mathbf{d}|^{2})]\mathscr{V}_{1}^{(n)}(r) + [d_{2}^{(n)} + \mathcal{O}_{r}(|\mu||\mathbf{d}| + |\mathbf{d}|^{2})]\mathscr{V}_{2}^{(n)}(r) \\
 &\qquad \qquad + \mathscr{V}_{3}^{(n)}(r)\mathcal{O}_{r}(|\mu||\mathbf{d}| + |\mathbf{d}|^{2})  + \mathscr{V}_{4}^{(n)}(r)\mathcal{O}_{r}(|\mu||\mathbf{d}| + |\mathbf{d}|^{2})\},
	\end{split}
\end{equation}
for each fixed $r\in[0,r_0]$. This follows from a variation-of-constants formula in the proof of \cite[Lemma 4.1]{hill2022approximate}. Furthermore, we recall that $\mathbf{U}=[(\mathbf{u}_{n},\mathbf{v}_{n})^{T}]_{n=0}^{N}$, and so , using the explicit forms of each $\mathscr{V}_{j}^{(n)}$, we can write
\begin{equation}\label{Core:Prof}
	\begin{split}
 		\mathbf{u}_n(r) &= d_{1}^{(n)}\sqrt{\tfrac{\pi}{2}}J_{mn}(r)\hat{U}_{0} + d_{2}^{(n)}\sqrt{\tfrac{\pi}{2}}[r J_{mn+1}(r)\hat{U}_{0} + 2J_{mn}(r)\hat{U}_{1}] + \mathcal{O}_{r}(|\mu||\mathbf{d}| + |\mathbf{d}|^{2}),\\
   \mathbf{v}_n(r) &= d_{1}^{(n)}\sqrt{\tfrac{\pi}{2}}\tfrac{\mathrm{d}}{\mathrm{d}r}[J_{mn}(r)]\hat{U}_{0} + d_{2}^{(n)}\sqrt{\tfrac{\pi}{2}}\big[\tfrac{\mathrm{d}}{\mathrm{d}r}[r J_{mn+1}(r)]\hat{U}_{0} + 2\tfrac{\mathrm{d}}{\mathrm{d}r}[J_{mn}(r)]\hat{U}_{1}\big] + \mathcal{O}_{r}(|\mu||\mathbf{d}| + |\mathbf{d}|^{2}),
	\end{split}
\end{equation}
for each $n\in[0,N]$.
\end{rmk}

In what follows we will be interested in the case when $r_0 \gg 1$. We begin by expressing the solution \eqref{U:Core;Ntup} in Lemma~\ref{Lemma:Ntup;Core} in terms of $\mathbf{u}_{n},\mathbf{v}_{n}\in\mathbb{R}^{2}$, where $\mathbf{U}=[(\mathbf{u}_n,\mathbf{v}_n)^{T}]_{n=0}^{N}$. Then, for each $n\in[0,N]$,
\begin{equation}
	\begin{split}
 \mathbf{u}_{n}(r_{0}) &= \sqrt{\tfrac{\pi}{2}}\big[d_{2}^{(n)}r_0 J_{mn+1}(r_0) + d_{1}^{(n)}J_{mn}(r_{0})  + \mathcal{O}_{r_0}(|\mu||\mathbf{d}| + |\mathbf{d}|^{2})\big]\hat{U}_{0}\\
 & \qquad  + 2\sqrt{\tfrac{\pi}{2}}\big[\nu Q^{m}_{n}(\mathbf{d}_{1})Y_{mn}(r_0) + d_{2}^{(n)}J_{mn}(r_0) + \mathcal{O}_{r_0}(|\mu||\mathbf{d}| + |\mathbf{d}_{2}|^{2} + |\mathbf{d}_{1}|^{3})\big]\hat{U}_{1},\\
 \mathbf{v}_{n}(r_{0}) &= \sqrt{\tfrac{\pi}{2}}\big[ d_{2}^{(n)} \tfrac{\mathrm{d}}{\mathrm{d}r}[r J_{mn+1}(r) ]_{r=r_0} + d_{1}^{(n)}\tfrac{\mathrm{d}}{\mathrm{d}r}[J_{mn}(r)]_{r=r_0} + \mathcal{O}_{r_0}(|\mu||\mathbf{d}| + |\mathbf{d}|^{2})\big]\hat{U}_{0}\\
 & \qquad  + 2\sqrt{\tfrac{\pi}{2}}\big[\nu Q^{m}_{n}(\mathbf{d}_{1}) \tfrac{\mathrm{d}}{\mathrm{d}r}[Y_{mn}(r)]_{r=r_0} + d_{2}^{(n)}\tfrac{\mathrm{d}}{\mathrm{d}r}[J_{mn}(r)]_{r=r_0} + \mathcal{O}_{r_0}(|\mu||\mathbf{d}| + |\mathbf{d}_{2}|^{2} + |\mathbf{d}_{1}|^{3})\big]\hat{U}_{1},\\
	\end{split}
\end{equation}
where the $\mathcal{O}_{r_{0}}(\cdot)$ remainders capture the higher order terms when $|\mathbf{d}|$ and $|\mu|$ are taken to be small. With the large $r$ asymptotics of the Bessel functions of the first kind $J_\nu(r)$ and their derivatives, we can expand $\mathbf{u}_{n}(r_{0}),\mathbf{v}_{n}(r_{0})$ as the following, 
\begin{equation}
\label{Core:un}
    \begin{split}
    \mathbf{u}_{n}(r_{0}) &= r_{0}^{-\frac{1}{2}}\left[d_{2}^{(n)}r_{0}\left[1 + \mathcal{O}\left(r_{0}^{-1}\right)\right] \sin(y_{n}) + d_{1}^{(n)}\left[1 + \mathcal{O}\left(r_{0}^{-1}\right)\right]\cos(y_{n}) + \mathcal{O}_{r_{0}}(|\mu||\mathbf{d}| + |\mathbf{d}|^{2})\right]\hat{U}_{0}\\
    &\quad + 2 r_{0}^{-\frac{1}{2}}\left[[\nu + \mathcal{O}(r_{0}^{-\frac{1}{2}})]Q_{n}^{m}(\mathbf{d}_{1}) \sin(y_{n}) + d_{2}^{(n)}\left[1 + \mathcal{O}\left(r_{0}^{-1}\right)\right]\cos(y_{n}) + \mathcal{O}_{r_{0}}(|\mu||\mathbf{d}| + |\mathbf{d}_{2}|^{2} + |\mathbf{d}_{1}|^{3})\right]\hat{U}_{1},\\
    \mathbf{v}_{n}(r_{0}) &= r_{0}^{-\frac{1}{2}}\left[d_{2}^{(n)}r_{0}\left[1 + \mathcal{O}\left(r_{0}^{-1}\right)\right] \cos(y_{n}) - d_{1}^{(n)}\left[1 + \mathcal{O}\left(r_{0}^{-1}\right)\right]\sin(y_{n}) + \mathcal{O}_{r_{0}}(|\mu||\mathbf{d}| + |\mathbf{d}|^{2})\right]\hat{U}_{0}\\
    &\quad + 2 r_{0}^{-\frac{1}{2}}\left[[\nu + \mathcal{O}(r_{0}^{-\frac{1}{2}})]Q_{n}^{m}(\mathbf{d}_{1}) \cos(y_{n}) -d_{2}^{(n)}\left[1 + \mathcal{O}\left(r_{0}^{-1}\right)\right]\sin(y_{n}) + \mathcal{O}_{r_{0}}(|\mu||\mathbf{d}| + |\mathbf{d}_{2}|^{2} + |\mathbf{d}_{1}|^{3})\right]\hat{U}_{1},
    \end{split}
\end{equation} 
for each $n\in[0,N]$, where we have written $y_{n}:=r_{0} - \frac{mn\pi}{2} - \frac{\pi}{4}$. We note that the asymptotic expansion of $J_{\nu}(r)$ is only valid when $r\gg |\nu^2 - \tfrac{1}{4}|$, and so we choose $r_0$ such that $r_0\gg (mN)^2$.

%%%%%%%%%%%%%%%%%%%%%%%%%%%%%%%%%%%%%%%%%%%%%%%%%%%%%%%%%%%%%%%%%%%%%%%%%%%%
\subsection{The Far-Field Manifold}\label{subsec:Far}

The next step of the proof is to understand the far-field manifold, denoted $\mathcal{W}_{+}^{s}(\mu)$, containing all small-amplitude solutions to \eqref{eqn:R-D;Galerk}, $\mathbf{U}(r)$, that exponentially decay as $r \to\infty$. Following Scheel~\cite{scheel2003radially}, we carry out a radial normal form analysis by working with the extended autonomous system
\begin{equation}\label{R-D:Farf}
	\begin{split}
    		\frac{\textnormal{d}}{\textnormal{d} r}\mathbf{U}_{n} &= \mathcal{A}_{\infty}\mathbf{U}_{n} + \widetilde{\mathbf{F}}_{n}(\mathbf{U}; \mu, \sigma),\qquad \forall n\in[0,N], \\
    		\frac{\textnormal{d}}{\textnormal{d} r}\sigma &= -\sigma^{2}.
	\end{split}
\end{equation}
{The above system has the property that $\sigma(r)= 1/r$ is an invariant subspace of \eqref{R-D:Farf}, which in turn leads to the radial system we began with. The reason for this additional variable $\sigma$, is that one needs a smooth autonomous system to apply many tools from dynamical systems. We note that $\mathcal{A}_{\infty}$ and $\widetilde{\mathbf{F}}_n$ are now autonomous and, in particular, 
\begin{equation*}
\begin{split}
\mathcal{A}_{\infty} ={}&\begin{pmatrix}
        \mathbb{O}_2 & \mathbbm{1}_{2}\\
        \mathbf{M}_1 & \mathbb{O}_2
    \end{pmatrix}, \\ \widetilde{\mathbf{F}}_n(\mathbf{U};\mu,\sigma) ={}& \begin{pmatrix}
        \mathbf{0} \\ \displaystyle (mn)^2\sigma^2 \mathbf{u}_n - \sigma \mathbf{v}_n + \mu \mathbf{M}_{2}\mathbf{u}_{n} + \sum_{i+j=n} \mathbf{Q}(\mathbf{u}_{|i|},\mathbf{u}_{|j|}) + \sum_{i+j+k=n} \mathbf{C}(\mathbf{u}_{|i|},\mathbf{u}_{|j|},\mathbf{u}_{|k|})
        \end{pmatrix}.
        \end{split}
\end{equation*}
% now includes all the radial terms, with the linear operator $\mathcal{A}_{\infty}$ becoming autonomous.  
}

To better analyse \eqref{R-D:Farf}, we introduce the complex-valued change of variable $\widetilde{\mathbf{A}}:=[\widetilde{A}_{n}]_{n=0}^{N}$, $\widetilde{\mathbf{B}}:=[\widetilde{B}_{n}]_{n=0}^{N}$, so that
\begin{equation}
    \begin{split}
    \mathbf{u}_{n} &= (\widetilde{A}_{n} + \overline{\widetilde{A}}_{n})\hat{U}_{0} + 2\textnormal{i}(\widetilde{B}_{n} - \overline{\widetilde{B}}_{n})\hat{U}_{1} , \qquad \mathbf{v}_{n} = \big[\textnormal{i}(\widetilde{A}_{n} - \overline{\widetilde{A}}_{n}) + (\widetilde{B}_{n} + \overline{\widetilde{B}}_{n})\big]\hat{U}_{0} - 2(\widetilde{B}_{n} + \overline{\widetilde{B}}_{n})\hat{U}_{1},\\
    \widetilde{A}_{n} &= \frac{1}{2}\big\langle \hat{U}_{0}^{*}, \mathbf{u}_{n} - \textnormal{i}\mathbf{v}_{n}\big\rangle_{2} - \frac{\textnormal{i}}{4}\big\langle \hat{U}_{1}^{*}, \mathbf{v}_{n}\big\rangle_{2}, \qquad \widetilde{B}_{n} = -\frac{\textnormal{i}}{4}\big\langle \hat{U}_{1}^{*}, \mathbf{u}_{n} - \textnormal{i}\mathbf{v}_{n}\big\rangle_{2},
    \end{split}
    \label{R-D:Transformation;Farfield}
\end{equation}
for each $n\in[0,N]$. With this change of variable \eqref{R-D:Farf} becomes
\begin{equation}
	\begin{split}
   		 \frac{\textnormal{d}}{\textnormal{d} r} \widetilde{\mathbf{A}} &= \textnormal{i}\widetilde{\mathbf{A}} - \frac{\sigma}{2}(\widetilde{\mathbf{A}} - \overline{\widetilde{\mathbf{A}}}) + \widetilde{\mathbf{B}} - \frac{\textnormal{i}}{2}[ \sigma^{2}\mathcal{C}_{N}^{m}[(\widetilde{\mathbf{A}} + \overline{\widetilde{\mathbf{A}}}) + \textnormal{i}(\widetilde{\mathbf{B}} - \overline{\widetilde{\mathbf{B}}})] + \mathcal{F}_{A}(\widetilde{\mathbf{A}},\widetilde{\mathbf{B}})],\\
    		\frac{\textnormal{d}}{\textnormal{d} r} \widetilde{\mathbf{B}} &= \textnormal{i}\widetilde{\mathbf{B}} - \frac{\sigma}{2}(\widetilde{\mathbf{B}} + \overline{\widetilde{\mathbf{B}}})  - \frac{1}{2}[ \sigma^{2}\mathcal{C}_{N}^{m}\textnormal{i}(\widetilde{\mathbf{B}} - \overline{\widetilde{\mathbf{B}}}) + \mathcal{F}_{B}(\widetilde{\mathbf{A}},\widetilde{\mathbf{B}})],\label{amp:AB;tilde} \\
    \frac{\textnormal{d}}{\textnormal{d} r} \sigma &= -\sigma^{2},
	\end{split}
\end{equation}
where we have defined $\mathcal{C}^{m}_{N}:=\textrm{diag}(0,(m)^{2},\dots,(mN)^{2})$, $\left[\mathcal{F}_{A}\right]_{n} := \big\langle \hat{U}_{0}^{*} + \frac{1}{2}\hat{U}_{1}^{*}, \mathcal{F}_{n}\big\rangle$,  $\left[\mathcal{F}_{B}\right]_{n} := \big\langle \frac{1}{2}\hat{U}_{1}^{*}, \mathcal{F}_{n}\big\rangle$, and 
\begin{align}
    \mathcal{F}_{n} &= \mu \bigg[\big(\widetilde{A}_{n} + \overline{\widetilde{A}}_{n}\big) \mathbf{M}_{2}\hat{U}_{0} + 2 \textnormal{i} \big(\widetilde{B}_{n} - \overline{\widetilde{B}}_{n}\big)\mathbf{M}_{2}\hat{U}_{1}\bigg] + Q_{0,0}\sum_{i+j=n}\widetilde{A}_{|i|} \widetilde{A}_{|j|} + 2 Q_{0,0}\sum_{i+j=n}\widetilde{A}_{|i|} \overline{\widetilde{A}}_{|j|} \nonumber\\
    & \quad + Q_{0,0}\sum_{i+j=n}\overline{\widetilde{A}}_{|i|} \overline{\widetilde{A}}_{|j|} + 4\textnormal{i} Q_{0,1}\sum_{i+j=n}\widetilde{A}_{|i|} \widetilde{B}_{|j|} + 4\textnormal{i} Q_{0,1}\sum_{i+j=n}\overline{\widetilde{A}}_{|i|} \widetilde{B}_{|j|} - 4\textnormal{i} Q_{0,1}\sum_{i+j=n}\widetilde{A}_{|i|} \overline{\widetilde{B}}_{|j|} \nonumber\\
    &\quad - 4\textnormal{i} Q_{0,1}\sum_{i+j=n} \overline{\widetilde{A}}_{|i|} \overline{\widetilde{B}}_{|j|} - 4 Q_{1,1}\sum_{i+j=n}  \widetilde{B}_{|i|} \widetilde{B}_{|j|} + 8 Q_{1,1}\sum_{i+j=n}  \widetilde{B}_{|i|} \overline{\widetilde{B}}_{|j|} - 4 Q_{1,1}\sum_{i+j=n}  \overline{\widetilde{B}}_{|i|} \overline{\widetilde{B}}_{|j|} \nonumber\\
    &\quad + C_{0,0,0} \sum_{i+j+k=n} \widetilde{A}_{|i|} \widetilde{A}_{|j|} \widetilde{A}_{|k|} + 3 C_{0,0,0} \sum_{i+j+k=n} \widetilde{A}_{|i|} \widetilde{A}_{|j|} \overline{\widetilde{A}}_{|k|}  + 3 C_{0,0,0} \sum_{i+j+k=n} \overline{\widetilde{A}}_{|i|} \overline{\widetilde{A}}_{|j|} \widetilde{A}_{|k|} \nonumber\\
    &\quad   + C_{0,0,0} \sum_{i+j+k=n} \overline{\widetilde{A}}_{|i|} \overline{\widetilde{A}}_{|j|} \overline{\widetilde{A}}_{|k|} + 6\textnormal{i} C_{0,0,1}\sum_{i+j+k=n}\widetilde{A}_{|i|} \widetilde{A}_{|j|} \widetilde{B}_{|k|} + 12\textnormal{i} C_{0,0,1}\sum_{i+j+k=n}\widetilde{A}_{|i|} \overline{\widetilde{A}}_{|j|} \widetilde{B}_{|k|}\nonumber\\
    &\quad + 6\textnormal{i} C_{0,0,1}\sum_{i+j+k=n}\overline{\widetilde{A}}_{|i|}\overline{\widetilde{A}}_{|j|} \widetilde{B}_{|k|} - 6\textnormal{i} C_{0,0,1}\sum_{i+j+k=n}\widetilde{A}_{|i|} \widetilde{A}_{|j|} \overline{\widetilde{B}}_{|k|} - 12\textnormal{i} C_{0,0,1}\sum_{i+j+k=n}\widetilde{A}_{|i|} \overline{\widetilde{A}}_{|j|} \overline{\widetilde{B}}_{|k|}\nonumber\\
    &\quad - 6\textnormal{i} C_{0,0,1}\sum_{i+j+k=n}\overline{\widetilde{A}}_{|i|} \overline{\widetilde{A}}_{|j|} \overline{\widetilde{B}}_{|k|} - 12 C_{0,1,1}\sum_{i+j+k=n} \widetilde{A}_{|i|} \widetilde{B}_{|j|} \widetilde{B}_{|k|} + 24 C_{0,1,1}\sum_{i+j+k=n} \widetilde{A}_{|i|} \overline{\widetilde{B}}_{|j|} \widetilde{B}_{|k|} \nonumber\\
    &\quad - 12 C_{0,1,1}\sum_{i+j+k=n} \widetilde{A}_{|i|} \overline{\widetilde{B}}_{|j|} \overline{\widetilde{B}}_{|k|} - 12 C_{0,1,1}\sum_{i+j+k=n} \overline{\widetilde{A}}_{|i|} \widetilde{B}_{|j|} \widetilde{B}_{|k|} + 24 C_{0,1,1}\sum_{i+j+k=n} \overline{\widetilde{A}}_{|i|} \overline{\widetilde{B}}_{|j|} \widetilde{B}_{|k|} \nonumber\\
    &\quad - 12 C_{0,1,1}\sum_{i+j+k=n} \overline{\widetilde{A}}_{|i|} \overline{\widetilde{B}}_{|j|} \overline{\widetilde{B}}_{|k|} - 8\textnormal{i} C_{1,1,1} \sum_{i+j+k=n}\widetilde{B}_{|i|} \widetilde{B}_{|j|} \widetilde{B}_{|k|} + 24\textnormal{i} C_{1,1,1} \sum_{i+j+k=n}\widetilde{B}_{|i|} \widetilde{B}_{|j|} \overline{\widetilde{B}}_{|k|} \nonumber\\
    &\quad - 24\textnormal{i} C_{1,1,1} \sum_{i+j+k=n}\overline{\widetilde{B}}_{|i|} \overline{\widetilde{B}}_{|j|} \widetilde{B}_{|k|} + 8\textnormal{i} C_{1,1,1} \sum_{i+j+k=n}\overline{\widetilde{B}}_{|i|} \overline{\widetilde{B}}_{|j|} \overline{\widetilde{B}}_{|k|},\nonumber
\end{align}
for each $n\in[0,N]$. We recall from Hypothesis~\ref{R-D:hyp;2} that $Q_{i,j} := \mathbf{Q}(\hat{U}_{i},\hat{U}_{j})$, $C_{i,j,k} := \mathbf{C}(\hat{U}_{i},\hat{U}_{j},\hat{U}_{k})$ for $i,j,k\in\{0,1\}$. 

The following lemma introduces a nonlinear normal form transformation into \eqref{amp:AB;tilde} to eliminate the non-resonant terms from the system. The proof is left to Appendix~\ref{app:LemProof} and follows from standard normal form computations. A notable example of such computations can be found in \cite{scheel2003radially} for radial systems such as ours in this manuscript.

\begin{lem}\label{Lemma:Ntup;normal} %Lemma: Normal form computation
Fix $N\in\mathbb{N}$. Then, for each $n\in[0,N]$, there exists a change of coordinates
\begin{equation}\label{Ntup:Normal;transf}
    \begin{pmatrix}
    A_{n} \\ B_{n}
    \end{pmatrix} := \textnormal{e}^{-\textnormal{i}\phi_{n}(r)}\left[\mathbbm{1} + \mathcal{T}_{n}(\sigma)\right]\begin{pmatrix}
    \widetilde{A}_{n} \\ \widetilde{B}_{n}
    \end{pmatrix} + \mathcal{O}((|\mu| + |\widetilde{\mathbf{A}}| + |\widetilde{\mathbf{B}}|)(|\widetilde{\mathbf{A}}| + |\widetilde{\mathbf{B}}|)),
\end{equation}
such that \eqref{amp:AB;tilde} becomes
\begin{equation}\label{Ntup:NormalForm}
	\begin{split}
    \frac{\textnormal{d}}{\textnormal{d}r} \mathbf{A} &= - \frac{\sigma}{2} \mathbf{A} + \mathbf{B} + \mathbf{R}_{\mathbf{A}}(\mathbf{A}, \mathbf{B},\sigma,\mu),\\
    \frac{\textnormal{d}}{\textnormal{d}r} \mathbf{B} &= -\frac{\sigma}{2} \mathbf{B} + c_{0}\,\mu \mathbf{A} + c_3 \mathbf{C}^{m}_{N}(\mathbf{A}) + \mathbf{R}_{\mathbf{B}}(\mathbf{A}, \mathbf{B},\sigma,\mu),\\
    \frac{\textnormal{d}}{\textnormal{d} r} \sigma &= -\sigma^{2},
	\end{split}
\end{equation}
where $c_{0}:= \frac{1}{4}\big\langle \hat{U}_{1}^{*}, -\mathbf{M}_{2}\hat{U}_{0}\big\rangle_{2}$,
\begin{equation}\label{c3:defn}
    \begin{split}
    c_{3} :&= - \bigg[\bigg(\frac{5}{6}\big[\big\langle \hat{U}_{0}^{*},Q_{0,0}\big\rangle_{2} + \big\langle \hat{U}_{1}^{*},Q_{0,1}\big\rangle_{2}\big] + \frac{19}{18}\big\langle \hat{U}_{1}^{*},Q_{0,0}\big\rangle_{2} \bigg)\big\langle \hat{U}_{1}^{*},Q_{0,0}\big\rangle_{2} +\frac{3}{4}\big\langle \hat{U}_{1}^{*}, C_{0,0,0}\big\rangle_{2}\bigg], \end{split}
\end{equation}
and
\begin{equation}\label{CN:defn}
    \begin{split}
    \big[\mathbf{C}^{m}_{N}(\mathbf{A})\big]_{n} :&= \sum_{i+j + k =n}(-1)^{\frac{m}{2}(|i| + |j| - |k| - n)} A_{|i|} A_{|j|}\overline{A}_{|k|}.
    \end{split}
\end{equation}
For each $n\in[0,N]$, the coordinate change \eqref{Ntup:Normal;transf} is polynomial in $(A_{n},B_{n},\sigma)$ and smooth in $\mu$, and $\mathcal{T}_{n}(\sigma) = \mathcal{O}(\sigma)$ is linear and upper triangular for each $\sigma$. The remainder terms satisfy
\begin{equation}
[\mathbf{R}_{\mathbf{A}/\mathbf{B}}]_{n} = \mathcal{O}([|\mu|^{2} + |\sigma|^{3} + (|\mathbf{A}| + |\mathbf{B}|)^{2}]|\mathbf{B}| + [|\mu|^{2} + |\sigma|^{3} + |\mathbf{A}|^{3}]|\mathbf{A}|),
\end{equation}
while $\phi_{n}(r)$ satisfies
\begin{equation}
    \frac{\textnormal{d}}{\textnormal{d} r}\phi_{n} = 1 + \mathcal{O}\left(|\mu| + |\sigma|^{2}\right), \qquad \phi_{n}(0) = {\textstyle - \frac{m n \pi}{2}}, \qquad \forall n\in[0,N].\nonumber
\end{equation}
\end{lem}

\begin{rmk}
    The main difference between the result of Lemma~\ref{Lemma:Ntup;normal} and its companion result \cite[Lemma~4.4]{hill2022approximate} is the presence of the cubic term $\mathbf{C}^m_N(\mathbf{A})$. For spot A-type solutions, the far-field behaviour is determined by a linear flow, and so the cubic-order terms are negligible in the analysis. As such, the normal form in \cite{hill2022approximate} does not contain an explicit form for the cubic nonlinearity, but rather absorbs it into the higher-order remainder terms $\mathbf{R}_{\mathbf{A}/\mathbf{B}}$. We expect ring-type solutions to depend on a cubic Ginzburg-Landau equation in the far-field, as seen in the radial problem \cite{lloyd2009localized,mccalla2013spots}, and so we explicitly calculate this term. 
\end{rmk}

%%%%%%%%%%%%%%%%%%%%%%%%%%%%%%%%%%%%%%%%%%%%%%%%%%%%%%%%%%%%%%%%%%%%%%%%%%%%
\subsection{The Rescaling and Transition Charts}\label{subsec:Rescaling}

The goal of this subsection is to define rescaling coordinates to identify exponentially decaying solutions to the normal form \eqref{Ntup:NormalForm} and trace them back to $r = r_0$ to be matched with the core manifold. We begin by augmenting system \eqref{Ntup:NormalForm} by introducing $\kappa$ so that $\mu=\kappa^2$, giving
\begin{equation}\label{Ntup:Normal:Ext}
	\begin{split}
    \frac{\textnormal{d}}{\textnormal{d}r} \mathbf{A} &= - \frac{\sigma}{2} \mathbf{A} + \mathbf{B} + \mathbf{R}_{\mathbf{A}}(\mathbf{A}, \mathbf{B},\sigma,\kappa),\\
    \frac{\textnormal{d}}{\textnormal{d}r} \mathbf{B} &= -\frac{\sigma}{2} \mathbf{B} + c_{0}\,\kappa^2 \mathbf{A} + c_3 \mathbf{C}^{m}_{N}(\mathbf{A}) + \mathbf{R}_{\mathbf{B}}(\mathbf{A}, \mathbf{B},\sigma,\kappa),\\
    \frac{\textnormal{d}}{\textnormal{d} r} \sigma &= -\sigma^{2},\\
    \frac{\textnormal{d}}{\textnormal{d} r} \kappa &= 0.
	\end{split}
\end{equation}
The expansions in Lemma~\ref{Lemma:Ntup;normal} give that 
\begin{equation}
     [\mathbf{R}_{\mathbf{A}/\mathbf{B}}]_{n} = \mathcal{O}([|\kappa|^{4} + |\sigma|^{3} + (|\mathbf{A}| + |\mathbf{B}|)^{2}]|\mathbf{B}| + [|\kappa|^{4} + |\sigma|^{3} + |\mathbf{A}|^{3}]|\mathbf{A}|).
\end{equation}
We introduce the following blow-up variables
\begin{equation}\label{Ntup:resc}
    	\mathbf{A}_{R} :=  \kappa^{-1}\mathbf{A}, \quad \mathbf{B}_{R} := \kappa^{-2}\mathbf{B},\quad \sigma_{R} :=  \kappa^{-1}\sigma, \quad \kappa_{R} := \kappa, \quad s := \kappa r.
\end{equation}
into the augmented system \eqref{Ntup:Normal:Ext}, to obtain the rescaling chart:
\begin{equation}\label{Ntup:Normal;Resc}
	\begin{split}
    \frac{\textnormal{d}}{\textnormal{d}s} \mathbf{A}_{R} &= - \frac{\sigma_{R}}{2}\mathbf{A}_{R} +\mathbf{B}_{R} +\mathbf{R}_{\mathbf{A}, R}(\mathbf{A}_{R}, \mathbf{B}_{R},\sigma_{R},\kappa_{R}),\\
    \frac{\textnormal{d}}{\textnormal{d}s} \mathbf{B}_{R} &= -\frac{\sigma_{R}}{2} \mathbf{B}_{R} + c_{0}\,\mathbf{A}_{R} + c_3 \mathbf{C}^{m}_{N}(\mathbf{A}_{R}) + \mathbf{R}_{\mathbf{B}, R}(\mathbf{A}_{R}, \mathbf{B}_{R},\sigma_{R},\kappa_{R}),\\
    \frac{\textnormal{d}}{\textnormal{d}s} \sigma_{R} &= -\sigma_{R}^{2},\\
    \frac{\textnormal{d}}{\textnormal{d}s} \kappa_{R} &= 0,\\
	\end{split}
\end{equation}
where we write
\begin{equation}\begin{split}
     \mathbf{R}_{\mathbf{A}, R}(\mathbf{A}_{R}, \mathbf{B}_{R},\sigma_{R},\kappa_{R}) &= \kappa_{R}^{-2}\mathbf{R}_{\mathbf{A}}(\kappa_{R}\mathbf{A}_{R}, \kappa_{R}^{2}\mathbf{B}_{R},\kappa_{R}\sigma_{R},\kappa_{R}),\\
     \mathbf{R}_{\mathbf{B}, R}(\mathbf{A}_{R}, \mathbf{B}_{R},\sigma_{R},\kappa_{R}) &= \kappa_{R}^{-3}\mathbf{R}_{\mathbf{B}}(\kappa_{R}\mathbf{A}_{R}, \kappa_{R}^{2}\mathbf{B}_{R},\kappa_{R}\sigma_{R},\kappa_{R}),
\end{split}\end{equation}
with
\begin{equation}\begin{split}
     [\mathbf{R}_{\mathbf{A},R}]_{n} &= |\kappa_{R}|^{2}\mathcal{O}([|\kappa_{R}|^{2} + |\kappa_{R}||\sigma_{R}|^{3} + (|\mathbf{A}_{R}| + |\kappa_{R}||\mathbf{B}_{R}|)^{2}]|\mathbf{B}_{R}| + [|\kappa_{R}| + |\sigma_{R}|^{3} + |\mathbf{A}_{R}|^{3}]|\mathbf{A}_{R}|),\\
     [\mathbf{R}_{\mathbf{B},R}]_{n} &= |\kappa_{R}|\mathcal{O}([|\kappa_{R}|^{2} + |\kappa_{R}||\sigma_{R}|^{3} + (|\mathbf{A}_{R}| + |\kappa_{R}||\mathbf{B}_{R}|)^{2}]|\mathbf{B}_{R}| + [|\kappa_{R}| + |\sigma_{R}|^{3} + |\mathbf{A}_{R}|^{3}]|\mathbf{A}_{R}|),\\
\end{split}\end{equation}
for all $n \in [0,N]$.

Thus far the work in this section has greatly resembled the work in our previous investigation \cite{hill2022approximate}. At this point in our analysis we see a major divergence. In particular, as discussed in the lead up to Theorem~\ref{thm:Ring}, we will require the following result due to van den Berg, Groothedde, and Williams on the existence of a nontrivial solution to the Ginzburg--Landau equation
\begin{equation}\label{Ginzb:Eqn}
    \left(\tfrac{\mathrm{d}}{\mathrm{d} s} + \tfrac{1}{2s}\right)^{2}q(s) = c_0 q(s) + c_3 q(s)^3.
\end{equation}
After precisely stating this result we then proceed with our work on the rescaling chart. Throughout the remainder of this section the appearance of $q_0 > 0$ and $q(\cdot)$ relate to those given in the following lemma. 

\begin{lem}[\cite{vandenberg2015Rigorous}, Theorem~1.1]\label{Lem:Ginzb}
For $c_3<0$, \eqref{Ginzb:Eqn} has a bounded nontrivial solution $q(s)$, and there are constants $q_0>0$ and $q_+\neq0$ so that
\begin{equation}\label{q:defn}
    q(s) = \begin{cases}
         q_0 s^{\frac{1}{2}} + \mathcal{O}(s^{\frac{3}{2}}), & s\to 0,\\
         (q_+ + \mathcal{O}(\mathrm{e}^{-\sqrt{c_0}s}))s^{-\frac{1}{2}}\mathrm{e}^{-\sqrt{c_0}s}, & s\to \infty.
    \end{cases}
\end{equation}
In addition, the linearisation of \eqref{Ginzb:Eqn} about $q(s)$
\begin{equation}\label{Ginzb:Eqn;lin}
    \left[\left(\tfrac{\mathrm{d}}{\mathrm{d} s} + \tfrac{1}{2s}\right)^{2} - c_0 - 3 c_3 q(s)^2\right]p(s) = 0,
\end{equation}
does not have a nontrivial solution that is bounded uniformly on $\mathbb{R}^{+}$. If $c_3>0$, then the only bounded solution of \eqref{Ginzb:Eqn} on $\mathbb{R}^{+}$ is $q(s)\equiv0$.
\end{lem}

\begin{lem}\label{Lem:Resc;Evo} %Lemma: 
Fix $c_3<0$. For each fixed choice of $r_{1}>0$, there is a constant $\kappa_{0}>0$ such that the set of exponentially decaying solutions 
to \eqref{Ntup:Normal;Resc} for $s\in[r_{1},\infty)$, evaluated at $s=r_{1}$, is given by
\begin{equation}
(\mathbf{A}_{R}, \mathbf{B}_{R}, \sigma_{R}, \kappa_{R})(r_{1}) = \bigg( q_0 r_{1}^{\frac{1}{2}} \mathbf{a}\,\textnormal{e}^{\textnormal{i}Y}(1 + \mathcal{O}(\mu^{\frac{1}{2}})), \tfrac{1}{2} q_0 r_{1}^{-\frac{1}{2}} \mathbf{a}\,\textnormal{e}^{\textnormal{i}Y}(1 + \mathcal{O}(\mu^{\frac{1}{2}})),r_{1}^{-1}, \mu^{\frac{1}{2}}\bigg),   
\end{equation}
for all $\mu<\kappa_{0}$, where $Y\in\mathbb{R}$ is arbitrary, and $\mathbf{a}\in\mathbb{R}^{(N+1)}$, $\mathbf{a}\neq\mathbf{0}$, is a solution of the fixed-point equation
\begin{equation}\label{Eqn:CN}
    \mathbf{a} = \mathbf{C}^{m}_{N}(\mathbf{a}).
\end{equation}
\end{lem}

\begin{proof}
We first note that $\kappa_{R}$ acts as a parameter, and $\Gamma_{c}:=\{(\sigma_{R},\kappa_{R}) = (\tfrac{1}{s},\sqrt{c})\}$ is an invariant subspace of \eqref{Ntup:Normal;Resc} for any fixed $c\geq0$. Evaluating \eqref{Ntup:Normal;Resc} on $\Gamma_{0}$, we arrive at the complex cubic non-autonomous Ginzburg-Landau system $\mathcal{T}[\mathbf{A}_{R}(s)]=\mathbf{0}$, where we have defined
\begin{equation}\label{Ginzb:Eqn;A}
    \mathcal{T}[\mathbf{z}(s)] :=\left(\tfrac{\mathrm{d}}{\mathrm{d} s} + \tfrac{1}{2s}\right)^{2}\mathbf{z} - c_0 \mathbf{z} - c_3 \mathbf{C}^{m}_{N}(\mathbf{z}), \qquad \qquad \forall \mathbf{z}(s)\in C^2(\mathbb{R},\mathbb{C}^{N+1}).
\end{equation}
We first assume that there exists some $\mathbf{a}\in\mathbb{R}^{N+1}$, $\mathbf{a}\neq\mathbf{0}$, that solves the cubic fixed-point equation \eqref{Eqn:CN}. We note that for any function $f(s):\mathbb{R}^{+}\to\mathbb{R}$, there exists some function $g(s):\mathbb{R}^{+}\to\mathbb{R}$ such that $\mathcal{T}[f(s)\mathbf{a}] = g(s)\mathbf{a}$. Hence, we conclude that $\Lambda:=\{(\mathbf{A}_{R},\mathbf{B}_{R},\sigma_{R},\kappa_{R}) = (A(s)\mathbf{a},B(s)\mathbf{a},\tfrac{1}{s},0)\}$ is an invariant subspace of \eqref{Ntup:Normal;Resc}. We begin by finding a solution to \eqref{Ntup:Normal;Resc} in the subspace $\Lambda$ such that $|\mathbf{A}_{R}(s)|$ decays exponentially fast as $s\to\infty$. For $\mathbf{a}\neq\mathbf{0}$, a solution of the equation $\mathcal{T}[A(s)\mathbf{a}] = \mathbf{0}$ must satisfy
\begin{equation}
    \left(\tfrac{\mathrm{d}}{\mathrm{d}s} + \tfrac{1}{2s}\right)^{2}A(s) = c_0 A(s) + c_3 A(s)^3.
\end{equation}
Hence, Lemma~\ref{Lem:Ginzb} tells us that for $c_3<0$ the equation $\mathcal{T}[A(s)\mathbf{a}]=\mathbf{0}$ has a solution $A(s) = q(s)$, where $q(s)$ has the asymptotic behaviour defined in \eqref{q:defn}. Furthermore, if $c_3>0$, then the only bounded solution of $\mathcal{T}[A(s)\mathbf{a}]=\mathbf{0}$ is $A(s)\equiv0$. In order for our core and far-field manifolds to form a transverse intersection, we require that $\mathcal{T}[A(s)\mathbf{a}]=\mathbf{0}$ linearised about our solution $A(s) = q(s)$ does not have a nontrivial bounded solution; see \cite[Lemma 2.4]{mccalla2013spots}. In the subspace $\Lambda$, the linear problem takes the following form
\begin{equation}\label{Lin:Prob}
    \mathcal{L}[q(s)\mathbf{a}]p(s)\mathbf{a} := \left[\left(\tfrac{\mathrm{d}}{\mathrm{d}s} + \tfrac{1}{2s}\right)^{2} - c_0 \right]p(s)\mathbf{a} - c_3 q(s)^2 p(s) D\mathbf{C}_{N}^{m}(\mathbf{a})\mathbf{a} = \mathbf{0},
\end{equation}
where $D\mathbf{C}_{N}^{m}(\mathbf{a})$ denotes the Jacobian of $\mathbf{C}_{N}^{m}$ evaluated at the point $\mathbf{a}\in\mathbb{R}^{N+1}$. 
Lemma~\ref{Lem:3Eigenvec} in the appendix tells us that any $\mathbf{a}\in\mathbb{R}^{N+1}\backslash\{\mathbf{0}\}$ that satisfies \eqref{Eqn:CN} is an eigenvector of $D\mathbf{C}_{N}^{m}(\mathbf{a})$, such that 
\begin{equation}
    D\mathbf{C}_{N}^{m}(\mathbf{a})\mathbf{a} = 3\mathbf{a},
\end{equation}
and so the linear problem reduces to 
\begin{equation}
    \left[\left(\tfrac{\mathrm{d}}{\mathrm{d}s} + \tfrac{1}{2s}\right)^{2} - c_0 - 3 c_3 q(s)^2\right]p(s)\mathbf{a} = 0.
\end{equation}
Then, by Lemma~\ref{Lem:Ginzb}, we see that there are no bounded nontrivial solutions of the linear problem \eqref{Lin:Prob}. Hence, we have found an exponentially decaying solution 
\begin{equation}
    \left(\mathbf{A}_{R},\mathbf{B}_{R},\sigma_{R},\kappa_{R}\right) = \left(q(s) \mathbf{a}, \left(\tfrac{\mathrm{d}}{\mathrm{d}s} + \tfrac{1}{2s}\right)q(s) \mathbf{a}, \tfrac{1}{s}, 0\right)\in\Lambda\subset[\mathbb{R}^{N+1}]^2\times\Gamma_{0},
\end{equation}
which forms a transverse connecting orbit between the core and far-field manifolds. We note that the $S^{1}$-symmetry of the equations \eqref{Ntup:Normal;Resc} evaluated on $\Lambda$ allows us to parametrise the complex phase of $(\mathbf{A}_{R},\mathbf{B}_{R})$ by a single parameter $Y\in\mathbb{R}$. That is, $\mathcal{T}[\mathrm{e}^{\mathrm{i}Y}\mathbf{A}_{R}] = \mathrm{e}^{\mathrm{i}Y}\mathcal{T}[\mathbf{A}_{R}]$ and so if $\mathbf{A}_{R}(s)\in\mathbb{C}^{N+1}$ is a solution of $\mathcal{T}[\mathbf{A}_{R}]=\mathbf{0}$, then so is $\mathrm{e}^{\mathrm{i}Y}\mathbf{A}_{R}(s)$. Furthermore, any exponentially decaying solution with $\kappa_{R} \equiv \sqrt{\varepsilon}\ll1$ will remain sufficiently close to the invariant subspace $\Lambda$, since the system \eqref{Ntup:Normal;Resc} evaluated on $\Gamma_{\varepsilon}$ is a regular perturbation of \eqref{Ntup:Normal;Resc} on $\Gamma_{0}$. In conclusion, setting $\kappa_{R}(s)\equiv \mu^{\frac{1}{2}}$ and introducing a phase parameter $Y\in\mathbb{R}$, we can express exponentially decaying solutions to \eqref{Ntup:Normal;Resc} in the following form
\begin{equation}\label{Sol:Resc}
    \left(\mathbf{A}_{R},\mathbf{B}_{R},\sigma_{R},\kappa_{R}\right) = \left(q(s)\mathbf{a}\,\mathrm{e}^{\mathrm{i}Y}(1 + \mathcal{O}(\mu^{\frac{1}{2}})),\left(\tfrac{\mathrm{d}}{\mathrm{d}s} + \tfrac{1}{2s}\right)q(s)\mathbf{a}\,\mathrm{e}^{\mathrm{i}Y}(1 + \mathcal{O}(\mu^{\frac{1}{2}})),\tfrac{1}{s},\mu^{\frac{1}{2}}\right),
\end{equation}
where the $\mathcal{O}(\mu^{\frac{1}{2}})$ terms are due to the $\mathcal{O}(|\kappa_{R}|)$ terms in \eqref{Ntup:Normal;Resc}. Evaluating \eqref{Sol:Resc} at $s=r_1\ll1$ and applying the asymptotic properties of $q(s)$ from \eqref{q:defn}, we arrive at our final result.
\end{proof} %End of proof

Undoing the rescaling transformation \eqref{Ntup:resc} gives exponentially decaying solution of \eqref{Ntup:Normal:Ext}. Precisely, evaluating at $r = \mu^{-\frac{1}{2}}r_1$ gives
\begin{equation}\label{Ntup:Sol;delta0} 
    (\mathbf{A}, \mathbf{B}, \sigma, \kappa)(\mu^{-\frac{1}{2}}r_1) = \bigg(q_0 \mu^{\frac{1}{2}} r_{1}^{\frac{1}{2}} \mathbf{a}\,\textnormal{e}^{\textnormal{i}Y}(1 + \mathcal{O}(\mu^{\frac{1}{2}})), q_0 \mu r_{1}^{-\frac{1}{2}} \mathbf{a}\,\textnormal{e}^{\textnormal{i}Y}(1 + \mathcal{O}(\mu^{\frac{1}{2}})),\mu^{\frac{1}{2}}r_{1}^{-1}, \mu^{\frac{1}{2}}\bigg),  
\end{equation}
for all $0 < \mu \ll 1$. Thus, we have parametrised the far-field manifold $\mathcal{W}^{s}_{+}(\mu)$ back to a transition point $r=\mu^{-\frac{1}{2}}r_1$. We now wish to extend these solutions back to some $r = r_0 < \mu^{-\frac{1}{2}}r_1$ to match these far-field solutions with those on the core manifold $\mathcal{W}^{cu}_{-}(\mu)$. This extension backwards in $r$ constitutes the transition chart, representing a transition region $r_0 \leq r \leq \mu^{-\frac{1}{2}}r_1$ which connects the core to the far-field. 

In order to characterise the transition region, we aim to solve the initial value problem
\begin{align}
    \frac{\textnormal{d}}{\textnormal{d} r}\mathbf{A} &= -\frac{\sigma}{2}\mathbf{A} + \mathbf{B} + \mathbf{R}_{\mathbf{A}}\left(\mathbf{A},\mathbf{B};\sigma,\kappa\right), & \mathbf{A}\left(\mu^{-\frac{1}{2}}r_1\right)&= \mu^{\frac{1}{2}}r_{1}^{\frac{1}{2}}\mathbf{A}_{0},\nonumber\\
    \frac{\textnormal{d}}{\textnormal{d} r} \mathbf{B} &= -\frac{\sigma}{2}\mathbf{B} + c_{0}\,\kappa^{2}\mathbf{A} + c_{3}\,\mathbf{C}^{m}_{N}(\mathbf{A}) + \mathbf{R}_{\mathbf{B}}\left(\mathbf{A},\mathbf{B};\sigma,\kappa\right), & 
    \mathbf{B}\left(\mu^{-\frac{1}{2}}r_1\right)&= \mu r_{1}^{-\frac{1}{2}}\mathbf{A}_{0},\nonumber\\
    \frac{\textnormal{d}}{\textnormal{d} r}\sigma &= -\sigma^{2}, & \sigma\left(\mu^{-\frac{1}{2}}r_1\right) &= \mu^{\frac{1}{2}}r_{1}^{-1},\label{Ntup:NormalForm;Initial}\\
    \frac{\textnormal{d}}{\textnormal{d} r}\kappa &= 0, & \kappa\left(\mu^{-\frac{1}{2}}r_1\right)&= \mu^{\frac{1}{2}},\nonumber
\end{align}
where $\mathbf{A}_{0}:=q_0\mathbf{a}\,\textnormal{e}^{\textnormal{i}Y}(1 + \mathcal{O}(\mu^{\frac{1}{2}}))$. In particular, we note that we are solving backward in $r$ with the initial condition prescribed at $r = \mu^{-\frac{1}{2}}r_1$ and the solution flowing backward to $r = r_0$. We present the following result.

\begin{lem} %Lemma: Transition chart
For each fixed choice of $0<r_{1}, r_{0}^{-1}\ll1$, there is a $\kappa_{0}>0$ such that solutions of the initial value problem \eqref{Ntup:NormalForm;Initial}, evaluated at $r=r_{0}$, are given by
\begin{equation}\label{Ntup:Farf;r0}
	\begin{split}
   		\mathbf{A}(r_{0}) &= q_0\mu^{\frac{3}{4}} r_{0}^{\frac{1}{2}} \mathbf{a}\,\textnormal{e}^{\textnormal{i}Y}(1 + \xi)\\
      		\mathbf{B}(r_{0}) &= q_0\mu^{\frac{3}{4}} r_{0}^{-\frac{1}{2}} \mathbf{a}\,\textnormal{e}^{\textnormal{i}Y}(1 + \xi) \\
		\sigma(r_{0}) &= r_{0}^{-1} \\
		\kappa(r_{0}) &=\mu^{\frac{1}{2}} \\
	\end{split}
\end{equation}
for all $\mu<\kappa_{0}$, where $\xi:= \mathcal{O}\left(\mu^{\frac{1}{2}} + r_{1} + r_{0}^{-1}\right)$, and $\mathbf{a}\in\mathbb{R}^{(N+1)}$ and $Y\in\mathbb{R}$ were introduced in Lemma \ref{Lem:Resc;Evo}.
\label{Lem:Transition;Evo}
\end{lem}
\begin{proof}
The proof follows the standard approach for computing leading-order solutions of an initial value problem, as seen in \cite[Proof of Lemma~4.7]{hill2022approximate}. We first solve \eqref{Ntup:NormalForm;Initial} for $\kappa(r)$ and $\sigma(r)$ explicitly, giving $\kappa=\mu^{\frac{1}{2}}$ and $\sigma=\tfrac{1}{r}$, and define the following transition coordinates
\begin{align}
    \mathbf{A}_{T} := r^{-\frac{1}{2}}\mathbf{A}, \qquad \mathbf{B}_{T} := r^{\frac{1}{2}}\mathbf{B}. \label{Ntup:Transition}
\end{align}
Then, integrating the remaining equations in \eqref{Ntup:NormalForm;Initial} over $r_{0}\leq r\leq \mu^{-\frac{1}{2}}r_1$, we obtain the integral equation
\begin{equation}\label{Ntup:Transition;Initial}
	\begin{split}
    		\mathbf{A}_{T}(r) &= \mu^{\frac{3}{4}}\mathbf{A}_{0} +  \int^{r}_{\mu^{-\frac{1}{2}}r_1}\left\{p^{-1}(\mathbf{B}_{T}-\mathbf{A}_{T}) + p^{-\frac{1}{2}}\mathbf{R}_{\mathbf{A}}\left(p^{\frac{1}{2}}\mathbf{A}_{T},p^{-\frac{1}{2}}\mathbf{B}_{T};\tfrac{1}{p},\mu^{\frac{1}{2}}\right)\right\}\textnormal{d}p, \\
    		\mathbf{B}_{T}(r) &= \mu^{\frac{3}{4}}\mathbf{A}_{0} + \int_{\mu^{-\frac{1}{2}}r_1}^{r}\left\{ c_{0}\,\mu p \mathbf{A}_{T} + c_3 p^2 \mathbf{C}^{m}_{N}(\mathbf{A}_{T}) + p^{\frac{1}{2}}\mathbf{R}_{\mathbf{B}}\left(p^{\frac{1}{2}}\mathbf{A}_{T},p^{-\frac{1}{2}}\mathbf{B}_{T};\tfrac{1}{p},\mu^{\frac{1}{2}}\right)\right\}\textnormal{d}p.
	\end{split}
\end{equation}
For sufficiently small values of $\mu$, we apply the contraction mapping principle to show that \eqref{Ntup:Transition;Initial} has a unique solution $(\mathbf{A}_{T},\mathbf{B}_{T})$ within a neighbourhood of the origin in $C([r_{0},\mu^{-\frac{1}{2}}r_1], \mathbb{C}^{2(N+1)})$; for an example, see \cite{Sandstede1997Convergence}. Furthermore, we can express the unique solution to \eqref{Ntup:Transition;Initial}, evaluated at $r=r_{0}$, as
\begin{equation}
    \mathbf{A}_{T}(r_{0}) = q_0\mu^{\frac{3}{4}} \mathbf{a}\,\textnormal{e}^{\textnormal{i}Y}(1 + \mathcal{O}(\mu^{\frac{1}{2}} + r_{1} + r_{0}^{-1})), \quad \mathbf{B}_{T}(r_{0}) = q_0\mu^{\frac{3}{4}} \mathbf{a}\,\textnormal{e}^{\textnormal{i}Y}(1 + \mathcal{O}(\mu^{\frac{1}{2}} + r_{1} + r_{0}^{-1})).
\end{equation}
Inverting the transition transformation \eqref{Ntup:Transition} by multiplying the above expressions by $r_0^{-\frac{1}{2}}$ brings us to the desired result.
\end{proof} %End of proof

The previous lemma thus parametrised the far-field manifold $\mathcal{W}^{s}_{+}(\mu)$ all the way back to the matching point $r = r_0$ when $0 < \mu \ll 1$. This allows us to match the core solution with that of the far-field at $r = r_0$, as will be undertaken in the following subsection.

%%%%%%%%%%%%%%%%%%%%%%%%%%%%%%%%%%%%%%%%%%%%%%%%%%%%%%%%%%%%%%%%%%%%%%%%%%%%
\subsection{Matching Core and Far Field}\label{subsec:CoreFarMatch}

To match the core and far field equations, we require them to be written in the same coordinates. Therefore, we first have that under the linear change of variable \eqref{R-D:Transformation;Farfield}, the core manifold \eqref{Core:un} becomes 
\begin{equation}
    \begin{split}
    \widetilde{A}_{n}(r_{0}) &= \frac{1}{2} r_{0}^{-\frac{1}{2}}\textnormal{e}^{\textnormal{i}y_{n}}\left[d_{1}^{(n)}\left[1 + \mathcal{O}\left(r_{0}^{-1}\right)\right] -\textnormal{i} r_{0} d_{2}^{(n)}\left[1 + \mathcal{O}\left(r_{0}^{-1}\right)\right] + \mathcal{O}_{r_{0}}(|\mu||\mathbf{d}| + |\mathbf{d}|^{2})\right],\\
    \widetilde{B}_{n}(r_{0}) &= -\frac{1}{2} r_{0}^{-\frac{1}{2}}\textnormal{e}^{\textnormal{i}y_{n}}\left[[\nu + \mathcal{O}(r_{0}^{-\frac{1}{2}})]Q_{n}^{m}(\mathbf{d}_{1})  + \textnormal{i} d_{2}^{(n)}\left[1 + \mathcal{O}\left(r_{0}^{-1}\right)\right] + \mathcal{O}_{r_{0}}(|\mu||\mathbf{d}| + |\mathbf{d}_{2}|^{2} + |\mathbf{d}_{1}|^{3})\right],
    \end{split}\label{Core:An}
\end{equation}
for each $n\in[0,N]$. Recall from Lemma~\ref{Lemma:Ntup;Core} that $\nu = \frac{1}{2}\sqrt{\frac{\pi}{6}} \langle \hat{U}_{1}^{*}, \mathbf{Q}(\hat{U}_{0},\hat{U}_{0})\rangle_{2}$ and the value $y_n = r_0 - \frac{mn\pi}{2} - \frac{\pi}{4}$ comes from equation \eqref{Core:un}.

Applying the nonlinear normal form change of variable \eqref{Ntup:Normal;transf} to \eqref{Core:An} then results in the system 
\begin{equation}\label{Core:Am;param}
	\begin{split}
    		\mathbf{A}(r_{0}) &= \frac{1}{2}r_{0}^{-\frac{1}{2}}\;\textnormal{e}^{\textnormal{i}Y_{0}} \left[ \mathbf{d}_{1}(1 + \mathcal{O}(r_{0}^{-1})) -\textnormal{i} r_{0} \mathbf{d}_{2}(1 + \mathcal{O}(r_{0}^{-1})) + \mathcal{R}_{\mathbf{A}}(\mathbf{d}_{1},\mathbf{d}_{2},\mu^{\frac{1}{2}})\right], \\
    		\mathbf{B}(r_{0}) &= -\frac{1}{2}r_{0}^{-\frac{1}{2}}\;\textnormal{e}^{\textnormal{i}Y_{0}} \left[\mathbf{Q}_N^m(\mathbf{d}_{1})(\nu + \mathcal{O}(r_{0}^{-\frac{1}{2}})) + \textnormal{i}\mathbf{d}_{2}(1 + \mathcal{O}(r_{0}^{-1})) + \mathcal{R}_{\mathbf{B}}(\mathbf{d}_{1},\mathbf{d}_{2},\mu^{\frac{1}{2}})\right],
	\end{split}
\end{equation}
with $\mathbf{Q}_N^m:=\left(Q_{0}^m, \dots, Q_{N}^m\right)\in\mathbb{R}^{(N+1)}$ for $Q^{m}_{n}$ defined in \eqref{Psi:d1;defn}, $Y_{0}:=- \frac{\pi}{4} + \mathcal{O}(|\mu|r_{0} + r_{0}^{-2})$, and
\begin{equation}\label{R;AB;Defn}
     [\mathcal{R}_{\mathbf{A}}]_{n}  = \mathcal{O}_{r_{0}}(|\mu||\mathbf{d}| + |\mathbf{d}|^{2}), \quad  [\mathcal{R}_{\mathbf{B}}]_{n} = \mathcal{O}_{r_{0}}(|\mu||\mathbf{d}| + |\mathbf{d}_{2}|^{2} + |\mathbf{d}_{1}|^{3}),
\end{equation}
for all $n\in[0,N]$. Equating the far-field parametrisation \eqref{Ntup:Farf;r0} and the core parametrisation results in
\begin{equation}\label{Ntup:match;corefar}
	\begin{split}
    		2\mu^{\frac{3}{4}}\, q_0 r_0 \textnormal{e}^{\textnormal{i}Y}\mathbf{a}\left(1 + \xi\right) &= \textnormal{e}^{\textnormal{i}Y_{0}} \left[ \mathbf{d}_{1}\left(1 + \mathcal{O}(r_{0}^{-1})\right) -\textnormal{i} r_{0} \mathbf{d}_{2}\left(1 + \mathcal{O}(r_{0}^{-1})\right) + \mathcal{R}_{\mathbf{A}}\right],\\
    		2 \mu^{\frac{3}{4}}\,  \textnormal{e}^{\textnormal{i}Y}q_0\mathbf{a}\left(1 + \xi\right) &= -\textnormal{e}^{\textnormal{i}Y_{0}} \left[\mathbf{Q}_N^m(\mathbf{d}_{1})\left(\nu + \mathcal{O}(r_{0}^{-\frac{1}{2}})\right) + \textnormal{i}\mathbf{d}_{2}\left(1 + \mathcal{O}(r_{0}^{-1})\right) + \mathcal{R}_{\mathbf{B}}\right]. 
	\end{split}
\end{equation}
Now, let us introduce the rescalings 
\begin{equation}\label{match:scaling}
    	\mathbf{d}_{1} = 2\mu^{\frac{3}{4}}q_0 r_0\widetilde{\mathbf{d}}_{1},\qquad \mathbf{d}_{2} = 2\mu^{\frac{3}{4}}q_0\widetilde{\mathbf{d}}_{2},\qquad Y = Y_{0} + \widetilde{Y}.
\end{equation}
Putting this into \eqref{Ntup:match;corefar} and dividing off the leading orders in $\mu$ results in the matching equations 
\begin{equation}\label{Ntup:match;corefar,scale}
	\begin{split}
    		\textnormal{e}^{\textnormal{i}\widetilde{Y}}\mathbf{a}\left(1 + \xi\right) &= \widetilde{\mathbf{d}}_{1}\left(1 + \mathcal{O}(r_{0}^{-1})\right) -\textnormal{i} \widetilde{\mathbf{d}}_{2}\left(1 + \mathcal{O}(r_{0}^{-1})\right) + \widetilde{\mathcal{R}}_{\mathbf{A}},\\
    		\textnormal{e}^{\textnormal{i}\widetilde{Y}}\mathbf{a}\left(1 + \xi\right) &= -\textnormal{i}\widetilde{\mathbf{d}}_{2}\left(1 + \mathcal{O}(r_{0}^{-1})\right) + \widetilde{\mathcal{R}}_{\mathbf{B}}. 
	\end{split}
\end{equation}
where
\begin{equation}
     [\widetilde{\mathcal{R}}_{\mathbf{A}}]_{n}  = |\mu|^{\frac{3}{4}}\mathcal{O}_{r_{0}}(|\mu|^{\frac{1}{4}}|\widetilde{\mathbf{d}}| + |\widetilde{\mathbf{d}}|^{2}), \quad  [\widetilde{\mathcal{R}}_{\mathbf{B}}]_{n} = |\mu|^{\frac{3}{4}}\mathcal{O}_{r_{0}}(|\mu|^{\frac{1}{4}}|\widetilde{\mathbf{d}}| + |\widetilde{\mathbf{d}}|^{2}),
\end{equation}
for all $n\in[0,N]$. In the parameter regimes $0 < r_0^{-1},r_1,\mu \ll 1$ system \eqref{Ntup:match;corefar,scale} is well-approximated by the system
\begin{equation}\label{Ntup:match;corefar,scale0}
	\begin{split}
    		\textnormal{e}^{\textnormal{i}\widetilde{Y}}\mathbf{a} &= \widetilde{\mathbf{d}}_{1} -\textnormal{i} \widetilde{\mathbf{d}}_{2},\\
    		\textnormal{e}^{\textnormal{i}\widetilde{Y}}\mathbf{a} &= -\textnormal{i}\widetilde{\mathbf{d}}_{2}. 
	\end{split}
\end{equation}
Splitting into real and imaginary parts, solutions of \eqref{Ntup:match;corefar,scale0} can be equivalently determined as the zeros of the functional
\begin{equation}\label{Ntup:G;defn}
    	\mathcal{G}:\mathbb{R}^{2(N+1)+1}\to\mathbb{R}^{2(N+1)+1}, \quad (\widetilde{\mathbf{d}}_{1},\widetilde{\mathbf{d}}_{2},{\widetilde{Y}})^{T}\mapsto(\mathbf{G}_1,\mathbf{G}_2,G_3)^{T},\\
\end{equation}
where
\begin{equation}\label{Ntup:G;expl}
	\begin{split}
    		\mathbf{G}_1 &= \widetilde{\mathbf{d}}_{2} + \sin(\widetilde{Y})\mathbf{a}, \qquad 
    		\mathbf{G}_2 = \widetilde{\mathbf{d}}_{1} -\cos(\widetilde{Y})\mathbf{a},\qquad G_3 = -\cos(\widetilde{Y}). 
	\end{split}
\end{equation}
The imaginary parts of the first and second equations in \eqref{Ntup:match;corefar,scale0} are identical, and so both are captured by the first component of $\mathcal{G}$ in \eqref{Ntup:G;expl}. It should be noted that taking the real part of the second equation in \eqref{Ntup:match;corefar,scale0} results in the vector equation $\cos(\widetilde{Y})\mathbf{a}=\mathbf{0}$. However, for $\mathbf{a}\neq\mathbf{0}$, this reduces to obtaining $\widetilde{Y}$ such that $\cos(\widetilde{Y}) = 0$, which is captured by the final component of $\mathcal{G}$, defined in \eqref{Ntup:G;expl}. 

\begin{lem}\label{Match:am;gen} %Lemma: Roots of G function
Fix $m,N\in\mathbb{N}$. Then, the functional $\mathcal{G}$, defined in \eqref{Ntup:G;expl}, has zeros of the form
\begin{equation}\label{Ntup:Match;Soln}
    \mathcal{V}^{*}:=(\widetilde{\mathbf{d}}_{1},\widetilde{\mathbf{d}}_{2},\widetilde{Y}) = (\mathbf{0}, (-1)^{\alpha+1}\mathbf{a}, \tfrac{(2\alpha+1)\pi}{2}),
\end{equation}
where $\alpha\in\mathbb{Z}$.
\end{lem}

\begin{proof} 
From the definition \eqref{Ntup:G;defn}, roots of $\mathcal{G}$ correspond to having $\mathbf{G}_1 = \mathbf{0}$, $\mathbf{G}_2 = \mathbf{0}$, and $G_3 = \mathbf{0}$, as they are given in \eqref{Ntup:G;expl}. Then, we immediately find that solving $G_3 = 0$ gives that $\widetilde{Y} = \tfrac{(2\alpha+1)\pi}{2}$, for any $\alpha\in\mathbb{Z}$. Substituting $\widetilde{Y}=\tfrac{(2\alpha+1)\pi}{2}$ into $\mathbf{G}_2$, we find that $\widetilde{\mathbf{d}}_{1} = \mathbf{0}$ is the unique choice that solves $\mathbf{G}_2 = \mathbf{0}$. Then, solving $\mathbf{G}_1=\mathbf{0}$ results in $\widetilde{\mathbf{d}}_{2}=(-1)^{\alpha+1}\mathbf{a}$. 
\end{proof} %End of proof

To match solutions from the core manifold to those of the far-field manifold, { we further require that roots of $\mathcal{G}$ be non-degenerate}. { The reason for this is that roots} of $\mathcal{G}$ only capture solutions of the leading order matching equations \eqref{Ntup:match;corefar,scale0}. So, in order to extend these solutions to the leading order matching equations to the full matching equations \eqref{Ntup:match;corefar,scale} with the implicit function theorem, we require that the Jacobian of $\mathcal{G}$, denoted $D\mathcal{G}$, evaluated at these roots be invertible. We present the following lemma to demonstrate that solutions $\mathcal{V}^*$ as in Lemma~\ref{Match:am;gen} always result in $D\mathcal{G}(\mathcal{V}^*)$ being invertible.  

\begin{lem}\label{lem:NondegenMatch} %Lemma: non-degeneracy condition on G
Fix $m,N\in\mathbb{N}$. Then, the function $\mathcal{G}$ defined in \eqref{Ntup:G;expl} and a root $\mathcal{V}^{*}$ as in \eqref{Ntup:Match;Soln} have the property that 
\begin{equation}\label{det:G}
    \det(D\mathcal{G}[\mathcal{V}^{*}]) = (-1)^{\alpha+N}.
\end{equation}
\end{lem}

\begin{proof}
The Jacobian $D\mathcal{G}$ can be written as
\begin{equation}
    D\mathcal{G} = \begin{pmatrix}
    D_{\widetilde{\mathbf{d}}_{1}}\mathbf{G}_1 & D_{\widetilde{\mathbf{d}}_{2}}\mathbf{G}_1 & D_{\widetilde{Y}}\mathbf{G}_1 \\
    D_{\widetilde{\mathbf{d}}_{1}}\mathbf{G}_2 &  D_{\widetilde{\mathbf{d}}_{2}}\mathbf{G}_2 & D_{\widetilde{Y}}\mathbf{G}_2 \\
    D_{\widetilde{\mathbf{d}}_{1}}G_3 & D_{\widetilde{\mathbf{d}}_{2}}G_3 & D_{\widetilde{Y}}G_3 
    \end{pmatrix} = \begin{pmatrix}
    \mathbb{O}_{N} & \mathbbm{1}_{N} & \cos(\widetilde{Y})\mathbf{a} \\
    \mathbbm{1}_{N} &  \mathbb{O}_{N} & \sin(\widetilde{Y})\mathbf{a} \\
    \mathbf{0}^{T} & \mathbf{0}^{T} & \sin(\widetilde{Y})
    \end{pmatrix}
\end{equation}
where $\mathbf{0}\in\mathbb{R}^{(N+1)}$, $\mathbb{O}_N\in\mathbb{R}^{(N+1)\times(N+1)}$ denote the zero vector and square matrix, respectively, and subscripts on the differential $D$ denote what variable the derivative is taken with respect to. Evaluating at the solution $\mathcal{V}^{*}$, defined in \eqref{Ntup:Match;Soln}, we find that
\begin{equation}
    D\mathcal{G}(\mathcal{V}^{*}) = \begin{pmatrix}
    \mathbb{O}_{N} & \mathbbm{1}_{N} & \mathbf{0} \\
    \mathbbm{1}_{N} &  \mathbb{O}_{N} & (-1)^{\alpha+1}\mathbf{a} \\
    \mathbf{0}^{T} & \mathbf{0}^{T} & (-1)^{\alpha+1}
    \end{pmatrix}.
\end{equation}
Hence, the determinant of $D\mathcal{G}(\mathcal{V}^{*})$ is
\begin{equation}
	\begin{split}
    \det\,[D\mathcal{G}(\mathcal{V}^{*})]&= (-1)^{\alpha+1}\det\,[-\mathbbm{1}_{N}]\det\,[\mathbbm{1}_{N}]= (-1)^{\alpha + N}.
	\end{split}
\end{equation}
Hence, the proof is complete.
\end{proof} %End of proof

With Lemma~\ref{lem:NondegenMatch} we can now solve \eqref{Ntup:match;corefar,scale} uniquely for all $0<\mu, r_{1}, r_{0}^{-1}\ll1$. Inverting the coordinate transformation \eqref{match:scaling} and using the roots $\mathcal{V}^*$ to \eqref{Ntup:match;corefar,scale0}, we get matching solutions
\begin{equation}\label{d1:d2;defn}
	\begin{split}
    		d_{1}^{(n)} &= \mu\mathcal{O}(r_{0}^{-1} + r_{1} + \mu^{\frac{1}{2}}) \\
    		d_{2}^{(n)} &= 2q_0\mu^{\frac{3}{4}}\,a_{n}\left(1+\mathcal{O}(r_{0}^{-1} + r_{1} + \mu^{\frac{1}{2}})\right),
	\end{split}
\end{equation}
for any $\mathbf{a} = \{a_{n}\}_{n=0}^{N}$ that is a fixed point of $\mathbf{C}^m_N$. We therefore arrive at the result of Theorem~\ref{thm:Ring}. Indeed, by solving the matching problem \eqref{Ntup:match;corefar,scale} for $0<\mu, r_{1}, r_{0}^{-1}\ll1$ we arrive at a localised $\mathbb{D}_m$ solution to the Galerkin system \eqref{eqn:R-D;Galerk}. 

In \eqref{RadialProfile} the form of the solution in the core region $r \in [0,r_0]$ comes from substituting the values of $d_{1}^{(n)},d_{2}^{(n)}$ from \eqref{d1:d2;defn} into the core behaviour \eqref{Core:Prof}. 
In the transition region $r_0 \leq r \leq \mu^{-\frac{1}{2}}r_1$ we have from \eqref{Ntup:Transition;Initial} that 
\begin{equation}
	\begin{split}
		\mathbf{A}(r) &= r^{\frac{1}{2}}\mu^{\frac{3}{4}}\mathbf{A}_{0}(1 + \mathcal{O}(\mu^{\frac{1}{2}})) \\ 
		\mathbf{B}(r) &= r^{-\frac{1}{2}}\mu^{\frac{3}{4}}\mathbf{A}_{0}(1 + \mathcal{O}(\mu^{\frac{1}{2}})),
	\end{split}
\end{equation} 
and so inverting the transformations in Lemma \ref{Lemma:Ntup;normal} and \eqref{R-D:Transformation;Farfield} results in
\begin{equation}\label{un:profile;transition}
    	\mathbf{u}_{n}(r) = 2 q_0\mu^{\frac{3}{4}} a_{n} \left[r^{\frac{1}{2}}\sin(r-\tfrac{mn\pi}{2}-\tfrac{\pi}{4})\hat{U}_{0} + 2r^{-\frac{1}{2}}\cos(r-\tfrac{mn\pi}{2} -\tfrac{\pi}{4})\hat{U}_{1}\right] + \mathcal{O}(\mu),
\end{equation}
for all $r \in [r_0,\mu^{-\frac{1}{2}}r_1]$. Finally, in the rescaled far-field $r \geq \mu^{-\frac{1}{2}}r_1$, solutions remain close to \eqref{Sol:Resc} for $\mathbf{A}_{R}(s), \mathbf{B}_{R}(s)$ such that 
\begin{equation}
	\begin{split}
		\mathbf{A}(r) &=\mu^{\frac{1}{2}} q(\mu^{\frac{1}{2}}r)\mathbf{a}\mathrm{e}^{\mathrm{i}Y}(1 + \mathcal{O}(\mu^{\frac{1}{2}})), \\
		\mathbf{B}(r) &=\mu^{\frac{1}{2}} \left(\tfrac{\mathrm{d}}{\mathrm{d}r} + \tfrac{1}{2r}\right)q(\mu^{\frac{1}{2}}r)\mathbf{a}\mathrm{e}^{\mathrm{i}Y}(1 + \mathcal{O}(\mu^{\frac{1}{2}})).
	\end{split}
\end{equation} 
Tracing back the transformations in Lemma \ref{Lemma:Ntup;normal} and \eqref{R-D:Transformation;Farfield} gives
\begin{equation}\label{un:profile;rescaling}
    	\mathbf{u}_{n}(r) = \pm 2\mu^{\frac{1}{2}}a_{n}\left[ q(\mu^{\frac{1}{2}}r)\sin(r-\tfrac{mn\pi}{2}-\tfrac{\pi}{4})\hat{U}_{0} + 2\left(\tfrac{\mathrm{d}}{\mathrm{d}r} + \tfrac{1}{2r}\right)q(\mu^{\frac{1}{2}}r)\cos(r-\tfrac{mn\pi}{2} -\tfrac{\pi}{4})\hat{U}_{1}\right] + \mathcal{O}(\mu),
\end{equation}
for $r \geq \mu^{-\frac{1}{2}}r_1$. The result is the radial amplitudes given in \eqref{RadialProfile}, which hold as $\mu \to 0^+$, uniformly in $r \geq 0$. This concludes the proof of Theorem~\ref{thm:Ring}.

%%%%%%%%%%%%%%%%%%%%%%%%%%%%%%%%%%%%%%%%%%%%%%%%%%%%%%%%%%%%%
\section{Discussion}\label{sec:Discussion} 

In this manuscript we have examined the emergence of localised interacting patterns arranged in a ring-like configuration to Galerkin projections of reaction-diffusion systems near from a Turing instability. The Galerkin system came from the projection of the reaction-diffusion PDE \eqref{e:RDsysSteady} onto finitely many cosine modes that respect an $m$-fold dihedral symmetry, leaving one only to identify the radially-dependent coefficients in the Fourier expansion. The finite-dimensional nature of the problem has enabled our application of radial centre manifold theory to produce solutions for the coefficients that remain regular near the origin and decay exponentially to zero. When put back into the cosine expansion, our solutions approximate the dihedral localised ring solutions that have been observed in pattern-forming PDEs, such as the Swift--Hohenberg equation, as we have shown in Figure~\ref{fig:ExRings}.

\begin{figure}[t!] %Figure: D2 Snaking figure
    \centering
    \includegraphics[width=\linewidth]{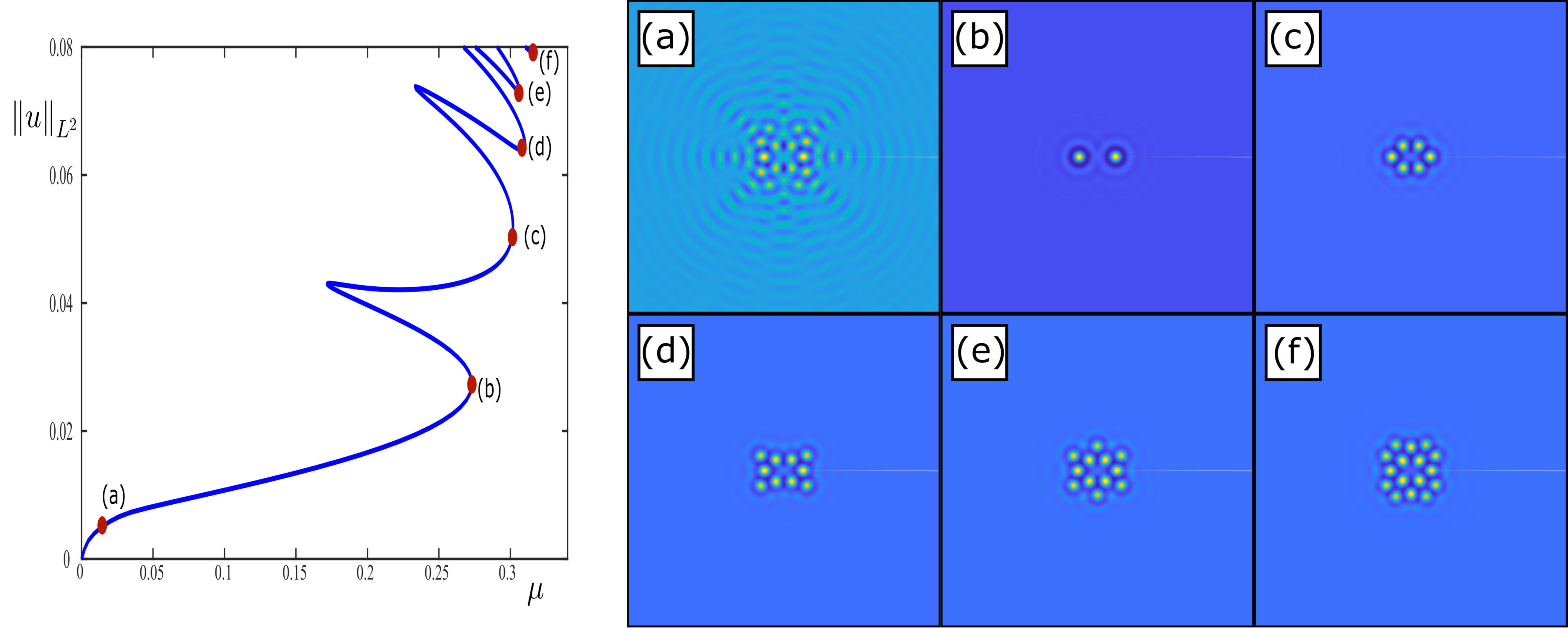}
    \caption{A snaking bifurcation curve in the $(\mu,\|u\|_{L^2})$-plane of a localised dihedral ring solution to the Swift--Hohenberg equation \eqref{e:SH} with $\mathbb{D}_2$ symmetry. Note that the existence curve bifurcates from the trivial state $u = 0$ at the Turing bifurcation point $\mu = 0$. The results of this work approximately capture this emergence of the localised pattern.}
    \label{fig:D2_Snaking}
\end{figure}

\begin{figure}[t!] %Figure: D3 and D5 "Isolas"
    \centering
    \includegraphics[width=\linewidth]{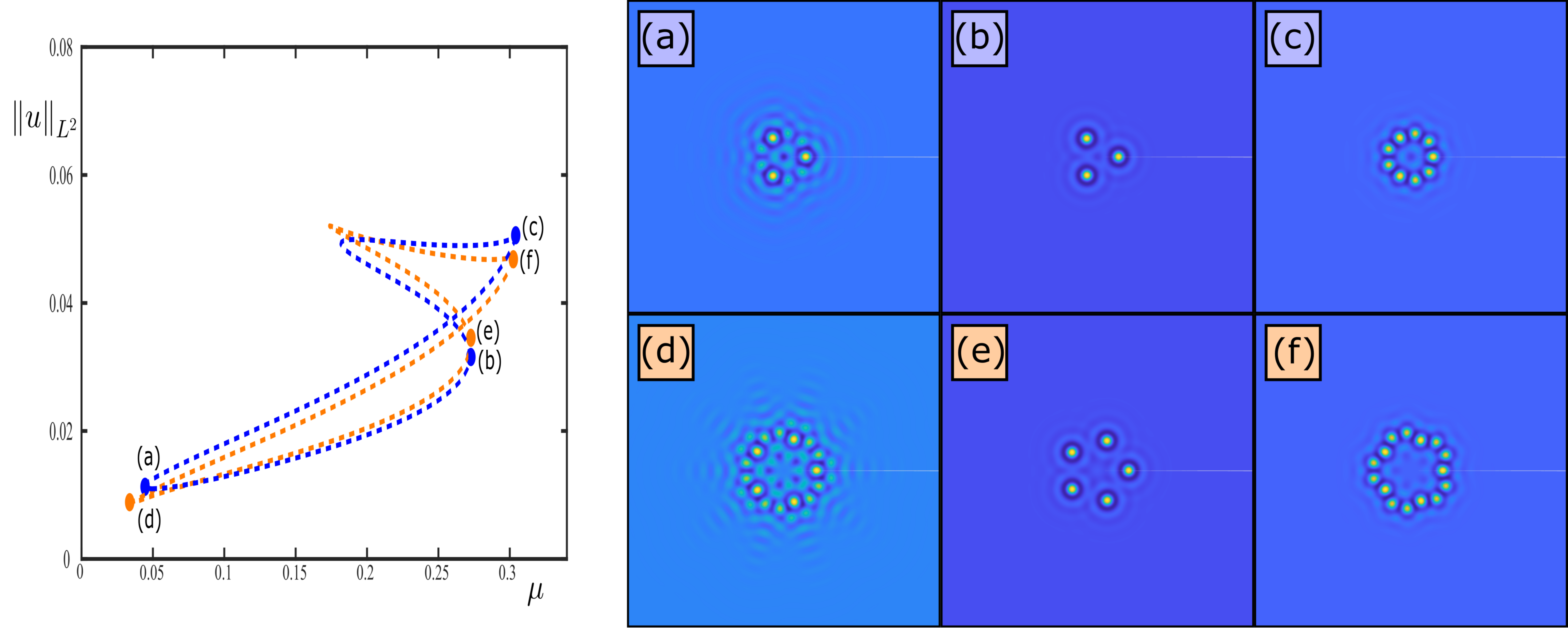}
    \caption{Not all localised dihedral ring solutions to the Swift--Hohenberg equation lead to snaking bifurcation curves. Here we present isolas of localised rings with $\mathbb{D}_3$ (blue) and $\mathbb{D}_5$ (orange) symmetries along with representative solutions at various points along the bifurcation curve. Representative solutions along the $\mathbb{D}_3$ curve are provided in panels (a) - (c), while solutions along the $\mathbb{D}_5$ curve are presented in panels (d) - (f).}
    \label{fig:D35_Isolas}
\end{figure}

\begin{figure}[t!] %Figure: D6 and D7 "Isolas"
    \centering
    \includegraphics[width=\linewidth]{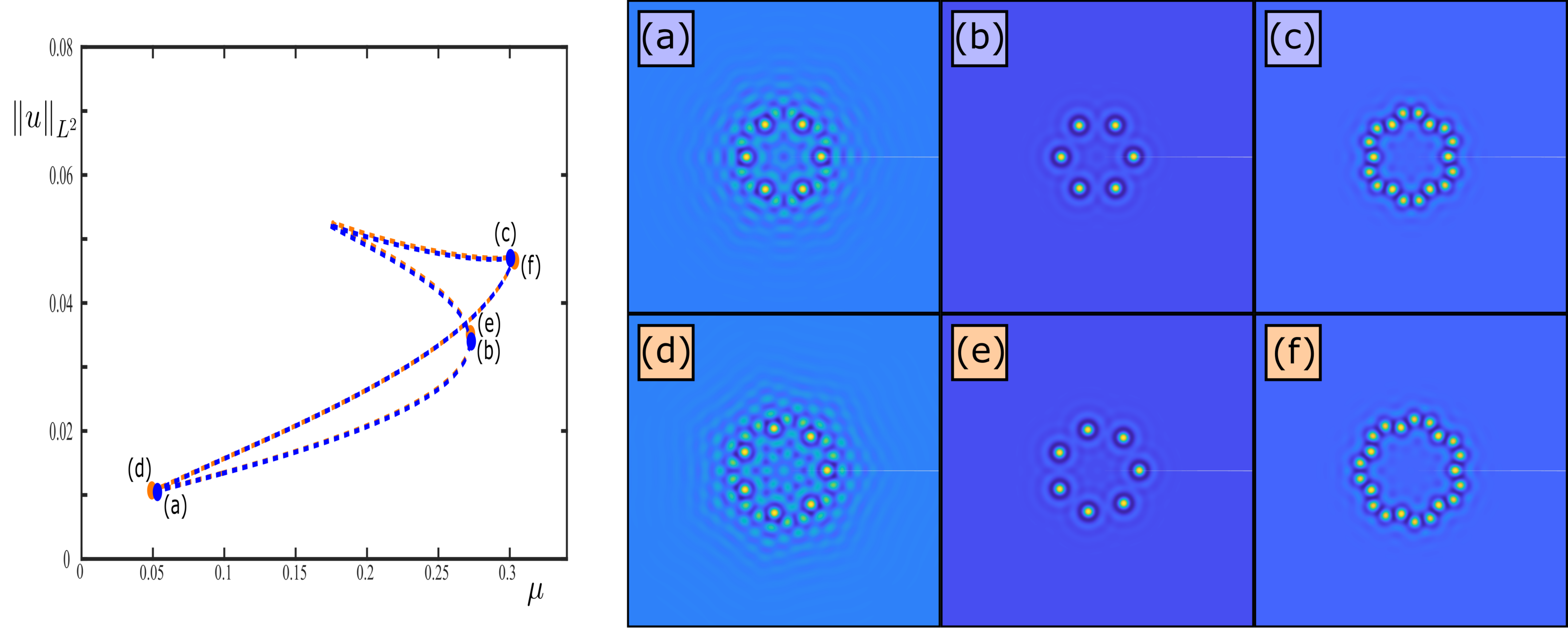}
    \caption{Isolas of localised ring solutions to the Swift--Hohenberg equation with $\mathbb{D}_6$ (blue) and $\mathbb{D}_7$ (orange) symmetries, along with representative solutions in panels (a) - (c) and (d) - (f), respectively, at various points along the bifurcation curve.}
    \label{fig:D67_Isolas}
\end{figure}

Our work comes as a novel interrogation of spatially-localised planar patterns that have long evaded an analytical treatment. Although our work does not exactly solve the PDEs in question, it does provide significant insight into what should be expected of such solutions to planar systems. Furthermore, since Galerkin projections of PDEs are a primary way to investigate pattern formation numerically, our work provides strong initial guesses that can be converged to numerical solutions and continued into moderate values of the bifurcation parameter $\mu$ with numerical continuation. As is shown in Figure~\ref{fig:D2_Snaking}, such continuation can result in the usual snaking bifurcation diagrams that are so commonly observed in the continuation of localised planar patterns. Alternatively, the existence branches need not lead to snaking, but the solutions may lie along closed curves termed `isolas', as observed in Figures~\ref{fig:D35_Isolas} and \ref{fig:D67_Isolas}. Predicting whether solutions snake or lie on isolas remains an open question in the investigation of localised structures. 

The most important question that remains from this work is of course how one can extend the solutions from our Galerkin model to the full PDE. True PDE solutions come from letting $N \to \infty$ in our cosine expansion. However, our proof of Theorem~\ref{thm:Ring} gives existence of the solutions for parameter values $\mu \in (0,\mu_0)$, where $\mu_0 \ll (mN)^{-4}$, meaning the interval of existence in the bifurcation parameter becomes smaller and smaller with $N$ becoming larger. One method to overcome this would be to carefully apply a Nash--Moser iteration that uses exponential convergence in $N$ { to counteract the parameter} interval shrinking. Recent results on the existence of quasipattern solutions have employed Nash--Moser iteration to overcome similar issues \cite{iooss2019existence} and therefore may serve to guide a follow-up investigation.

\section*{Code availability}
Codes used to produce the results in this paper are available at: %\href{https://github.com/Dan-Hill95/Localised-Dihedral-Patterns}{https://github.com/Dan-Hill95/Localised-Dihedral-Patterns}.
\href{https://github.com/Dan-Hill95/Dihedral-Rings}{https://github.com/Dan-Hill95/Dihedral-Rings}

%\section*{Data availability statement}
%All data used to produce the results in this paper will be made available upon reasonable request.

\section*{Acknowledgements}

JJB gratefully acknowledges support from the Institute of Advanced Studies at the University of Surrey and an NSERC Discovery Grant. DH gratefully acknowledges support from the Alexander von Humboldt Foundation.

\appendix

\setcounter{equation}{0}
\renewcommand\theequation{\Alph{section}.\arabic{equation}}

% Appendix 1: Proof of normal form lemma
%%%%%%%%%%%%%%%%%%%%%%%%%%%%%%%%%%%%%%%%%%%%%%%%%%%%%%%%%%%%%
\section{Proof of Lemma~\ref{Lemma:Ntup;normal}}\label{app:LemProof}

It follows from \cite[Lemma 3.10]{scheel2003radially} that, for a given $0<M<\infty$, we can write \eqref{amp:AB;tilde} in the form
\begin{equation}\label{amp:AB;musigma}
	\begin{split}
   		 \frac{\textnormal{d}}{\textnormal{d} r} \hat{A}_{n} &= [\textnormal{i}k_{1,n}(\mu,\sigma) + k_{2,n}(\mu,\sigma)]\hat{A}_{n} + \hat{B}_{n} + \mathcal{O}((|\hat{\mathbf{A}}| + |\hat{\mathbf{B}}|)^2 + |\mu||\sigma|^M(|\hat{\mathbf{A}}| + |\hat{\mathbf{B}}|)),\\
    		\frac{\textnormal{d}}{\textnormal{d} r} \hat{B}_{n} &= [\textnormal{i}k_{1,n}(\mu,\sigma) + k_{2,n}(\mu,\sigma)]\hat{B}_{n} + c_{0,n}(\mu,\sigma)\hat{A}_{n} + \mathcal{O}((|\hat{\mathbf{A}}| + |\hat{\mathbf{B}}|)^2 + |\mu||\sigma|^M(|\hat{\mathbf{A}}| + |\hat{\mathbf{B}}|)),
	\end{split}
\end{equation}
for each $n\in[0,N]$ by a transformation 
\begin{equation}\label{Normal:Transf;Lin}
    \begin{pmatrix}
    \hat{A}_{n} \\ \hat{B}_{n}
    \end{pmatrix} := \left[\mathbbm{1} + \mathcal{T}_{n}(\sigma)\right]\begin{pmatrix}
    \widetilde{A}_{n} \\ \widetilde{B}_{n}
    \end{pmatrix} + \mathcal{O}(|\mu|(|\widetilde{\mathbf{A}}| + |\widetilde{\mathbf{B}}|)).
\end{equation}
Since the linearisation of \eqref{amp:AB;tilde} about $(\widetilde{\mathbf{A}},\widetilde{\mathbf{B}})=\mathbf{0}$ decouples for each $n\in[0,N]$, the derivation of these transformations follows in the same way as for the radial problem. We use matched asymptotics in order to calculate the following leading order expansions,
\begin{equation}
    \begin{split}
        \mathcal{T}_{n}(\sigma)\begin{pmatrix}
    		\widetilde{A}_{n} \\ \widetilde{B}_{n}
    		\end{pmatrix} &= \begin{pmatrix}
    		 - [\frac{\textnormal{i}\sigma}{4} + \mathcal{O}(\sigma^2)]\overline{\widetilde{A}}_{n} - [\frac{\sigma}{4}+\mathcal{O}(\sigma^2)]\overline{\widetilde{B}}_{n}\\    
    		[\frac{\textnormal{i}\sigma}{4} + \mathcal{O}(\sigma^2)]\overline{\widetilde{B}}_{n}
    		\end{pmatrix},\\
        k_{1,n}(\mu,\sigma) &= 1 + \mathcal{O}(|\mu| + |\sigma|^2), \qquad k_{2,n}(\mu,\sigma) = -{\textstyle\frac{\sigma}{2}} + \mathcal{O}(|\sigma|^3), \\
        c_{0,n}(\mu,\sigma) &= \mu\big[{\textstyle\frac{1}{4}}\big\langle \hat{U}_{1}^{*}, -\mathbf{M}_{2}\hat{U}_{0}\big\rangle_{2} + \mathcal{O}(|\mu| + |\sigma|^2)\big],\\
    \end{split}
\end{equation}
where we have used the symmetry of \eqref{amp:AB;tilde} with respect to the reverser $\mathscr{R}:(\widetilde{\mathbf{A}}, \widetilde{\mathbf{B}}, \sigma, r)\mapsto(\overline{\widetilde{\mathbf{A}}}, - \overline{\widetilde{\mathbf{B}}}, -\sigma, -r)$ in order to write down the higher-order terms containing $\sigma$. 

We now set $\mu=\sigma=0$, such that
\begin{align}
    \frac{\textnormal{d}}{\textnormal{d} r}\hat{A}_{n} &= \textnormal{i} \hat{A}_{n} + \hat{B}_{n} - \tfrac{\textnormal{i}}{2}\big\langle \hat{U}_{0}^{*} + \tfrac{1}{2}\hat{U}_{1}^{*}, \mathcal{F}_{n}\big\rangle_{2}, & \qquad
    \frac{\textnormal{d}}{\textnormal{d} r} \hat{B}_{n} &= \textnormal{i} \hat{B}_{n} - \tfrac{1}{4}\big\langle \hat{U}_{1}^{*}, \mathcal{F}_{n}\big\rangle_{2}.\label{AmpEqn:rinf}
    \end{align}
    and look to remove any non-resonant nonlinear quadratic terms. Following \cite[Lemma 2.6]{Scheel2014Grain}, we note that there exist smooth homogeneous polynomials $\{\Phi_{n}, \xi_{n}\}_{n=0}^{N}$ of degree 2 such that the change of coordinates
\begin{equation}\label{Normal:Transf;Quad}
    \hat{A}_{n} = \Breve{A}_{n} + \Phi_{n}(\{\Breve{A}_{k}, \Breve{B}_{k}, \overline{\Breve{A}}_{k}, \overline{\Breve{B}}_{k}\}_{k=0}^{N}),\qquad
    \hat{B}_{n} = \Breve{B}_{n} + \xi_{n}(\{\Breve{A}_{k}, \Breve{B}_{k}, \overline{\Breve{A}}_{k}, \overline{\Breve{B}}_{k}\}_{k=0}^{N}),
\end{equation}
transforms \eqref{AmpEqn:rinf} into the normal form
\begin{equation}\label{amp:Normal;musigma}
	\begin{split}
   		 \frac{\textnormal{d}}{\textnormal{d} r} \Breve{A}_{n} &= \textnormal{i}\Breve{A}_{n} + \Breve{B}_{n} + \mathcal{O}((|\Breve{\mathbf{A}}| + |\Breve{\mathbf{B}}|)^3),\qquad 
    		\frac{\textnormal{d}}{\textnormal{d} r} \Breve{B}_{n} = \textnormal{i}\Breve{B}_{n} + \mathcal{O}((|\Breve{\mathbf{A}}| + |\Breve{\mathbf{B}}|)^3),
	\end{split}
\end{equation}
for each $n\in[0,N]$. As in \cite[Proof of Lemma 2.6]{Scheel2014Grain}, this result holds for $(\Breve{\mathbf{A}},\Breve{\mathbf{B}})$ in a neighbourhood of the origin if the quadratic terms in \eqref{amp:AB;musigma} belong to the range of the operator $(\mathcal{D}-\textnormal{i})$, where
\begin{equation}
    \mathcal{D} = \sum_{n=0}^{N} \left\{(\textnormal{i} \Breve{A}_n + \Breve{B}_n)\frac{\partial}{\partial \Breve{A}_n} + (\textnormal{i} \Breve{B}_n)\frac{\partial}{\partial \Breve{B}_n} + (-\textnormal{i} \overline{\Breve{A}}_n + \overline{\Breve{B}}_n)\frac{\partial}{\partial \overline{\Breve{A}}_n} + ( -\textnormal{i}\overline{\Breve{B}}_n)\frac{\partial}{\partial \overline{\Breve{B}}_n}\right\}.
\end{equation}
An explicit calculation verifies that every type of quadratic monomial belongs to the range of $(\mathcal{D} - \textnormal{i})$, and so we conclude that $\{\Phi_{n}, \xi_{n}\}_{n=0}^{N}$ exist such that we can remove every quadratic term from \eqref{amp:AB;musigma}. Then, \eqref{AmpEqn:rinf} becomes, up to cubic order,
\begin{align}
    \frac{\textnormal{d}}{\textnormal{d} r}\Breve{A}_{n} &= \textnormal{i} \Breve{A}_{n}  + \Breve{B}_{n}  -\tfrac{\textnormal{i}}{2}\big\langle \hat{U}_{0}^{*} + \tfrac{1}{2}\hat{U}_{1}^{*}, \mathcal{F}_{n} - \mathcal{F}_{n}|_{2}\big\rangle_{2} \big|_{3}, & \qquad 
    \frac{\textnormal{d}}{\textnormal{d} r} \Breve{B}_{n} &= \textnormal{i} \Breve{B}_{n} - \tfrac{1}{4}\big\langle \hat{U}_{1}^{*}, \mathcal{F}_{n} - \mathcal{F}_{n}|_{2}\big\rangle_{2}\big|_{3},\label{Eqn:Amp;Cubic}
    \end{align}
Again, we note that there exist smooth homogeneous polynomials $\{\Psi_n, \chi_n\}_{n=0}^{N}$ of degree 3 such that the change of coordinates
\begin{equation}\label{Normal:Transf;Cube}
    \Breve{A}_{n} = A_{n} + \Psi_{n}(\{A_{k}, B_{k}, \overline{A}_{k}, \overline{B}_{k}\}_{k=0}^{N}),\qquad
    \Breve{A}_{n} = B_{n} + \chi_{n}(\{A_{k}, B_{k}, \overline{A}_{k}, \overline{B}_{k}\}_{k=0}^{N}),
\end{equation}
transforms \eqref{AmpEqn:rinf} into the normal form
\begin{equation*}
	\begin{split}
   		 \frac{\textnormal{d}}{\textnormal{d} r} A_{n} &= \textnormal{i}A_{n} + B_{n} + \mathcal{O}(|\mathbf{B}|^{2}(|\mathbf{A}| + |\mathbf{B}|) + |\mathbf{A}|^4),\qquad 
    		\frac{\textnormal{d}}{\textnormal{d} r} B_{n} = \textnormal{i}B_{n} + c_3\big[\widehat{\mathbf{C}}_{N}(\mathbf{A})\big]_{n} + \mathcal{O}(|\mathbf{B}|^{2}(|\mathbf{A}| + |\mathbf{B}|) + |\mathbf{A}|^4),
	\end{split}
\end{equation*}
where 
\begin{equation}
    \begin{split}\big[\widehat{\mathbf{C}}_{N}(\mathbf{A})\big]_{n} &= \sum_{i+j + k =n} A_{|i|} A_{|j|}\overline{A}_{|k|},
    \end{split}
\end{equation}
for each $n\in[0,N]$. Again, as in \cite[Proof of Lemma 2.6]{Scheel2014Grain}, this result holds for $(\mathbf{A},\mathbf{B})$ in a neighbourhood of the origin if the cubic terms in \eqref{amp:AB;musigma} belong to the range of the operator $(\mathcal{D}-\textnormal{i})$. The only cubic terms that do not belong to the range of $(\mathcal{D}-\textnormal{i})$, and hence can't be removed by standard normal form transformations, are of the form
\begin{align}
    & {B}_{|i|} {B}_{|j|}\overline{ {B}}_{|k|},& \quad   & {A}_{|i|} {B}_{|j|}\overline{ {B}}_{|k|},& \quad & {B}_{|i|} {B}_{|j|}\overline{ {A}}_{|k|},& \quad 
    & {B}_{|i|} {A}_{|j|}\overline{ {A}}_{|k|},& \quad   & {A}_{|i|} {A}_{|j|}\overline{ {B}}_{|k|},& \quad & {A}_{|i|} {A}_{|j|}\overline{ {A}}_{|k|}.& \quad &\nonumber
\end{align}
However, we can remove ${A}_{|i|} {A}_{|j|}\overline{ {A}}_{|k|}$ terms from the $\frac{\textnormal{d}}{\textnormal{d} r}A_{n}$ equation, at the expense of changing the ${B}_{|i|} {A}_{|j|}\overline{{A}}_{|k|}$ terms in the $\frac{\textnormal{d}}{\textnormal{d} r}B_{n}$ equation. In order to compute the value of $c_3$, we examine the ${A}_{|i|} {A}_{|j|}\overline{ {A}}_{|k|}$ terms in $- \frac{1}{4}\big\langle \hat{U}_{1}^{*}, \left.\mathcal{F}_{n}-\mathcal{F}_{n}\right|_{2}\big\rangle_{2}$, where we find that
\begin{equation}\label{Fn:Atilde}\begin{split}
    \left.\mathcal{F}_{n}\right|_{A_{|i|}A_{|j|}\overline{A}_{|k|}} &= Q_{0,0}\sum_{i+j=n}\widetilde{A}_{|i|} \widetilde{A}_{|j|} + 2 Q_{0,0}\sum_{i+j=n}\widetilde{A}_{|i|} \overline{\widetilde{A}}_{|j|} + Q_{0,0}\sum_{i+j=n}\overline{\widetilde{A}}_{|i|} \overline{\widetilde{A}}_{|j|} \\
    & \quad + 4\textnormal{i} Q_{0,1}\sum_{i+j=n}\widetilde{A}_{|i|} \widetilde{B}_{|j|} + 4\textnormal{i} Q_{0,1}\sum_{i+j=n}\overline{\widetilde{A}}_{|i|} \widetilde{B}_{|j|} - 4\textnormal{i} Q_{0,1}\sum_{i+j=n}\widetilde{A}_{|i|} \overline{\widetilde{B}}_{|j|} \\
    &\quad - 4\textnormal{i} Q_{0,1}\sum_{i+j=n} \overline{\widetilde{A}}_{|i|} \overline{\widetilde{B}}_{|j|} + 3 C_{0,0,0} \sum_{i+j+k=n} \overline{\widetilde{A}}_{|i|} \overline{\widetilde{A}}_{|j|} \widetilde{A}_{|k|}, 
\end{split}\end{equation}
We apply the polynomial coordinate changes $\{\Phi_n,\xi_n\}_{n=0}^{N}$ and $\{\Psi_n,\chi_n\}_{n=0}^{N}$, which results in
\begin{equation}\label{An:Transf;cubic}\begin{split}
    \widetilde{A}_{|j|}|_{A_{|i|}A_{|j|}\overline{A}_{|k|}} &= A_{|j|} + \sum_{ k + \ell=|j|} \left\{ p_{1} A_{|k|}A_{|\ell|} + p_{2} A_{|k|}\overline{A}_{|\ell|} \right\},\\
    \overline{\widetilde{A}}_{|j|}|_{A_{|i|}A_{|j|}\overline{A}_{|k|}} &= \overline{A}_{|j|} + \sum_{ k + \ell=|j|} \left\{ \overline{p}_{2} A_{|k|}\overline{A}_{|\ell|} + \overline{p}_{3}A_{|k|}A_{|\ell|}\right\},\\
    \widetilde{B}_{|j|}|_{A_{|i|}A_{|j|}\overline{A}_{|k|}} &= B_{|j|} +  \sum_{ k + \ell=|j|} \left\{ q_{1} A_{|k|}A_{|\ell|} + q_{2} A_{|k|}\overline{A}_{|\ell|} \right\},\\
    \overline{\widetilde{B}}_{|j|}|_{A_{|i|}A_{|j|}\overline{A}_{|k|}} &= \overline{B}_{|j|} +  \sum_{ k + \ell=|j|} \left\{ \overline{q}_{2} A_{|k|}\overline{A}_{|\ell|} + \overline{q}_{3}A_{|k|}A_{|\ell|} \right\},\end{split}\end{equation}
for some fixed $p_i, q_i\in\mathbb{C}$, to be determined. Note we are only writing down the terms that result in a change in the $A_{i}A_{j}\overline{A}_{k}$ terms in \eqref{Fn:Atilde}; the polynomial transformations result in many more terms, but they are not relevant to this analysis and so have been omitted. Substituting \eqref{An:Transf;cubic} into $\mathcal{F}_{n} - \mathcal{F}_{n}|_{2}$, we find 
\begin{align}
    \left(\mathcal{F}_{n} - \mathcal{F}_{n}|_{2}\right)|_{A_{|i|}A_{|j|}\overline{A}_{|k|}} &= \left[2(p_{2} + \overline{p}_{2})Q_{0,0} + 4\textnormal{i}( q_{2} - \overline{q}_{2}) Q_{0,1}\right]\sum_{i+j=n} \sum_{ k + \ell=|j|} \left\{  A_{|i|}A_{|k|}\overline{A}_{|\ell|} \right\}  \nonumber\\
    &\qquad + \left[2( p_{1}+\overline{p}_{3})Q_{0,0} + 4\textnormal{i} (q_{1} - \overline{q}_{3})Q_{0,1}\right]\sum_{\ell+j=n} \sum_{ k + i=|j|} \left\{  A_{|i|}A_{|k|}\overline{A}_{|\ell|}\right\} \nonumber\\
    &\qquad + 3 C_{0,0,0} \sum_{i+j+k=n} A_{|i|} A_{|j|} \overline{A}_{|k|}. \nonumber
\end{align}
Furthermore, we note that
\begin{equation*}
    \begin{split}
        \sum_{i+j+k=n} A_{|i|} A_{|j|} \overline{A}_{|k|} &= \sum_{i+\ell=n}\sum_{j+k=\ell} A_{|i|} A_{|j|} \overline{A}_{|k|} = \sum_{i+\ell=n}\sum_{j+k=|\ell|} A_{|i|} A_{|j|} \overline{A}_{|k|},
    \end{split}
\end{equation*} 
and so we can write
\begin{align}
    \left(\mathcal{F}_{n} - \mathcal{F}_{n}|_{2}\right)|_{A_{|i|}A_{|j|}\overline{A}_{|k|}} &= \left(2(p_{2} + \overline{p}_{2} + p_{1}+\overline{p}_{3})Q_{0,0} + 4\textnormal{i}( q_{2} - \overline{q}_{2} + q_{1} - \overline{q}_{3}) Q_{0,1} + 3 C_{0,0,0}\right) \sum_{i+j+k=n} A_{|i|} A_{|j|} \overline{A}_{|k|}. \nonumber
\end{align} 
In order to determine the resonant cubic coefficients of \eqref{AmpEqn:rinf}, we require an explicit form for the polynomials $\{\Phi_{n},\xi_{n}\}_{n=0}^{N}$; hence, we introduce the general transformation
\begin{align}
    \widetilde{A}_{n} &= A_{n} + \Phi_{n}, \qquad \widetilde{B}_{n} = B_{n} + \xi_{n},\label{AppNormal:quadratic;transf}\end{align}
with
\begin{align}
    \Phi_{n} &= \sum_{ i + j=n} \left\{ p_{1} A_{|i|}A_{|j|} + p_{2} A_{|i|}\overline{A}_{|j|} + p_{3}\overline{A}_{|i|}\overline{A}_{|j|}  + p_{4} A_{|i|}B_{|j|} + p_{5} \overline{A}_{|i|}B_{|j|} + p_{6} A_{|i|}\overline{B}_{|j|} \right.\nonumber\\
    &\qquad\qquad \qquad \qquad\left.+ p_{7} \overline{A}_{|i|}\overline{B}_{|j|} + p_{8} B_{|i|}B_{|j|} + p_{9} B_{|i|}\overline{B}_{|j|} + p_{10} \overline{B}_{|i|}\overline{B}_{|j|} \right\},\nonumber\\
    \xi_{n} &= \sum_{ i + j=n} \left\{ q_{1} A_{|i|}A_{|j|} + q_{2} A_{|i|}\overline{A}_{|j|} + q_{3}\overline{A}_{|i|}\overline{A}_{|j|}  + q_{4} A_{|i|}B_{|j|} + q_{5} \overline{A}_{|i|}B_{|j|} + q_{6} A_{|i|}\overline{B}_{|j|} \right.\nonumber\\
    &\qquad\qquad \qquad \qquad\left.+ q_{7} \overline{A}_{|i|}\overline{B}_{|j|} + q_{8} B_{|i|}B_{|j|} + q_{9} B_{|i|}\overline{B}_{|j|} + q_{10}\overline{B}_{|i|}\overline{B}_{|j|} \right\},\nonumber\end{align}
    where each coefficient is fixed such that all quadratic terms in \eqref{AmpEqn:rinf} are removed; note that, since the polynomial transformations $\{\Psi_n, \chi_n\}_{n=0}^{N}$ consist of cubic monomials, the coefficients $p_i, q_i$ are identical to those in \eqref{An:Transf;cubic} for $i=1,2,3$. Then, up to quadratic order, the polynomials $\{\Phi_{n},\xi_{n}\}_{n=0}^{N}$ must satisfy
\begin{align}
    (\mathcal{D} - \textnormal{i}) \Phi_{n} &= \xi_{n}  - \frac{\textnormal{i}}{2}\big\langle \hat{U}_{0}^{*} + \frac{1}{2}\hat{U}_{1}^{*}, \mathcal{F}_{n}\big\rangle_{2}\big|_{2}, & \qquad
    (\mathcal{D} - \textnormal{i}) \xi_{n} &=  - \frac{1}{4 }\big\langle \hat{U}_{1}^{*}, \mathcal{F}_{n}\big\rangle_{2}\big|_{2}.\label{Eqn:Amp;Quad}
\end{align}
 We begin with the second equation in \eqref{Eqn:Amp;Quad}, where the left-hand side becomes
\begin{align}
    (\mathcal{D} - \textnormal{i})\xi_{n} &= \sum_{ i + j=n} \left\{ \textnormal{i} q_{1} A_{|i|}A_{|j|} -\textnormal{i} q_{2}  A_{|i|}\overline{A}_{|j|} -3\textnormal{i} q_{3} \overline{A}_{|i|}\overline{A}_{|j|} \right\} + \mathcal{O}([|\mathbf{A}| + |\mathbf{B}|]|\mathbf{B}|).\nonumber
    \intertext{Likewise, we explicitly compute the right-hand side to be}
    -\tfrac{1}{4}\big\langle \hat{U}_{1}^{*},\mathcal{F}_{n}\big\rangle_{2}|_{2} &= -\tfrac{1}{4}\big\langle \hat{U}_{1}^{*},Q_{0,0}\big\rangle_{2}\sum_{i+j=n}\left\{  A_{|i|} A_{|j|}  + 2 A_{|i|} \overline{A}_{|j|} 
    % \right.\nonumber\\ &\qquad\qquad\qquad\left.
    + \overline{A}_{|i|} \overline{A}_{|j|} \right\} + \mathcal{O}([|\mathbf{A}| + |\mathbf{B}|]|\mathbf{B}|),\nonumber
\end{align}
and so, matching the left- and right-hand sides, we obtain the following solution
\begin{align}
    &q_{1} = \tfrac{\textnormal{i}}{4}\big\langle \hat{U}_{1}^{*},Q_{0,0}\big\rangle_{2},&\qquad 
    &q_{2} = -\tfrac{\textnormal{i}}{2} \big\langle \hat{U}_{1}^{*},Q_{0,0}\big\rangle_{2},&\qquad
    &q_{3} = -\tfrac{\textnormal{i}}{12} \big\langle \hat{U}_{1}^{*},Q_{0,0}\big\rangle_{2},&\nonumber
\end{align}
such that the quadratic terms are removed from the second equation in \eqref{Eqn:Amp;Quad}.Turning our attention to the first equation in \eqref{Eqn:Amp;Quad}, we see that
\begin{align}
    (\mathcal{D} - \textnormal{i})\Phi_{n} &= \sum_{ i + j=n} \left\{ \textnormal{i} p_{1} A_{|i|}A_{|j|} -\textnormal{i} p_{2} A_{|i|}\overline{A}_{|j|} -3\textnormal{i} p_{3} \overline{A}_{|i|}\overline{A}_{|j|}\right\} + \mathcal{O}([|\mathbf{A}| + |\mathbf{B}|]|\mathbf{B}|),\nonumber\\
    \xi_{n} &= \tfrac{\textnormal{i}}{4}\big\langle \hat{U}_{1}^{*},Q_{0,0}\big\rangle_{2} \sum_{i+j=n}\big\{ A_{|i|} A_{|j|} -2 A_{|i|} \overline{A}_{|j|} -\tfrac{1}{3} \overline{A}_{|i|} \overline{A}_{|j|} \big\}  + \mathcal{O}([|\mathbf{A}| + |\mathbf{B}|]|\mathbf{B}|), \nonumber\\
    -\frac{\textnormal{i}}{4}\big\langle \hat{U}_{1}^{*},\mathcal{F}_{n}\big\rangle_{2}\big|_{2} &= -\tfrac{\textnormal{i}}{4}\big\langle \hat{U}_{1}^{*},Q_{0,0}\big\rangle_{2}\sum_{i+j=n}\left\{  A_{|i|} A_{|j|}  + 2 A_{|i|} \overline{A}_{|j|} 
    % \right.\nonumber\\ &\qquad\qquad\qquad\left.
    + \overline{A}_{|i|} \overline{A}_{|j|} \right\} + \mathcal{O}([|\mathbf{A}| + |\mathbf{B}|]|\mathbf{B}|),\nonumber\\
    -\frac{\textnormal{i}}{2}\big\langle \hat{U}_{0}^{*},\mathcal{F}_{n}\big\rangle_{2}\big|_{2} &= -\tfrac{\textnormal{i}}{2}\big\langle \hat{U}_{0}^{*},Q_{0,0}\big\rangle_{2}\sum_{i+j=n}\left\{  A_{|i|} A_{|j|}  + 2 A_{|i|} \overline{A}_{|j|} 
    % \right.\nonumber\\ &\qquad\qquad\qquad\left.
    + \overline{A}_{|i|} \overline{A}_{|j|} \right\} + \mathcal{O}([|\mathbf{A}| + |\mathbf{B}|]|\mathbf{B}|).\nonumber
\end{align}
Then, we can match the left- and right-hand sides of the first equation in \eqref{Eqn:Amp;Quad} in order to find
\begin{align}
    &p_{1} = -\frac{1}{2}\big\langle \hat{U}_{0}^{*},Q_{0,0}\big\rangle_{2},&
    \qquad 
    &p_{2} =  \bigg\langle \hat{U}_{0}^{*} + \hat{U}_{1}^{*},Q_{0,0}\bigg\rangle_{2} ,&\qquad &p_{3} =  \frac{1}{6} \bigg\langle \hat{U}_{0}^{*} + \frac{2}{3}\hat{U}_{1}^{*},Q_{0,0}\bigg\rangle_{2},&
    \nonumber
\end{align}
and so we have obtained
\begin{align}
    \left.\mathcal{F}_{n}\right|_{A_{|i|}A_{|j|}\overline{A}_{|k|}} &= 4\left[\tfrac{5}{6}\big\langle \hat{U}_{0}^{*},Q_{0,0}\big\rangle_{2}Q_{0,0} + \tfrac{19}{18}\big\langle \hat{U}_{1}^{*},Q_{0,0}\big\rangle_{2}Q_{0,0} +  \tfrac{5}{6}\big\langle \hat{U}_{1}^{*},Q_{0,0}\big\rangle_{2}Q_{0,1} + \tfrac{3}{4}C_{0,0,0}\right]\sum_{i+k+\ell=n}   A_{|i|}A_{|j|}\overline{A}_{|k|},\nonumber
\end{align}
which implies that
\begin{equation*}
    \begin{split}
    c_{3} &= - \frac{1}{4}\big\langle \hat{U}_{1}^{*}, \left.\mathcal{F}_{n}\right|_{A_{|i|}A_{|j|}\overline{A}_{|k|}}\big\rangle_{2}= - \bigg[\big(\tfrac{5}{6}\big\langle \hat{U}_{0}^{*},Q_{0,0}\big\rangle_{2} + \tfrac{5}{6}\big\langle \hat{U}_{1}^{*},Q_{0,1}\big\rangle_{2} + \tfrac{19}{18}\big\langle \hat{U}_{1}^{*},Q_{0,0}\big\rangle_{2} \big)\big\langle \hat{U}_{1}^{*},Q_{0,0}\big\rangle_{2} +\tfrac{3}{4}\big\langle \hat{U}_{1}^{*}, C_{0,0,0}\big\rangle_{2}\bigg]. \end{split}
\end{equation*}
Hence, applying all of these transformations to \eqref{amp:AB;tilde}, we find that
\begin{equation}\label{NormalForm}
	\begin{split}
    \frac{\textnormal{d}}{\textnormal{d}r} \mathbf{A} &= - \frac{\sigma}{2} \mathbf{A} + \mathbf{B} + \mathbf{R}_{\mathbf{A}}(\mathbf{A}, \mathbf{B},\sigma,\mu),\\
    \frac{\textnormal{d}}{\textnormal{d}r} \mathbf{B} &= -\frac{\sigma}{2} \mathbf{B} + c_{0}\,\mu \mathbf{A} + c_{3}\widehat{\mathbf{C}}_{N}(\mathbf{A}) + \mathbf{R}_{\mathbf{B}}(\mathbf{A}, \mathbf{B},\sigma,\mu),\\
    \frac{\textnormal{d}}{\textnormal{d} r} \sigma &= -\sigma^{2}.
	\end{split}
\end{equation}
Finally, we remove a relative phase from $(A_{n}, B_{n})$ for each $n\in[0,N]$. We define $\phi_{n}(r)$ to be the solution of
\begin{equation}
    \frac{\textnormal{d}}{\textnormal{d} r} \phi_n = k_{1,n}(\mu,\sigma) = 1 + \mathcal{O}(|\mu| + |\sigma|^2), \qquad \phi_{n}(0) = {\textstyle-\frac{mn \pi}{2}}, 
\end{equation}
and employ the transformation $(A_n, B_n)\mapsto \textnormal{e}^{\textnormal{i}\phi_n}(A_n, B_n)$ for each $n\in[0,N]$. This transformation, after absorbing the higher order terms of $k_{2,n}(\mu,\sigma)$ and $c_{0,n}(\mu,\sigma)$ into the remainder, turns \eqref{amp:AB;musigma} into the desired form \eqref{Ntup:NormalForm}, where
\begin{equation}
    \begin{split}
    \big[\widehat{\mathbf{C}}_{N}(\mathbf{A})\big]_{n} \mapsto \big[\mathbf{C}^{m}_{N}(\mathbf{A})\big]_{n} :&= \sum_{i+j + k =n}\textnormal{e}^{-\textnormal{i}\frac{m(|i|+|j|-|k|-n)\pi}{2}} A_{|i|} A_{|j|}\overline{A}_{|k|},\\
    &= \sum_{i+j + k =n}(-1)^{\frac{m}{2}(|i|+|j|-|k|-n)} A_{|i|} A_{|j|}\overline{A}_{|k|},
    \end{split}
\end{equation}
and thus completes the proof.

% Appendix 2: Properties of cubic vector Cn
%%%%%%%%%%%%%%%%%%%%%%%%%%%%%%%%%%%%%%%%%%%%%%%%%%%%%%%%%%%%%
\section{Properties of the Cubic Vector Function}

In this appendix we explore several properties of the cubic vector function $\mathbf{C}_{N}(\mathbf{a})$ defined in \eqref{CN:defn}. We begin by proving two results regarding the Jacobian of $\mathbf{C}_{N}(\mathbf{a})$ such that our localised solutions form a transverse intersection between the core and far-field manifolds; see Lemma~\ref{Lem:Resc;Evo}. Following this, we highlight various symmetries of the matching equation and conclude by providing numerical solutions for small truncation orders $N=0,1,2,3$.

\subsection{Properties of the Jacobian}
We recall that $\mathbf{C}^{m}_{N}(\mathbf{a})$ is defined as follows,
\begin{equation}\label{CN;App}\begin{split}
[\mathbf{C}^{m}_{N}(\mathbf{a})]_{n} &= \sum_{i+j+k=n} (-1)^{\frac{m(|i| + |j| - |k| - n)}{2}}a_{|i|}a_{|j|}a_{|k|},
\end{split}\end{equation}
and we define the Jacobian of $\mathbf{C}^{m}_{N}$, evaluated at some point $\mathbf{a}\in\mathbb{R}^{N+1}$, as follows
\begin{equation}
    D\mathbf{C}^{m}_{N}(\mathbf{a}) = \begin{pmatrix}
    D_0[\mathbf{C}^{m}_{N}(\mathbf{a})]_{0} & D_1[\mathbf{C}^{m}_{N}(\mathbf{a})]_{0} & \dots & D_N[\mathbf{C}^{m}_{N}(\mathbf{a})]_{0} \\
    D_0[\mathbf{C}^{m}_{N}(\mathbf{a})]_{1} & D_1[\mathbf{C}^{m}_{N}(\mathbf{a})]_{1} & \dots & D_N[\mathbf{C}^{m}_{N}(\mathbf{a})]_{1} \\
    \vdots & \vdots & \ddots & \vdots \\
    D_0[\mathbf{C}^{m}_{N}(\mathbf{a})]_{N} & D_1[\mathbf{C}^{m}_{N}(\mathbf{a})]_{N} & \dots & D_N[\mathbf{C}^{m}_{N}(\mathbf{a})]_{N} 
    \end{pmatrix},
\end{equation}
where $D_{i}$ denotes the derivative with respect to $a_{i}$. We now investigate some properties of $D\mathbf{C}^{m}_{N}(\mathbf{a})$,
\begin{lem}\label{Lem:DCN;Symm}
Fix $N,m>0$. Then, for any $\mathbf{a}\in\mathbb{R}^{N+1}$,
% the Jacobian $D\mathbf{C}^{m}_{N}(\mathbf{a})\in\mathbb{R}^{(N+1)\times(N+1)}$ of the cubic vector function $\mathbf{C}^{m}_{N}$ possesses the following symmetries
\begin{equation}
    D_{n}[\mathbf{C}^{m}_{N}(\mathbf{a})]_{\ell} = D_{\ell}[\mathbf{C}^{m}_{N}(\mathbf{a})]_{n},  \quad\qquad D_{n}[\mathbf{C}^{m}_{N}(\mathbf{a})]_{0} = 2D_{0}[\mathbf{C}^{m}_{N}(\mathbf{a})]_{n}, \qquad\qquad \forall \ell,n\in[1,N].
\end{equation}
\end{lem}
\begin{proof}
We assume $\ell,n\in[1,N]$ throughout this proof. Using the definition of $\mathbf{C}^{m}_{N}$ in \eqref{CN;App} we see that, 
\begin{equation}\begin{split}
D_\ell[\mathbf{C}^{m}_{N}(\mathbf{a})]_{n} &= 2\sum_{j+k=n-\ell} (-1)^{\frac{m(|j| - |k| + \ell - n)}{2}}a_{|j|}a_{|k|} + \sum_{j+k=n-\ell} (-1)^{\frac{m(|j| + |k| - \ell - n)}{2}}a_{|j|}a_{|k|}\\
% &\quad + \sum_{j+k=n-\ell} (-1)^{\frac{m(|j| - |k| + \ell - n)}{2}}a_{|j|}a_{|k|} + \sum_{j+k=n+\ell} (-1)^{\frac{m(|j|- |k| + \ell - n)}{2}}a_{|j|}a_{|k|}\\
&\quad  + 2\sum_{j+k=n+\ell} (-1)^{\frac{m(|j| - |k| + \ell - n)}{2}}a_{|j|}a_{|k|}+ \sum_{j+k=n+\ell} (-1)^{\frac{m(|j| + |k| - \ell - n)}{2}}a_{|j|}a_{|k|}.
\end{split}\end{equation}
For the four summations on the right-hand-side, we perform the following respective coordinate transformations $(j,k)\mapsto(-k,-j)$, $(j,k)\mapsto(-j,-k)$, $(j,k)\mapsto(k,j)$, $(j,k)\mapsto(j,k)$, such that
\begin{equation}\begin{split}
D_\ell[\mathbf{C}^{m}_{N}(\mathbf{a})]_{n} &= 2\sum_{j+k=\ell-n} (-1)^{\frac{m(|j| - |k| + n - \ell)}{2}}a_{|j|}a_{|k|} + \sum_{j+k=\ell-n} (-1)^{\frac{m(|j| + |k| - n - \ell)}{2}}a_{|j|}a_{|k|}\\
&\quad + 2\sum_{j+k=n+\ell} (-1)^{\frac{m(|j| - |k| + n - \ell)}{2}}a_{|j|}a_{|k|}+ \sum_{j+k=n+\ell} (-1)^{\frac{m(|j| + |k| - n - \ell)}{2}}a_{|j|}a_{|k|}\\
&=D_n[\mathbf{C}^{m}_{N}(\mathbf{a})]_{\ell}.
\end{split}\end{equation}
Hence, the first equation holds for all $\ell,n\in[1,N]$. For the second equation, we note that 
\begin{equation}\begin{split}
D_{0}[\mathbf{C}^{m}_{N}(\mathbf{a})]_{n} &= 2\sum_{j+k=n} (-1)^{\frac{m(|j| - |k| - n)}{2}}a_{|j|}a_{|k|} + \sum_{j+k=n} (-1)^{\frac{m(|j| + |k|- n)}{2}}a_{|j|}a_{|k|},
\end{split}\end{equation}
and
\begin{equation}\label{Eqn:DnC0}\begin{split}
D_{n}[\mathbf{C}^{m}_{N}(\mathbf{a})]_{0} &= 2\sum_{j+k=-n} (-1)^{\frac{m(|j| - |k| + n)}{2}}a_{|j|}a_{|k|} + \sum_{j+k=-n} (-1)^{\frac{m(|j| + |k| - n)}{2}}a_{|j|}a_{|k|}\\
&\quad + 2\sum_{j+k=n} (-1)^{\frac{m(|j| - |k| + n)}{2}}a_{|j|}a_{|k|} + \sum_{j+k=n} (-1)^{\frac{m(|j| + |k| - n)}{2}}a_{|j|}a_{|k|}.\\
\end{split}\end{equation}
Again, for the four summmations on the RHS of \eqref{Eqn:DnC0}, we perform the following respective coordinate transformations $(j,k)\mapsto(-k,-j)$, $(j,k)\mapsto(-j,-k)$, $(j,k)\mapsto(k,j)$, $(j,k)\mapsto(j,k)$, such that
\begin{equation}\begin{split}
D_{n}[\mathbf{C}^{m}_{N}(\mathbf{a})]_{0} &= 4\sum_{j+k=n} (-1)^{\frac{m(|j| - |k| - n)}{2}}a_{|j|}a_{|k|} + 2\sum_{j+k=n} (-1)^{\frac{m(|j| + |k| - n)}{2}}a_{|j|}a_{|k|}\\
&= 2D_{0}[\mathbf{C}^{m}_{N}(\mathbf{a})]_{n},
\end{split}\end{equation}
and so the second equation holds for all $n\in[1,N]$. 
\end{proof}

 \begin{lem}\label{Lem:3Eigenvec}
 Fix $m,N>0$ and assume that there exists a fixed point $\mathbf{a}\in\mathbb{R}^{N+1}\backslash\{\mathbf{0}\}$ of the vector function $\mathbf{C}^{m}_{N}$, as defined in \eqref{CN;App}. Then, 
 \begin{equation}
     D\mathbf{C}^{m}_{N}(\mathbf{a})\mathbf{a} = 3\mathbf{a}.
 \end{equation}
 \end{lem}
 \begin{proof}
 Taking derivatives of $\mathbf{C}^{m}_{N}(\mathbf{a})$, with $\ell\in[1,N]$ and $n\in[0,N]$, we find that 
 \begin{equation}\begin{split}
D_\ell[\mathbf{C}^{m}_{N}(\mathbf{a})]_{n} &= 2\sum_{j+k=\ell-n} (-1)^{\frac{m(|j| - |k| + \ell - n)}{2}}a_{|j|}a_{|k|} + \sum_{j+k=\ell-n} (-1)^{\frac{m(|j| + |k| - \ell - n)}{2}}a_{|j|}a_{|k|}\\
&\quad + 2\sum_{j+k=\ell+n} (-1)^{\frac{m(|j| - |k| + \ell - n)}{2}}a_{|j|}a_{|k|} + \sum_{j+k=\ell+n} (-1)^{\frac{m(|j| + |k| - \ell - n)}{2}}a_{|j|}a_{|k|},\\
% D_{n}[\mathbf{C}^{m}_{N}(\mathbf{a})]_{0} &= 4\sum_{j+k=n} (-1)^{\frac{m(|j| - |k| - n)}{2}}a_{|j|}a_{|k|} + 2\sum_{j+k=n} (-1)^{\frac{m(|j| + |k| - n)}{2}}a_{|j|}a_{|k|}\\
D_{0}[\mathbf{C}^{m}_{N}(\mathbf{a})]_{n} &= 2\sum_{j+k=n} (-1)^{\frac{m(|j| - |k| + n)}{2}}a_{|j|}a_{|k|} + \sum_{j+k=n} (-1)^{\frac{m(|j| + |k| - n)}{2}}a_{|j|}a_{|k|}\\
% D_{0}[\mathbf{C}^{m}_{N}(\mathbf{a})]_{0} &= 2\sum_{j+k=0} (-1)^{\frac{m(|j| - |k| )}{2}}a_{|j|}a_{|k|} + \sum_{j+k=0} (-1)^{\frac{m(|j| + |k| )}{2}}a_{|j|}a_{|k|}.\\
\end{split}\end{equation}
Furthermore, we note that
\begin{equation*}\begin{split}
[\mathbf{C}^{m}_{N}(\mathbf{a})]_{0} &= \frac{1}{3}\left[2\sum_{i+j+k=0} (-1)^{\frac{m(|i| + |j| - |k|)}{2}}a_{|i|}a_{|j|}a_{|k|} + \sum_{i+j+k=0} (-1)^{\frac{m(|i| + |j| - |k|)}{2}}a_{|i|}a_{|j|}a_{|k|}\right],\\
&= \frac{1}{3}\left[2\sum_{n=-N}^{N} a_{|n|}\sum_{j+k=|n|}(-1)^{\frac{m(|j| - |k| + |n|)}{2}}a_{|j|}a_{|k|} + \sum_{n=-N}^{N} a_{|n|}\sum_{i+j=|n|} (-1)^{\frac{m(|i| + |j| - |n|)}{2}}a_{|i|}a_{|j|}\right],\\
&= \frac{1}{3}\sum_{i=-N}^{N} a_{|i|}\left\{2\sum_{j+k=|i|}(-1)^{\frac{m(|j| - |k| + |i|)}{2}}a_{|j|}a_{|k|} + \sum_{j+k=|i|} (-1)^{\frac{m(|j| + |k| - |i|)}{2}}a_{|j|}a_{|k|}\right\},\\
&= \frac{1}{3}\sum_{i=-N}^{N} a_{|i|}D_0[\mathbf{C}^{m}_{N}(\mathbf{a})]_{|i|},\\
\end{split}\end{equation*}
and, for $\ell\in[1,N]$,
\begin{equation*}\begin{split}
[\mathbf{C}^{m}_{N}(\mathbf{a})]_{\ell} &= \frac{1}{3}\left[2\sum_{i+j+k=\ell} (-1)^{\frac{m(|i| + |j| - |k|-\ell)}{2}}a_{|i|}a_{|j|}a_{|k|} + \sum_{i+j+k=\ell} (-1)^{\frac{m(|i| + |j| - |k|-\ell)}{2}}a_{|i|}a_{|j|}a_{|k|}\right], \\
&= \frac{2}{6}\sum_{n=-N}^{N}a_{|n|}\left[\sum_{j+k=\ell-|n|} (-1)^{\frac{m(|n| + |j| - |k|-\ell)}{2}}a_{|j|}a_{|k|} + \sum_{j+k=\ell+|n|} (-1)^{\frac{m(|n| + |j| - |k|-\ell)}{2}}a_{|j|}a_{|k|}\right]\\
&\quad + \frac{1}{6}\sum_{n=-N}^{N}a_{|n|}\left[\sum_{i+j=\ell-|n|} (-1)^{\frac{m(|i| + |j| - |n|-\ell)}{2}}a_{|i|}a_{|j|} + \sum_{i+j=\ell+|n|} (-1)^{\frac{m(|i| + |j| - |n|-\ell)}{2}}a_{|i|}a_{|j|}\right], \\
&= \frac{1}{6}\sum_{i=-N}^{N}a_{|i|}\left[2\sum_{j+k=\ell-|i|} (-1)^{\frac{m(|j| - |k| + |i| -\ell)}{2}}a_{|j|}a_{|k|} + \sum_{j+k=\ell-|i|} (-1)^{\frac{m(|j| + |k| - |i|-\ell)}{2}}a_{|j|}a_{|k|}\right]\\
&\quad + \frac{1}{6}\sum_{i=-N}^{N}a_{|i|}\left[2\sum_{j+k=\ell+|i|} (-1)^{\frac{m(|j| - |k| + |i| -\ell)}{2}}a_{|j|}a_{|k|} +  \sum_{j+k=\ell+|i|} (-1)^{\frac{m(|j| + |k| - |i|-\ell)}{2}}a_{|j|}a_{|k|}\right],\\
&=\frac{1}{6}\sum_{i=-N}^{N}a_{|i|}D_{\ell}[\mathbf{C}^{m}_{N}(\mathbf{a})]_{|i|}.
\end{split}\end{equation*}
Thus, we see that the following identities hold
\begin{equation}
    \begin{cases}
        \displaystyle a_{0} D_{0}[\mathbf{C}^{m}_{N}(\mathbf{a})]_{0} + 2\sum_{i=1}^{N}a_{i} D_{0}[\mathbf{C}^{m}_{N}(\mathbf{a})]_{i} = 3 [\mathbf{C}^{m}_{N}(\mathbf{a})]_{0}, & \\
         \displaystyle a_{0} D_{\ell}[\mathbf{C}^{m}_{N}(\mathbf{a})]_{0} + 2\sum_{i=1}^{N}a_{i} D_{\ell}[\mathbf{C}^{m}_{N}(\mathbf{a})]_{i} = 6[\mathbf{C}^{m}_{N}(\mathbf{a})]_{\ell}, &\qquad \forall \ell\in[1,N],\\
    \end{cases}
\end{equation}
which, by applying the symmetries introduced in Lemma~\ref{Lem:DCN;Symm}, one can simplify to
\begin{equation}
     \sum_{i=0}^{N}a_{i} D_{i}[\mathbf{C}^{m}_{N}(\mathbf{a})]_{\ell} = 3[\mathbf{C}^{m}_{N}(\mathbf{a})]_{\ell}, \qquad \forall \ell\in[0,N],
\end{equation}
Assuming that $\mathbf{a}\in\mathbb{R}^{N+1}\backslash\{\mathbf{0}\}$ is a fixed point of $\mathbf{C}^{m}_{N}$, we see that 
\begin{equation}
     \sum_{i=0}^{N}a_{i} D_{i}[\mathbf{C}^{m}_{N}(\mathbf{a})]_{\ell} = 3a_{\ell}, \qquad \forall \ell\in[0,N],
\end{equation}
and so we conclude by writing this in matrix form as
\begin{equation}
     D\mathbf{C}^{m}_{N}(\mathbf{a}) \mathbf{a} = 3\mathbf{a}.
\end{equation}
 \end{proof}
\subsection{Symmetries of the Cubic Matching Equation}\label{app:Match;sym}
We note that the matching equation \eqref{MatchEq} is invariant under the transformations $\mathcal{R}$, $\mathcal{S}$, where
\begin{equation}
    \begin{aligned}
    \mathcal{R} &:a_n\mapsto(-1)^n a_n, &\qquad\qquad\qquad \mathcal{S}&:a_n\mapsto -a_n,
    \end{aligned}
\end{equation}
and so if $\mathbf{a}\in\mathbb{R}^{N+1}$ is a solution of \eqref{MatchEq}, then $\mathcal{R}(\mathbf{a})$ and $\mathcal{S}(\mathbf{a})$ are also solutions of \eqref{MatchEq}. Furthermore, for any $k\in\left[2,N\right]$, 
there is the transformation $\mathcal{H}^{k}[\mathbf{a}]: \mathbb{R}^{\lfloor\frac{N}{k}\rfloor}\to\mathbb{R}^{N}$, where
\begin{equation}\label{HarmonicTransform}
    \mathcal{H}^{k}[\mathbf{a}]_{n} = \begin{cases}
        a_{i}, & n=ik,\\
        0 & \mathrm{otherwise,}
    \end{cases}
\end{equation}
such that
\begin{equation}
    \begin{aligned}
        \mathbf{a} &= \mathbf{C}_{\lfloor \frac{N}{k}\rfloor}^{mk}(\mathbf{a}), &\qquad\qquad \implies \qquad \mathcal{H}^{k}[\mathbf{a}] &= \mathbf{C}_{N}^{m}(\mathcal{H}^{k}[\mathbf{a}]).  
    \end{aligned}
\end{equation}
Thus, for fixed $m,N\in\mathbb{N}$ the matching equation \eqref{MatchEq} has `symmetric solutions' $\mathcal{R}(\mathbf{a})$, $\mathcal{S}(\mathbf{a})$ and $\mathcal{R}[\mathcal{S}(\mathbf{a})]$ for each solution $\mathbf{a}$, and `harmonic solutions' $\mathcal{H}^{k}(\mathbf{a}_{k})$, where $\mathbf{a}_{k}$ are solutions of
\begin{equation}\label{HarmMatchEq}
        \mathbf{a}_{k} = \mathbf{C}_{\lfloor \frac{N}{k}\rfloor}^{mk}(\mathbf{a}_{k}),
\end{equation}
 for each $k\in[2,N]$. We note that $\mathcal{H}^{k}(\mathcal{R}[\mathbf{a}])\neq \mathcal{R}[\mathcal{H}^{k}(\mathbf{a})]$ in general, and so the symmetries of the harmonic equation \eqref{HarmMatchEq} result in additional solutions to the full matching equation \eqref{MatchEq}.

Let us briefly discuss the intuition for these symmetric and harmonic solutions. The symmetric solutions generated by $\mathcal{S}$ are equivalent to the $\mathbb{Z}_2$ symmetry $\mathbf{u}(r,\theta)\mapsto-\mathbf{u}(r,\theta)$ in the Fourier expansion \eqref{FourierExp:R-D}. Such symmetry is the result of rings emerging from a pitchfork bifurcation, even though the PDE system \eqref{eqn:R-D} does not possess $\mathbb{Z}_{2}$ symmetry. We note that  applying $\mathcal{R}$ to the Fourier expansion \eqref{FourierExp:R-D} results in 
\begin{equation}\begin{split}
    \mathcal{R}\mathbf{u}(r,\theta) &= \sum_{n=-N}^{N} (-1)^{n} \mathbf{u}_{|n|}(r)\cos(m n \theta) = \sum_{n=-N}^{N} \mathbf{u}_{|n|}(r)\cos(m n \theta + n\pi) = \mathbf{u}(r, \theta + \tfrac{\pi}{m}),\\
\end{split}\end{equation}
where $\mathbf{u}_{n}(r)$ is of the form defined in \eqref{RadialProfile}, and so $\mathcal{R}$ generates symmetric solutions that have been rotated by a half-period. Finally, the harmonic solutions come from the fact that $\mathbb{D}_{m}$ is a subgroup of $\mathbb{D}_{m k}$, and so
\begin{equation}\begin{split}
        \sum_{n=-N}^{N} \begin{Bmatrix}\mathbf{u}_{|i|}(r), & \mathrm{if} \; n = i k, \\
        0, & \mathrm{otherwise,}\end{Bmatrix}\cos(m n \theta) 
        = \sum_{i=-\lfloor\frac{N}{k}\rfloor}^{\lfloor\frac{N}{k}\rfloor} \mathbf{u}_{|i|}(r) \cos(m k i \theta),
\end{split}\end{equation}

%%%%%%% Appendix: Solutions to the matching equation for small N%%%%%%%%%%%%%%%%%%%%%%%%%%%%%%%%%%%  

%\subsection{Low-Truncation Solutions of the Cubic Matching Equation}\label{app:MatchSol}

%%%%%%%%%%%%%%%%%%%%%%%%%%%%%%%%%%%%%%%%%%%%%%%%%%
%%%%%%%%%%%%%%%%%%%%%%%%%%%%%%%%%%%%%%%%%%%%%%%%%%
%%%%%%%%%%%%%%%%%% Bibliography %%%%%%%%%%%%%%%%%%
%%%%%%%%%%%%%%%%%%%%%%%%%%%%%%%%%%%%%%%%%%%%%%%%%%
%%%%%%%%%%%%%%%%%%%%%%%%%%%%%%%%%%%%%%%%%%%%%%%%%%

\bibliographystyle{abbrv}
\bibliography{Bibliography.bib}
\end{document}